\documentclass{IEEEtran}
\IEEEoverridecommandlockouts
\usepackage{cite}
\usepackage{changes}
\usepackage{pgfplots}
\pgfplotsset{compat=1.15}
\usepgfplotslibrary{units}
\usepackage{amsmath,amsfonts}
\usepackage{multirow}
\usepackage{amsthm}
\usepackage{graphicx}
\usepackage{textcomp}
\usepackage{epsfig}
\usepackage[noend]{algorithmic}
\usepackage{cleveref}[2012/02/15]% v0.18.4; 
\usepackage{epstopdf}
\usepackage[ruled]{algorithm}
\usepackage{enumerate}
\usepackage{float}
\usepackage{graphics,graphicx}
\usepackage{subfigure}
\usepackage{array}
\usepackage[justification=centering]{caption}
\usepackage{footnote}
\usepackage{lipsum}
\makeatletter
\usepackage{color}
\usepackage{xspace}

\usepackage{footnote}

\newcommand{\mb}[1]{\mathbf{#1}}
\newcommand{\mbb}[1]{\mathbb{#1}}
\newcommand{\mc}[1]{\mathcal{#1}}

\def\T{\operatorname{T}}
\def\R{\mathbb{R}}
\def\N{\mathbb{N}}
\def\prox{\operatorname{prox}}

\def\bx{\mathbf{x}}
\def\bz{\mathbf{z}}
\def\by{\mathbf{y}}

\def\BibTeX{{\rm B\kern-.05em{\sc i\kern-.025em b}\kern-.08em
		T\kern-.1667em\lower.7ex\hbox{E}\kern-.125emX}}

\crefformat{footnote}{#2\footnotemark[#1]#3}

\newtheorem{lemma}{Lemma}
\newtheorem{theorem}{Theorem}
\newtheorem{definition}{Definition}
\newtheorem{rem}{Remark}
\newtheorem{assumption}{Assumption}
\newtheorem{proposition}{Proposition}

\begin{document}
	
	\title{\LARGE\bf Asynchronous Distributed Optimization with Delay-free Parameters}
	\author{Xuyang Wu, Changxin Liu, Sindri Magn\'{u}sson, and Mikael Johansson\thanks{X. Wu, C. Liu, and M. Johansson are with the Division of Decision and Control Systems, School of EECS, KTH Royal Institute of Technology, SE-100 44 Stockholm, Sweden. Email: {\tt {\{xuyangw,changxin,mikaelj\}@kth.se}.}}
		\thanks{S. Magn\'{u}sson is with the Department of Computer and System Science, Stockholm University, SE-164 07 Stockholm, Sweden. Email: {\tt sindri.magnusson@dsv.su.se}.} \thanks{This work was supported by WASP and the Swedish Research Council (Vetenskapsr\r{a}det) under grants 2019-05319 and 2020-03607. }}
	\maketitle
	\iffalse
	{\color{black}
		\begin{itemize}
			\item not sure if we can still change the title, but I would prefer "algorithms" to methods, or skipping "methods for" in the title
		\end{itemize}
	}
	\fi
	\begin{abstract}
		Existing asynchronous distributed optimization algorithms often use diminishing step-sizes that cause slow practical convergence, or use fixed step-sizes that depend on and decrease with an upper bound of the delays. Not only are such delay bounds hard to obtain in advance, but they also tend to be large and rarely attained, resulting in unnecessarily slow convergence. This paper develops asynchronous versions of two distributed algorithms, Prox-DGD and DGD-ATC, for solving consensus optimization problems over undirected networks. In contrast to alternatives, our algorithms can converge to the fixed point set of their synchronous counterparts using step-sizes that are independent of the delays. We establish convergence guarantees for convex and strongly convex problems under both partial and total asynchrony. We also show that the convergence speed of the two asynchronous methods adjusts to the actual level of asynchrony rather than being constrained by the worst-case. Numerical experiments demonstrate a strong practical performance of our asynchronous algorithms.
		
	\end{abstract}

	% \begin{IEEEkeywords}
	% asynchronous optimization, distributed optimization, delay-free parameters, DGD.
	% \end{IEEEkeywords}

	\section{Introduction}\label{sec:intro}
	\iffalse
	{\color{black}
		\begin{itemize}
			\item language can be improved in the introduction; I am focusing on the other sections, but please go over carefully.
		\end{itemize}
	}
	\fi
	% background
	
	% In many scenarios, the problem data of an optimization problem is distributed over a network of nodes, and all nodes aim to solve the problem collaboratively via local communications only (i.e., only one-hop neighbors can communicate with each other). This problem has found many applications in distributed learning, power systems, robot, etc.
	
	% literature survey: broad perspective
	
	%Distributed optimization has attracted much attention in the last decade and has found applications in diverse areas such as cooperative control, machine learning, and power systems. The literature on distributed optimization has primarily focused on synchronous methods. These methods iterate in a serialized manner, proceeding to the next iteration only after the current one is completed. Synchronous methods also require all nodes to maintain a consistent view of optimization variables without any information delay, which makes the algorithms easier to analyse. However, achieving synchronization through a network can be challenging and inefficient, as the time required per iteration is often determined by the slowest node and the process can suffer from single-node failure. As a result, using synchronous methods can lead to inefficient use of computation and communication resources.
	
	Distributed optimization has attracted much attention in the last decade and has found applications in diverse areas such as cooperative control, machine learning, and power systems. 
	The literature on distributed optimization has primarily focused on synchronous methods that iterate in a serialized manner, proceeding to the next iteration only after the current one is completed. Synchronous methods require all nodes to maintain a consistent view of the optimization variables without any information delay. This makes the algorithms easier to analyze, but more difficult to realize. Global synchronization through a network can be challenging in practice. Additionally, synchronized updates are inefficient and unreliable since the time taken per iteration is determined by the slowest node and the optimization process is vulnerable to single-node failures. %Consequently, employing synchronous methods can lead to inefficient use of computation and communication resources.

	Asynchronous distributed methods that do not require synchronization between nodes are often better suited for practical implementation \cite{assran2020advances,huba2022papaya}. However, asynchronous methods must account for information delays and allow for nodes to work with inconsistent views of the optimization variables, which makes them difficult to analyze. Despite the inherent challenges, there have been notable successes in studying the mathematical properties of asynchronous optimization algorithms. One focus area has been the asynchronous consensus optimization algorithms \cite{nedic2010convergence,mishchenko2018delay,zhang2019asyspa,doan2017convergence,kungurtsev2023decentralized,assran2020asynchronous,spiridonoff2020robust,zhang2019fully,wu2017decentralized,tian2020achieving,bianchi2021distributed}. These include asynchronous variants of well-established algorithms such as decentralized gradient descent (DGD), PG-EXTRA, and gradient-tracking-based methods. Asynchronous distributed algorithms developed for other settings encompass Asy-FLEXA \cite{cannelli2020asynchronous}, the asynchronous primal-dual algorithm \cite{latafat2022primal}, the asynchronous dual decomposition \cite{su2022convergence}, the asynchronous coordinate descent method \cite{ubl2021totally, Wu23ICML}, and ADGD \cite{wang2021asynchronous}. %{\color{teal} Highlight that our method and analysis allows for communication failures while those who use augmented graph to transform the asynchronous methods into synchronous methods do not allow for communication failures. Then, two things to do. First, highlight this in our paper, Second, understand that algorithms that can converge over time-varying and directed networks may also be implemented asynchronously. What about the Fenchel dual gradient method??}

	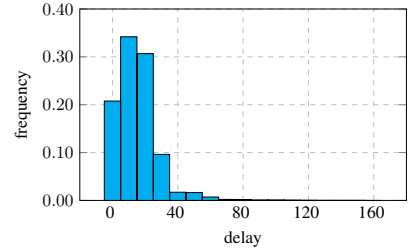
\begin{figure}
		\centering
		\begin{tikzpicture}[scale=0.66]
		\begin{axis}[
		ybar,
		xlabel={delay},
		ylabel={frequency},
		width=0.45\textwidth,
		height=0.3\textwidth,
		bar width=10,
		xmin = -20,
		xmax = 180,
		ymin=0,
		ymax=0.4,
		grid=major,
		grid style=dashed,
		xtick distance=40,
		xticklabel style={rotate=0, anchor=north},
		xticklabels={0, 0, 40, 80, 120, 160},
		yticklabel style={
			/pgf/number format/.cd,
			fixed,
			fixed zerofill,
			precision=2,  % Set the desired precision here
		},
		xtick align=inside,  % Control the direction of x-axis tick marks
		legend style={
			at={(0.5,-0.2)},
			anchor=north,
			legend columns=-1,
		},
		xticklabel style={
			align=left,  % Center the tick labels
			yshift= 0ex,   % Adjust the vertical position
			xshift=-0.3ex,
		},
		xtick style={draw=none},
		]
		% Frequency data
		\addplot[fill=cyan] coordinates {
			(0, 0.2076)
			(10, 0.3419)
			(20, 0.3068)
			(30, 0.0962)
			(40, 0.0171)
			(50, 0.0164)
			(60, 0.0071)
			(70, 0.0021)
			(80, 0.0018)
			(90, 0.0009)
			(100, 0.0009)
			(110, 0.0006)
			(120, 0.0003)
			(130, 0.0001)
			(140, 0.0001)
			(150, 0.0001)
			(160, 0)
        };
		\end{axis}
		\end{tikzpicture}
		\caption{distribution of real delays (30 nodes).}
		\label{fig:histogram}
	\end{figure}

    Most of the studies cited above use diminishing step-sizes~\cite{zhang2019asyspa,doan2017convergence,kungurtsev2023decentralized,assran2020asynchronous,spiridonoff2020robust} or fixed step-sizes that depend on information delays in the  system~\cite{wu2017decentralized, tian2020achieving, cannelli2020asynchronous,su2022convergence, latafat2022primal, zhang2019fully,bianchi2021distributed}. Diminishing step-sizes are effective in stochastic optimization, but can result in unnecessarily slow convergence on deterministic problems. At the same time, asynchronous algorithms that use fixed step-sizes often determine step-sizes based on an upper bound $\tau$ on the maximal delay and use step-size selection policies that decrease with $\tau$, either as $O(\frac{1}{\tau+1})$ \cite{wu2017decentralized} or as $O(\frac{1}{(\tau+1)^2})$ \cite{latafat2022primal}. Since the worst-case delay is hard to estimate before actually running the algorithm, it is difficult to use these methods in a theoretically justified way. Moreover, even when the maximal information delay is known, these methods tend to converge slowly since worst-case delays are often (very) large, while the majority of delays are small.
\iffalse
 it can also result in overly conservative step-sizes and unnecessarily slow convergence as, in real systems, the worst-case delay can exhibit significant magnitude, whereas the majority of delays tend to be notably smaller. 
 \fi
To demonstrate this dilemma, we have conducted real experiments depicted in Fig. \ref{fig:histogram}\footnote{Asynchronous implementation of DGD \cite{yuan2016convergence} on a multi-thread computer using MPI4py \cite{dalcin2008mpi} for distributed training of a classifier. Each local loss function is a logistic function defined by $2000$ samples in MNIST \cite{lecun1998gradient}, all threads are homogeneous and each serves as a node, and the communication network is a line graph. We use a global iteration number to index the iterates, which is increased by one whenever a node updates its local iterate. The delays are defined as the difference between the current iteration index and the indexes of the iterates used in the update. All delays are generated by real interactions between threads rather than any theoretical models.}, %Fig~\ref{fig:my_plot} records the worst-case delays over $1000$ iterations along with step-sizes that depend on worst case delays for various number of nodes. }
which plots the delay distribution over $1000$ iterations in a 30-node network. In this experiment, the maximum delay was $160$, while over $95\%$ delays were smaller than $40$.
% These conservative step-sizes are used to guard against the worst-case delay, which may be unnecessary because the worst-case delay rarely happens from Fig. \ref{fig:histogram}: 
 % such conservative step-sizes may not be  
%A similar phenomenon is observed by our earlier work \cite{Wu23ICML}, wherein we implement an asynchronous block coordinate gradient descent method on a $8$-core machine and report a maximum and average delay of around $28$ and $8$, respectively.
Similar phenomena were observed in \cite{mishchenko2018delay,Wu23ICML}. For example, \cite{mishchenko2018delay} implemented asynchronous SGD on a 40-core CPU and reported maximum and average delays of $1200$ and $40$, respectively. {\color{black}A few studies \cite{Sun17,wu2017decentralized} consider stochastic delays and let the step-sizes rely on parameters of the delay distribution rather than the worst-case delay. However, they both assume that the system delays are independent across iterations,  which is less practical}. % {\color{teal} Here I do not replace the work of Mishchenko with our ICML work because the delays in Mishchenko's work are more extreme.}%This suggests that there is a huge potential in considering new step-size rules that do not depend on the worst-case delay.

 The papers  \cite{nedic2010convergence,ubl2021totally,wang2021asynchronous,mishchenko2018delay,Wu23ICML} show that some asynchronous distributed algorithms can converge with fixed step-sizes that can be determined without any information about the delays. Compared with the delay-dependent ones in \cite{wu2017decentralized, tian2020achieving, cannelli2020asynchronous,su2022convergence,latafat2022primal, zhang2019fully,bianchi2021distributed}, such delay-free step-sizes can be safely determined prior to deployment. Moreover, since they do not guard against the worst-case delay, their convergence tends to be governed by the actual experienced delays and not be hampered by large and rare worst-case delays. For example, consider an extreme scenario where the algorithm experiences a large delay $\tau$ in a single iteration, while all remaining delays are $0$. In such a case, delay-dependent step-sizes will be reduced to guard against the worst case, and the algorithm will converge slower also in iterations where there are no delays. Under delay-free step-sizes, on the other hand, the algorithm will only have slow progress in iterations with large delays. However, existing research on delay-free step-sizes  only addresses quadratic problems \cite{nedic2010convergence,ubl2021totally} or star networks \cite{wang2021asynchronous,mishchenko2018delay,Wu23ICML}.

	This paper studies asynchronous variants of two popular distributed algorithms for consensus-optimization over undirected networks, proximal decentralized gradient descent (Prox-DGD) \cite{zeng2018nonconvex} and DGD using the adapt-then-combine technique (DGD-ATC) \cite{pu2020asymptotic}. Our contributions include:
	\begin{enumerate}
		\item We bound the optimality gap between the fixed-points of the two synchronous methods under fixed step-sizes and the optimum of the consensus optimization problem. Our results on Prox-DGD are more comprehensive than the current literature, and those on DGD-ATC were absent.
		\item We show theoretically that, i) with strongly convex objective functions under total asynchrony (Assumption \ref{asm:totalasynchrony}), and ii) with convex objective functions and under partial asynchrony (Assumption \ref{asm:partialasynchrony}), the proposed methods can use \emph{fixed step-sizes that are set without any information about the delays}, and still converge to fixed-points of their synchronous counterparts. 
		\item By assuming both partial asynchrony and strongly convex objective functions, we improve the convergence guarantee from asymptotic to linear and show that the convergence {\color{black}{adjusts}} to the actual level of asynchrony ({\color{black}the update frequency of each node and the delays}), {\color{black} and are not penalized by large worst-case delays that occur rarely.}
          \item Our algorithms can handle non-quadratic objective functions and general undirected networks, while existing methods with fixed delay-free step-sizes~\cite{nedic2010convergence,ubl2021totally,mishchenko2018delay} require quadratic objective functions or star networks.
          \item Our algorithms are guaranteed to converge under total asynchrony while those using delay-dependent fixed step-sizes \cite{wu2017decentralized, zhang2019fully,tian2020achieving,bianchi2021distributed} or diminishing step-sizes \cite{zhang2019asyspa,assran2020asynchronous,doan2017convergence,kungurtsev2023decentralized,spiridonoff2020robust} are not.

	\end{enumerate}
Many algorithms that use delay-dependent fixed step-sizes \cite{wu2017decentralized, zhang2019fully,tian2020achieving,bianchi2021distributed} or diminishing step-sizes \cite{zhang2019asyspa,assran2020asynchronous,doan2017convergence,kungurtsev2023decentralized} can converge to the optimal consensus solution. Our algorithms, on the other hand, suffer from an unfavourable inexact convergence inherited from their synchronous counterparts. However, they admit delay-free step-sizes that are easy to tune and may yield faster practical convergence as shown by our numerical experiments.
%Surprisingly, our algorithms can even be faster than their synchronous counterparts in our numerical experiments.

          	% Admittedly, our asynchronous algorithms suffer from unfavourable inexact convergence inherited from their synchronous counterparts. 
	% Notation

 A preliminary conference version can be found in \cite{wu23CDC}. It studies the asynchronous DGD and DGD-ATC for solving strongly convex and smooth problems, does not show any convergence results in terms of the actual asynchrony level, and includes no proofs. Compared to \cite{wu23CDC}, this article extends the asynchronous DGD to the asynchronous Prox-DGD which allows for composite objective functions, derives convergence results of the two asynchronous methods for both strongly convex and general convex functions, proves that the convergence rates adjusts to the actual level of asynchrony rather than being determined by the worst-case, proposes an adaptive step-size for the asynchronous Prox-DGD, and includes all the proofs.

	\subsection*{Paper Organization and Notation}
	
	The outline of this paper is as follows: Section II formulates the consensus optimization problem, revisits the synchronous algorithm Prox-DGD for solving the problem with composite objective functions, and establishes its optimality error bound. Section III introduces the asynchronous Prox-DGD and analyses its convergence. Section IV studies the DGD-ATC variant of DGD for smooth problems, introduces its asynchronous counterpart, and proves similar convergence results as for the asynchronous Prox-DGD. Finally, Section V evaluates the practical performance of the asynchronous algorithms using numerical experiments and Section VI concludes the paper.
	
	We use $\mb{1}_d$, $\mathbf{0}_{d\times d}$, and $I_d$ to denote the $d$-dimensional all-one vector, the $d\times d$ all-zero matrix, and the $d\times d$ identity matrix, respectively. Subscripts are omitted when it is clear from context. The notation $\otimes$ represents the Kronecker product and $\N_0$ is the set of natural numbers including $0$. For any symmetric matrix $W\in\mbb{R}^{n\times n}$, $\lambda_i(W)$, $1\le i\le n$ denotes the $i$th largest eigenvalue of $W$, $\operatorname{Null}(W)$ is its null space, and $W\succ \mathbf{0}$ and $W\succeq \mb{0}$ mean that $W$ is positive definite and positive semi-definite, respectively. For any vector $x\in\mathbb{R}^n$, we use $\|x\|$ to represent the $\ell_2$ norm and define $\|x\|_W=\sqrt{x^TWx}$ for any positive definite matrix $W\in\mathbb{R}^{n\times n}$. For any function $f:\mathbb{R}^d\rightarrow\mathbb{R}\cup\{+\infty\}$, we define
 \begin{equation*}
     \operatorname{prox}_f(x) = \operatorname{\arg\;\min}_y f(y)+\frac{1}{2}\|y-x\|^2.
 \end{equation*}
 For any differentiable function $f:\mathbb{R}^d\rightarrow\mathbb{R}$, we say that {\color{black}it is $M$-Lipschitz continuous for some $M>0$ if 
 \[\|f(y)-f(x)\|\le M\|y-x\|,~\forall x,y\in\mathbb{R}^d,\]
 }
 it is $L$-smooth for some $L>0$ if 
	\begin{equation*}
	\|\nabla f(y)-\nabla f(x)\|\le L\|y-x\|,~\forall x,y\in\mathbb{R}^d,
	\end{equation*}
	it is $\mu$-strongly convex for some $\mu>0$ if \begin{equation*}
	\langle \nabla f(y)-\nabla f(x), y-x\rangle\ge \mu\|y-x\|^2,~\forall x,y\in\mathbb{R}^d,
	\end{equation*}
{\color{black}     and it is coercive if $\lim_{\|x\|\rightarrow +\infty} f(x) = +\infty$. When $f$ is twice continuously differentiable, the $M$-Lipschitz continuity, $L$-smoothness, and the $\mu$-strong convexity properties are equivalent to that, for any $z$, $\|\nabla f(z)\|\le M$, $\nabla^2 f(z)\preceq L I$, and $\nabla^2 f(z)\succeq \mu I$, respectively.}
	
\section{Problem Formulation and Prox-DGD}

This section formulates the consensus optimization problem, revisits the synchronous proximal distributed gradient descent (Prox-DGD) \cite{zeng2018nonconvex} for solving it, and derives bounds on the optimality error of the fixed points of Prox-DGD.
%with respect to the consensus optimization problem. % The asynchronous Prox-DGD will be introduced in Section \ref{sec:asycalg}.

\subsection{Consensus Optimization}

{\color{black}We consider consensus optimization over an undirected network $\mc{G}=(\mc{V}, \mc{E})$, where $\mc{V}=\{1,\ldots,n\}$ is the vertex set and $\mc{E}\subseteq \mc{V}\times \mc{V}$ is the edge set that excludes self-loops. We impose the following assumption on the network $\mc{G}$.
\begin{assumption}
    The network $\mc{G}$ is undirected and connected, and each node has a self-loop.
\end{assumption}}
%}{Consider a network of $n$ nodes described by an undirected, connected graph $\mc{G}=(\mc{V}, \mc{E})$, where $\mc{V}=\{1,\ldots,n\}$ is the vertex set and $\mc{E}\subseteq \mc{V}\times \mc{V}$ is the edge set.}
In the network, each node $i$ observes local cost functions $f_i:\mathbb{R}^d\rightarrow \mathbb{R}$ and $h_i:\mathbb{R}^d\rightarrow \mathbb{R}\cup\{+\infty\}$, and can only interact with its neighbors in $\mc{N}_i=\{j: \{i,j\}\in\mc{E}\}$. Consensus optimization aims to find a common decision that minimizes
\begin{equation}\label{eq:consensusprob}
    \begin{split}
        \underset{x\in\mathbb{R}^d}{\operatorname{minimize}}~&\sum_{i\in\mc{V}} f_i(x)+h_i(x),
    \end{split}
\end{equation}
where popular choices of $h_i$ include $\ell_1$-norm and indicator function of a convex set. Problem \eqref{eq:consensusprob} describes many engineering problems in control and machine learning for networked systems, including distributed model predictive control \cite{giselsson2013accelerated}, 
empirical risk minimization \cite{Vapnik}, etc.

Distributed algorithms for solving problem \eqref{eq:consensusprob} include the distributed subgradient method \cite{nedic2009distributed}, Prox-DGD \cite{zeng2018nonconvex}, distributed gradient-tracking-based algorithm \cite{nedic17}, distributed dual averaging\cite{liu2022decentralized}, and PG-EXTRA \cite{shi2015proximal}. While these algorithms were originally designed to be executed synchronously, they have been extended to allow for asynchronous implementations. However, existing asynchronous methods for solving \eqref{eq:consensusprob} often suffer from slow convergence due to the use of either diminishing step-sizes, or small fixed step-sizes that depend on a (usually unknown and large) upper bound on all delays.

	\subsection{Synchronous Prox-DGD}
	
	%	The first algorithm is DGD~\cite{yuan2016convergence}.
    To present the algorithm in a compact form, let $x_i\in\mbb{R}^d$ be a local copy of $x$ in \eqref{eq:consensusprob} held by node $i\in\mc{V}$, and define
	\begin{align*}
	&\mathbf{x}=(x_1^T, \ldots, x_n^T)^T,~~f(\bx) = \sum_{i\in\mc{V}} f_i(x_i),\\
	&h(\bx) = \sum_{i\in\mc{V}} h_i(x_i),~~F(\bx) =f(\bx)+h(\bx).
	\end{align*}
	{\color{black} Let
	$\mb{W}=W\otimes I_d$ where
	$W\in\mathbb{R}^{n\times n}$ is an averaging matrix associated with $\mc{G}$:
    \begin{definition}[averaging matrix]\label{def:def_av}
        A matrix $W\in\mathbb{R}^{n\times n}$ is called an averaging matrix associated with $\mc{G}$ if
        \begin{enumerate}[1)]
            \item $W$ is symmetric and doubly stochastic, i.e., it is non-negative, and each row and each column sums to $1$;
            \item For any $i,j\in\mc{V}$, $w_{ij}>0$ if $\{i,j\}\in\mc{E}$ or $i=j$, and $w_{ij}=0$ otherwise.
        \end{enumerate}
    \end{definition}
    This matrix can be easily formed in a distributed manner, with many options available, as outlined in \cite[Section 2.4]{shi2015extra}.} We use $k\in\mathbb{N}_0$ as iteration index and $\mathbf{x}^k$ as the value of $\bx$ at iteration $k$. Then, the Prox-DGD algorithm progresses according to the following iteration: 
	\begin{equation}\label{eq:DGD}
	\mathbf{x}^{k+1}= \prox_{\alpha h}(\mb{W}\bx^k - \alpha \nabla f(\bx^k)),
	\end{equation}
	where the parameter $\alpha>0$ is the step-size.% for a given initialization $\mathbf{x}^0$.
	
	Prox-DGD can be seen as the proximal gradient descent method for solving the following penalized problem of \eqref{eq:consensusprob}:
	\begin{equation}\label{eq:penalprob}
	\underset{\bx\in\mathbb{R}^{nd}}{\operatorname{minimize}}~f(\bx)+h(\bx)+\frac{1}{2\alpha}\bx^T(I-\mathbf{W})\bx.
	\end{equation}
	Every fixed point of Prox-DGD is an optimum to \eqref{eq:penalprob} and vice versa. As a result, when studying  Prox-DGD, it is natural to assume that the optimal solution set of \eqref{eq:penalprob} is non-empty.
	\begin{assumption}\label{asm:optsolexist}
		For any $\alpha>0$, the minimizer of $F(\bx)+\frac{1}{2\alpha}\bx^T(I-\mathbf{W})\bx$ exists.
	\end{assumption}
% {\color{black} While I agree with what you have written, I do not find this result in the cited reference. This might be a problem of versions (the KTH library has the second edition). I would prefer if we could cite a clean result that guarantees existence of optima. I also do not think that we need the sentence that says that guaranteeing the existence of optima is difficult, but we should motivate why this paper provides a proof for existence of optima for the least-squares problem (why do we need it?) }
	When $F$ is proper, lower semicontinuous, and convex, Assumption \ref{asm:optsolexist} holds if $F$ is coercive \cite[Proposition 11.15]{bauschke2011convex}, i.e., $\lim_{\|\bx\|\rightarrow+\infty} F(\bx)=+\infty$, which holds when $F$ is strongly convex or has a bounded domain. However, even if $F$ is not coercive, the minimizer may still exist for many practical problems. For example, if each $f_i(x_i)=\|A_ix_i-b_i\|^2$ and $h_i\equiv 0$, then even if $A_i$ does not have full row rank so that $F$ is not coercive, the minimizer of $F(\bx)+\frac{1}{2\alpha}\bx^T(I-\mathbf{W})\bx$ still exists (proof in Appendix \ref{append:qpoptsolexist}).

\subsection{Optimality of the Fixed Points of Prox-DGD}
The fixed points of Prox-DGD, or equivalently, the optimal solutions of problem \eqref{eq:penalprob}, are generally not optimal to \eqref{eq:consensusprob}. This subsection bounds the error of the fixed points of Prox-DGD relative to the optimizers of \eqref{eq:consensusprob} under proper assumptions.

%{\color{black} Here, the notation is not super-clear, so it would be good to help the reader a little. I mean, $F^{\star}$ is the optimal value of (1), but $F(x^{\star})$ is not the optimal value of (3), only the value of $\sum f_i(x_i) + h_i(x_i)$ at the fixed-point.  Indeed, since we do not enforce the consensus constraints in (3), $x_i^{\star}$ may in general be different between nodes, and $F(x^{\star})$ may be smaller than the optimal value that can be attained by a consensus solution.

% I also think that we can move in the definition of $\beta$ into Lemma 1.

%We first impose the assumption below.
\begin{assumption}\label{asm:prob}
    Each $f_i$ and $h_i$ are proper, closed, and convex, and each $f_i$ is $L_i$-smooth for some $L_i>0$. Moreover, each $f_i+h_i$ is bounded from the below and problem \eqref{eq:consensusprob} has a non-empty optimal solution set. 
\end{assumption}
{\color{black}The smoothness condition in Assumption \ref{asm:prob} is standard for convergence analysis of optimization methods \cite{mishchenko2018delay,tian2020achieving,wu2017decentralized,Li23}.} Under Assumption \ref{asm:prob}, we introduce our first optimality gap result. To this end, let $F_{\operatorname{opt}}$ be the optimal value of \eqref{eq:consensusprob}.
\begin{lemma}\label{lemma:DGDoptimalitygap}
    Suppose that Assumptions \ref{asm:optsolexist}--\ref{asm:prob} hold. For any fixed point $\bx^\star$ of Prox-DGD, it holds that $F(\bx^\star)\le F_{\operatorname{opt}}$ and
    \begin{equation}\label{eq:errorDGD}
		\|x_i^\star - \bar{x}^\star\|\le O\left(\sqrt{\frac{\alpha}{1-\beta}}\right),\quad\forall i\in\mc{V},
    \end{equation}
    where $\bar{x}^\star=\frac{1}{n}\sum_{i\in\mc{V}} x_i^\star$ and $\beta=\max\{|\lambda_2(W)|, |\lambda_n(W)|\}\in [0,1)$.
\end{lemma}
\begin{proof}
    See Appendix \ref{append:gapDGD}.
\end{proof}
The above error bound can be improved under additional assumptions. We first consider the general case where the nonsmooth functions may be different.
\begin{assumption}[non-identical $h_i$'s]\label{asm:strongercond}
    One of the following statements is true:
    \begin{enumerate}[i)]
	\item each $f_i+h_i$ is Lipschitz continuous;
	\item each $\operatorname{dom} h_i=\mathbb{R}^d$ and $f_i+h_i$ is coercive.
    \end{enumerate}
\end{assumption}
If all the non-smooth functions are identical, then even if Assumption \ref{asm:strongercond} fails to hold, we are still able to improve the bound in Lemma \ref{lemma:DGDoptimalitygap} under the following assumption.

\begin{assumption}[identical $h_i$'s]\label{asm:identicalstrongcond}
    All the $h_i$'s are identical, and one of the following holds:
    \begin{enumerate}[i)]
	\item each $f_i$ and $h_i$ are bounded from the below;
	\item each $f_i$ is Lipschitz continuous;
	\item each $f_i+h_i$ is coercive\footnote{}.
    \end{enumerate}
\end{assumption}
{\color{black}These three conditions are standard in convergence analysis and hold for many problems. For example, condition i) holds in typical machine learning setting where each $f_i$ is a non-negative loss function such as quadratic loss or logistic loss and each $h_i$ is a non-negative regularizer; Condition ii) is common for convergence analysis of subgradient-type methods \cite{nedic2009distributed}; Condition iii) holds for many choices of $h_i$ such as 1) $h_i(x)=\mc{I}_{X}(x)$ where $X$ is a convex and compact {\color{black}(closed and bounded)} set; 2) $h_i(x)=c\|x\|_1$ where $c>0$.}%; 3) $h_i(x)=c\|x\|^2$ where $c>0$;}
	\begin{lemma}\label{lemma:stronggap}
		Suppose that Assumptions \ref{asm:optsolexist}--\ref{asm:prob} hold. For any fixed point $\bx^\star$ of Prox-DGD,
		\begin{enumerate}
			\item Under Assumption \ref{asm:strongercond},
			\begin{equation}\label{eq:betaDGD}
			\|x_i^\star - \bar{x}^\star\|\le O\left(\frac{\alpha}{1-\beta}\right);
			\end{equation}
			\item Under Assumption \ref{asm:identicalstrongcond}, \eqref{eq:betaDGD} holds and
			\begin{equation}\label{eq:funcave}
			\begin{split}
			F(\bar{\bx}^\star) &\le F_{\operatorname{opt}}+O\left(\frac{\alpha}{1-\beta}+\frac{L\alpha^2}{2(1-\beta)^2}\right),
			\end{split}
			\end{equation}
			where $\bar{\bx}^\star=\mb{1}_n\otimes \bar{x}^\star$ and $L=\max_{i\in\mc{V}} L_i$.
		\end{enumerate}
	\end{lemma}
	\begin{proof}
		See Appendix \ref{append:stronggap}.
	\end{proof}
 \iffalse
 {\color{black}
 I believe that you could organize this paragraph better. There are so many different links between all the results that one easily gets the impression that the results may be new, but not that much. And also the phrase novel is a little bit too strong.

 I would prefer to say that "Lemma 2 extends and generalizes several results from the literature. For example, [20] and [23] considered the special case that all $f_i$'s are smooth and no regularizers are present ($h_i\equiv 0$), and reference [34] studied a set-up where all $f_i$'s are smooth and the $h_i$'s are all equal and encode a projection onto the same set $X$. Our results, on the other hand, allow for general and non-identical $h_i$'s.
 
 Then you should add something about (5). But you should then also explain \emph{why} it is important. 
 "
 }
\fi

Lemma \ref{lemma:stronggap} extends and generalizes several results from the literature. For example, \cite{yuan2016convergence,zeng2018nonconvex} considered the special case that all $f_i$'s are smooth and no regularizers are present ($h_i\equiv 0$), and reference \cite{choi2023convergence} studied a set-up where all $f_i$'s are smooth and the $h_i$'s are all equal and encode a projection onto the same set $X$. Our results, on the other hand, allow for general and non-identical $h_i$'s. We also provide an optimality gap \eqref{eq:errorDGD} in Lemma \ref{lemma:DGDoptimalitygap} that can be applied to more general problems without requiring Assumptions \ref{asm:strongercond} or \ref{asm:identicalstrongcond}. This allows us to study a broader range of problems, such as problem \eqref{eq:consensusprob} where the $f_i$ functions are smooth, and the $h_i$ functions encode projections onto distinct, convex sets. An optimality gap of the same order as \eqref{eq:errorDGD} is provided in \cite{zhang2021penalty}, which allows distinct $h_i$'s but requires strong convexity of $f_i$'s.

	\section{Asynchronous Prox-DGD}\label{sec:asycalg}

	In this section, we introduce the asynchronous Prox-DGD and analyze its convergence. The algorithm is fully asynchronous in the sense that it requires neither global synchronization between nodes nor access to a global clock. %It is analyzed in a setting where 
    
	\iffalse
	For the two asynchronous algorithms, we consider the following setting: 1) in the entire optimization process, each node $i\in\mc{V}$ is activated at discrete time points, and can update and share its local variables once it is activated. 2) each node $i\in\mc{V}$ has a buffer $\mc{B}_i$, which can receive and store the message from its neighbors all the time.
	\fi
	%\textcolor{teal}{\textbf{To me, activated can not be a state, it is a transition between the states inactive and active. I see two potential approaches to make this work nicely: a) introduce the states active and inactive or b) don't introduce the states, just say that Each node is activated at discrete time points and is otherwise inactive (or a sleep. Option b) is likely better because we don't really use the states in the rest of the text, just the transitions of becoming active.}} {\color{black} I see. Then I use active/inactive to represent the status, but use the expression "once node $i$ is activated" to refer to the action/moment of activating}
	
%	\subsection{Asynchronous Algorithm}
	\subsection{Algorithm Description}

%{\color{black}

%I believe that it is confusing to talking about setting $x_{ij}=x_j$. Because $x_j$ is information held at node $j$, which is conceptually different and could (as you point out), be different than the version in the local buffer (say, if we include the possibility of message delays or losses).

%I also think that it is a little awkward to introduce $\bar{\mathcal N}_j$ in (9). It would have been easier to introduce it earlier (perhaps around (8)).}
%\bigskip

 In the asynchronous Prox-DGD algorithm, each node is activated at discrete points in time. {\color{black}When node $i\in {\mathcal V}$ is activated, it first updates its local iterate $x_i \in {\mathbb R}^d$ and shares it with its neighbors $j\in {\mathcal N}_i$, and then enters a sleep mode.} %{When node $i\in {\mathcal V}$ is activated, it updates its local iterate $x_i \in {\mathbb R}^d$ and shares it with its neighbors $j\in {\mathcal N}_i$ before going back to sleep.} 
 A node does not need to be active to receive information from its neighbors. Instead, each node $i$ uses a buffer ${\mathcal B}_i$ to store the information it receives from its neighbors between activation times.  When node $i$ is activated, it reads all $x_{ij}$ (the most recent $x_j$ it has received from neighboring node $j$) from the buffer, performs the local update
 
\iffalse 
 In the asynchronous Prox-DGD, we let each node $i\in\mc{V}$ hold $x_i\in\mathbb{R}^d$ and $x_{ij}\in\mathbb{R}^d$ $\forall j\in\mc{N}_i$, where $x_i$ is the current local iterate of node $i$ and $x_{ij}$ records the most recent $x_j$ it received from node $j\in\mc{N}_i$. Each node is activated at discrete time points, and can update and share its local variables once it is activated. In addition, every node $i\in\mc{V}$ has a buffer $\mc{B}_i$ in which it can receive and store messages from neighbors all the time (even inactive). Once activated, node $i$ reads all $x_j$ in the buffer $\mc{B}_i$ and then sets $x_{ij}=x_j$ and, in case $\mc{B}_i$ contains multiple $x_j$'s for a particular $j\in\mc{N}_i$, node $i$ sets  $x_{ij}$ as the most recently received $x_j$. Next, it updates $x_i$ by
 \fi
	\begin{equation}\label{eq:DGDxiupdate}
	x_i \leftarrow \prox_{\alpha h_i}\biggl(w_{ii}x_i+\sum_{j\in\mc{N}_i} w_{ij}x_{ij} - \alpha \nabla f_i(x_i)\biggr),
	\end{equation}
	and broadcasts the new $x_i$ to all its neighbors. Once a node $j\in\mc{N}_i$ receives this new $x_i$, it stores $x_i$ in its buffer $\mc{B}_j$. A detailed implementation is given in Algorithm \ref{alg:DGD}.

	To describe the asynchronous Prox-DGD mathematically, we index the iterates by $k\in\N_0$. The index is increased by $1$ whenever updates are performed by some {\color{black}(one or more)} nodes, and does not need to be known by the nodes -- $k$ is only introduced to order events in our theoretical analysis. With this indexing, each $x_{ij}$ in \eqref{eq:DGDxiupdate} is simply a delayed version of $x_j$ and the execution of (\ref{eq:DGDxiupdate}) can be written as
\iffalse
-- each node $i$ updates using the most recently received $x_j$ for higher efficiency but it is, in general, not the newest $x_j$ computed by node $j$. 

 Let $\mc{K}_i\subseteq \N_0$ denote the set of iterations in which node $i$ updates its iterate. For convenient notation, we define $\bar{\mc{N}}_i=\mc{N}_i\cup\{i\}$ for all $i\in\mc{V}$. Then, the asynchronous Prox-DGD can be described as follows. For each $i\in\mc{V}$ and $k\in\N_0$,
\fi
	\begin{equation}\label{eq:asyDGDupdateindex}
	x_i^{k+1} \!=\! \begin{cases}
	\prox_{\alpha h_i}\bigl(\sum_{j\in\bar{\mc{N}_i}} w_{ij}x_j^{s_{ij}^k} \!- \alpha \nabla f_i(x_i^k)\bigr), & k\in \mc{K}_i,\\
	x_i^k, & \text{otherwise}.
	\end{cases}
	\end{equation}
where $\mc{K}_i\subseteq \N_0$ is the set of iterations in which node $i$ is active, $\bar{\mc{N}}_i=\mc{N}_i\cup\{i\}$, {\color{black}$s_{ij}^k=\max\{t\le k: x_j^t = x_j\}\in [0, k]$ is the maximum index of the most recent version of $x_j$ available to node $i$ at iteration $k$}, and $s_{ii}^k=k$. {\color{black} The sets $\mc{K}_i, i\in\mc{V}$ are not necessarily disjoint. For example, in the synchronous setting where $k$ is increased by $1$ only if all nodes finish one update, the algorithm can be described by \eqref{eq:asyDGDupdateindex} with $\mc{K}_i=\N_0$ $\forall i\in\mc{V}$.}

%
%If $\mc{K}_i=\N_0$ $\forall i\in\mc{V}$ and $s_{ij}^k=k$ $\forall \{i,j\}\in\mc{E}, \forall k\in\N_0$, then \eqref{eq:asyDGDupdateindex} reduces to the synchronous Prox-DGD iteration \eqref{eq:DGD}.
 \begin{algorithm}[t!]
    \makeatletter
    \renewcommand\footnoterule{%
        \kern-3\p@
    \hrule\@width.4\columnwidth
    \kern2.6\p@}
    \makeatother
    \caption{Asynchronous Prox-DGD}
    \begin{minipage}{\linewidth}
    \renewcommand{\thempfootnote}{\arabic{mpfootnote}}
    
    \label{alg:DGD}
		\begin{algorithmic}[1]
			\STATE {\bfseries Initialization:} All the nodes agree on $\alpha>0$, and cooperatively set $w_{ij}$ $\forall \{i,j\}\in\mc{E}$.
			\STATE Each node $i\in\mc{V}$ chooses $x_i\in\mathbb{R}^d$, creates a local buffer $\mc{B}_i$, and shares $x_i$ with all neighbors in $\mc{N}_i$.
			% 			\STATE Each node $i\in\mc{V}$ receives $x_j$ and saves $(x_j, j)$ in $\mc{B}_i$ for all $j\in\mc{N}_i$.
			\FOR{each node $i\in \mc{V}$}
			\STATE %\footnote{The first activation of each node $i$ becomes possible only after it collects $x_j$ from all $j\in\mc{N}_i$} 
			keep \emph{receiving $x_j$ from neighbors and store $x_j$ in $\mc{B}_i$ until activation}\footnote{In the first iteration, each node $i\in\mc{V}$ can be activated only after it received $x_j$ from all $j\in\mc{N}_i$. If for some $j\in\mc{N}_i$, node $i$ receives multiple $x_j$'s, then it only stores the most recently received one and drop the remaining ones.}.
			\STATE set $x_{ij}=x_j$ for all $x_j\in\mc{B}_i$.% If multiple $(x_j,j)\in \mc{B}_i$ for some $j$, choose the most recently received one.
			\STATE empty $\mc{B}_i$.
			\STATE update $x_i$ according to \eqref{eq:DGDxiupdate}.
			\STATE send $x_i$ to all neighbors $j\in\mc{N}_i$.
			\ENDFOR
			\STATE \textbf{Until} a termination criterion is met.
		\end{algorithmic}
\end{minipage}
\end{algorithm}

    \subsection{Convergence Analysis}\label{sec:convana}
	In this section, we analyse the convergence of the asynchronous Prox-DGD under two typical models of asynchrony and show that the convergence can adjust to real delays rather than be determined by the worst-case delay. A key difference with the literature is that our results are established under a condition on the step-size $\alpha$ that does not depend on any information about the delays. We refer to this type of condition as a \emph{delay-free} step-size condition. {\color{black}Unless stated otherwise, all convergence results in this section assume that $W$ is an averaging matrix associated with $\mc{G}$ (see Definition \ref{def:def_av}).} Our first result allows for total asynchrony in the sense of Bertsekas and Tsitsiklis~\cite{bertsekas2015parallel}.
	\iffalse
	In addition to the smoothness condition in Assumption \ref{asm:prob}, we also assume strong convexity of all $f_i$'s.
	
	According to the discussion below Assumption \ref{asm:prob} and Lemma \ref{lemma:DGDATCoptimalitygap}, Assumptions \ref{asm:prob}-\ref{asm:strongconvexity} yield unique existence of the fixed point for DGD \eqref{eq:DGD} and DGD-ATC \eqref{eq:DGD-ATC}.
	\fi

	\begin{assumption}[total asynchrony]\label{asm:totalasynchrony}
		The following holds:
		\begin{enumerate}
			\item $\mc{K}_i$ is an infinite subset of $\N_0$ for each $i\in\mc{V}$.
			\item $\lim_{k\rightarrow+\infty} s_{ij}^k= +\infty$ for any $i\in\mc{V}$ and $j\in\mc{N}_i$.
		\end{enumerate}
	\end{assumption}
    {\color{black}Assumption \ref{asm:totalasynchrony} pose conditions on the update frequency of each node and the delays. Specifically, condition 1) requires each node to update for infinite times, and condition 2) allows the information delays $k-s_{ij}^k$ to grow arbitrarily large but requires that %no node can cease to update and 
    old information must eventually be purged from the system.}
    This assumption is well-suited for scenarios where communication and computation delays are unpredictable, %"unstable", 
        e.g., in massively parallel computing grids with heterogeneous computing nodes where delays can quickly add up if a node is saturated \cite{pmlr-v80-zhou18b}. Under total asynchrony, even the asynchronous consensus method \cite{nedic2010convergence}, which is equivalent to the asynchronous Prox-DGD for solving problem \eqref{eq:consensusprob} with $f_i=h_i\equiv 0$ {\color{black}(all agents collaborate to find a common decision)}, cannot be guaranteed to converge \cite[Example 1.2, Section 7.1]{bertsekas2015parallel}. Therefore, to establish convergence of the asynchronous Prox-DGD under total asynchrony, we assume that all $f_i$'s are strongly convex.
	\begin{assumption}\label{asm:strongconvexity}
		Each $f_i$ is $\mu_i$-strongly convex.
	\end{assumption}
	Clearly, Assumptions \ref{asm:prob} and \ref{asm:strongconvexity} imply Assumption \ref{asm:optsolexist}.
In the following theorem, we show that in contrast to the literature \cite{wu2017decentralized, tian2020achieving, cannelli2020asynchronous,su2022convergence, latafat2022primal, zhang2019fully,bianchi2021distributed} in which the parameters depend on the delay, the asynchronous Prox-DGD can converge to the same fixed point set as its synchronous counterpart, even under the total asynchrony assumption and a \emph{delay-free} step-size condition.
	
	\iffalse
	Moreover, to guarantee convergence of asynchronous iterative methods, it is in general necessary to assume 1) no node ceases to update and 2) old information is eventually purged from the system, which is referred to as total asynchrony \cite{bertsekas2015parallel} and formally stated below.\textbf{}
	\fi
	%
	%Apart from Assumption \ref{asm:totalasynchrony}, the partial asynchrony model \cite{bertsekas2015parallel} is also frequently used, which assumes bounded update duration and bounded delay and is more restrictive than Assumption \ref{asm:totalasynchrony}.

    \begin{theorem}[total asynchrony]\label{thm:total}
		Suppose that Assumptions \ref{asm:prob}, \ref{asm:totalasynchrony}, \ref{asm:strongconvexity} hold and \begin{equation}\label{eq:stepsizecond}
		\alpha\in \left(0, 2\min_{i\in\mc{V}} \frac{w_{ii}}{L_i}\right).
		\end{equation}
		Then, $\{\bx^k\}$ generated by the asynchronous Prox-DGD converges to a fixed point of the synchronous Prox-DGD.
	\end{theorem}
	\begin{proof}
		See Appendix \ref{append:thmtotal}.
	\end{proof}
	
	Under total asynchrony, 
	%it is unlikely to derive non-asymptotic convergence rates due to the lack of a 
	there is no lower bound on the update frequency of nodes and no upper bound on the delays, and we are only able to give asymptotic convergence guarantees and handle strongly convex objective functions.
	%This makes it difficult to derive non-asymptotic convergence rate estimates. 
	To derive non-asymptotic convergence rates and handle general convex objective functions, we assume the more restrictive notion of partial asynchrony~\cite{bertsekas2015parallel}.
	\begin{assumption}[partial asynchrony]\label{asm:partialasynchrony}
		There exist non-negative integers $B$ and $D$ such that
		\begin{enumerate}
			\item For every $i\in\mc{V}$ and for every $k\in\N_0$, at least one element in the set $\{k,\ldots,k+B\}$ belongs to $\mc{K}_i$.
			\item There holds
			\begin{equation*}
			k-D \le s_{ij}^k \le k
			\end{equation*}
			for all $i\in\mc{V}$, $j\in\mc{N}_i$, and $k\in \mc{K}_i$.
		\end{enumerate}
	\end{assumption}
    In Assumption \ref{asm:partialasynchrony}, $B$ and $D$ characterize the minimum update frequency and the maximal information delay, respectively. If $B=D=0$, then Assumption \ref{asm:partialasynchrony} reduces to the synchronous scheme where all local variables $x_i^k$ $\forall i\in\mc{V}$ are instantaneously updated at every iteration $k\in\N_0$ and there are no information delays. Moreover, it is worth noting that different from the bounded transmission delay assumption in \cite{zhang2019asyspa,zhang2019fully,assran2020asynchronous} that does not allow for packet loss, both Assumptions \ref{asm:totalasynchrony}, \ref{asm:partialasynchrony} allow for packet losses. Specifically, packets can be classified as ``used in update'' and ``unused (possibly lost or on the fly)'', and Assumptions \ref{asm:totalasynchrony}, \ref{asm:partialasynchrony} only pose restrictions on delays of the ``used in update'' packets ({\color{black}effective delay \cite{spiridonoff2020robust}}).
	
	To state our convergence result, we define the block-wise maximum norm for any $\bx=(x_1^T,\ldots, x_n^T)^T\in\mathbb{R}^{nd}$ as
	\begin{equation*}
	\|\bx\|_{\infty}^b = \max_{i\in\mc{V}} \|x_i\|.
	\end{equation*}
	The following theorem establishes asymptotic convergence of the asynchronous Prox-DGD for convex objective functions under partial asynchrony, and improves the rate to linear when the objective functions are strongly convex.
	\begin{theorem}[partial asynchrony]\label{thm:partial}
		Suppose that Assumptions \ref{asm:optsolexist}, \ref{asm:prob}, \ref{asm:partialasynchrony} and the step-size condition \eqref{eq:stepsizecond} hold. Then, $\{\bx^k\}$ generated by the asynchronous Prox-DGD converges to a fixed point $\bx^\star$ of its synchronous counterpart. 
  
      If, in addition, Assumption \ref{asm:strongconvexity} holds, then
		\begin{equation}\label{eq:linear_convergence}
		\|\mathbf{x}^k-\mathbf{x}^\star\|_{\infty}^b\le \rho^{\lfloor k/(B+D+1)\rfloor}\|\mathbf{x}^0-\mathbf{x}^\star\|_{\infty}^b,
		\end{equation}
		where
		\begin{align}
		\rho=\sqrt{1-\alpha\min_{i\in\mc{V}}\left(\mu_i\left(2-\frac{\alpha L_i}{w_{ii}}\right)\right)}.\label{eq:rhodgd}
		\end{align}
	\end{theorem}
	\begin{proof}
		See Appendix \ref{append:thmpartial}.
	\end{proof}

	   {\color{black}By Theorems \ref{thm:total} and \ref{thm:partial}, under the delay-free step-size condition \eqref{eq:stepsizecond} the asynchronous Prox-DGD can converge to a fixed point of its synchronous counterpart, and the convergence guarantees that we can give improve as the amount of asynchrony decreases from total to partial asynchrony. By Lemma \ref{lemma:DGDoptimalitygap}, under the conditions in Theorems \ref{thm:total} or \ref{thm:partial}, any fixed point of the synchronous Prox-DGD is a sub-optimal solution to problem \eqref{eq:consensusprob} with a bounded optimality gap. If, in addition, Assumptions \ref{asm:strongercond} or \ref{asm:identicalstrongcond} hold, then the optimality gap result can be improved according to Lemma \ref{lemma:stronggap}.
    %that is independent of the delays in the system, 
    %Moreover,  %Moreover, in Lemmas \ref{lemma:DGDoptimalitygap}--\ref{lemma:stronggap}, we have shown that \textit{under proper conditions}, the fixed point of the synchronous Prox-DGD is an approximate optimal solution to problem \eqref{eq:consensusprob} with bounded optimality gap.

    \begin{rem}[effect of network topology]
        To analyse the effects of network topology, we suppose that $L_i=L$ and $\mu_i=\mu$ for all $i\in\mc{V}$ and some $L\ge \mu>0$, and set $\alpha = \min_{i\in\mc{V}}  w_{ii}/L$ which is the midpoint of the step-size range \eqref{eq:stepsizecond}. By Lemmas \ref{lemma:DGDoptimalitygap}--\ref{lemma:stronggap}, the optimality gap satisfies
        \[\|x_i^\star - \bar{x}^\star\|\le O\left(\sqrt{\frac{\min_{i\in\mc{V}} w_{ii}}{L(1-\beta)}}\right),\quad\forall i\in\mc{V},\]
        or
        \[\|x_i^\star - \bar{x}^\star\|\le O\left(\frac{\min_{i\in\mc{V}} w_{ii}}{L(1-\beta)}\right),\quad\forall i\in\mc{V},\]
        depending on the assumptions. Moreover, the convergence rate $\rho$ in \eqref{eq:rhodgd} reduces to
        \[\rho=\sqrt{1-\frac{\mu\min_{i\in\mc{V}}w_{ii}}{L}}.\] From the above discussions, a smaller $\min_{i\in\mc{V}} w_{ii}$ yields a smaller optimality gap but a slower convergence, and a smaller $\beta$ leads to a smaller optimality gap. To analyse the effect of the network topology, we consider a particular yet widely used averaging matrix -- the Metropolis weight matrix \cite{nedic17} where $\min_{i\in\mc{V}} w_{ii}=\frac{1}{1+\max_{i\in\mc{V}} |\mc{N}_i|}$. For this weight matrix, a denser network often implies a larger $\max_{i\in\mc{V}} |\mc{N}_i|$ and a smaller $\beta$ (complete graph yields $\beta=0$), which, according to the above discussions, implies a smaller optimality gap but a slower convergence to the fixed point. %However, if we set a lower bound to $\min_i w_{ii}$, e.g., $\min_i w_{ii}\ge 1/2$ when we set $W=(\tilde{W}+I)/2$ for another mixing matrix $\tilde{W}$, then for denser networks, the optimality gap become smaller since $\beta$ is smaller, but the convergence rate is not significantly affected.
    \end{rem}}

    {\color{black} An issue in Theorems \ref{thm:total}--\ref{thm:partial} is the dependency of the step-size condition \eqref{eq:stepsizecond} on the typically unknown constant $L_i$. To address this issue, we modify \eqref{eq:DGDxiupdate} as
    \begin{equation}\label{eq:adapt_DGDxiupdate}
	x_i \leftarrow \prox_{\alpha h_i}\bigl(x_i+\gamma_i((w_{ii}x_i+\sum_{j\in\mc{N}_i} w_{ij}x_{ij} - \alpha \nabla f_i(x_i))-x_i)\bigr),
	\end{equation}
    where $\gamma_i$ is the step-size and \eqref{eq:adapt_DGDxiupdate} reduces to \eqref{eq:DGDxiupdate} when $\gamma_i=1$. The update \eqref{eq:adapt_DGDxiupdate} can be treated as the proximal gradient descent method for solving problem \eqref{eq:penalprob}, with local step-sizes $\gamma_i$, $i\in\mc{V}$. For the update \eqref{eq:adapt_DGDxiupdate}, we impose an adaptive step-size condition ($\gamma_i=\gamma_i^k$) that can guarantee its convergence.
    \begin{theorem}\label{thm:adapt_step}
        Suppose that Assumptions \ref{asm:prob}, \ref{asm:totalasynchrony}, \ref{asm:strongconvexity} hold. If for each $i\in\mc{V}$, $\gamma_i^{\min} = \min_{k\in\mc{K}_i} \gamma_i^k>0$ and
        \begin{equation}\label{eq:adapt_step}
            \gamma_i^k\le  \frac{1}{2\alpha L_i^k+1-w_{ii}},\quad\forall i\in\mc{V},~k\in\mc{K}_i,
        \end{equation}
        where $L_i^k=
    \frac{\langle \nabla f_i(x_i^{k+1}) - \nabla f_i(x_i^k), x_i^{k+1}-x_i^k\rangle}{\|x_i^{k+1}-x_i^k\|^2}$, then $\{\bx^k\}$ generated by \eqref{eq:adapt_DGDxiupdate} with $\gamma_i=\gamma_i^k$ converges to a fixed point of the synchronous Prox-DGD. If, in addition, Assumption \ref{asm:partialasynchrony} holds, so does  \eqref{eq:linear_convergence}  with $\rho = \sqrt{\frac{1}{1+\alpha\min_{i\in\mc{V}}(\gamma_i^{\min}\mu_i)}}$.
    \end{theorem}
    \begin{proof}
        See Appendix \ref{append:proof_thm_adapt_step}.
    \end{proof}
    }
    {\color{black} Compared to Theorems \ref{thm:total}--\ref{thm:partial}, Theorem \ref{thm:adapt_step} does not impose upper bounds on $\alpha$, which allows us to choose larger $\alpha$ that possibly achieve faster convergence. Moreover, the adaptive step-size condition \eqref{eq:adapt_step} does not include $L_i$ that may be difficult to determine in practice. Below, we give a simple back-tracking scheme to find $\gamma_i^k$ that meets \eqref{eq:adapt_step}:
    
    {\bf Back-tracking}: Set $c\in(0,1)$ and let $L_i^{-1}$ be a small positive number. At each iteration $k$,
    \begin{itemize}
        \item {\bf Step} 0: set $\gamma_i^k = c/(2\alpha L_i^{k-1}+1-w_{ii})$.
        \item {\bf Step} 1: update $x_i^{k+1}$ by \eqref{eq:adapt_DGDxiupdate} and compute $L_i^k$.
        \item {\bf Step} 2: check \eqref{eq:adapt_step}. If it holds, set $k=k+1$. Otherwise, set $\gamma_i^k = c/(2\alpha L_i^k+1-w_{ii})$ and go to {\bf Step} 1.
    \end{itemize}
    Note that if \eqref{eq:adapt_step} fails to hold, then $\gamma_i^k>1/(2\alpha L_i^k+1-w_{ii})$ and setting the new $\gamma_i^k= c/(2\alpha L_i^k+1-w_{ii})$ ensures $\gamma_i^k$ to be smaller than $c$ times of its old value. Moreover, because $L_i^k\le L_i$, as long as $\gamma_i^k\le 1/(2\alpha L_i+1-w_{ii})$, \eqref{eq:adapt_step} holds. Therefore,  \eqref{eq:adapt_step} will be satisfied within a finite number of attempts.}

    \subsubsection{Comparison with related methods}\label{sssec:comparison}
    To appreciate the novelties of the asynchronous Prox-DGD, we categorize the literature \cite{nedic2010convergence,zhang2019asyspa,doan2017convergence,assran2020asynchronous,zhang2019fully,wu2017decentralized,tian2020achieving, kungurtsev2023decentralized,ubl2021totally,mishchenko2018delay,bianchi2021distributed} on asynchronous consensus optimization according to their step-sizes:
    
    \noindent\textbf{\emph{Diminishing step-size}}: References \cite{zhang2019asyspa,doan2017convergence,assran2020asynchronous,kungurtsev2023decentralized,spiridonoff2020robust} consider diminishing step-sizes that decrease to zero as the number of iterations increases to infinity. The step-sizes are delay-free, but decrease rapidly and can often lead to slow practical convergence, especially for deterministic optimization problems.  {\color{black}Under proper problem assumptions and partial asynchrony (Assumption \ref{asm:partialasynchrony}), \cite{zhang2019asyspa,doan2017convergence,assran2020asynchronous,kungurtsev2023decentralized,spiridonoff2020robust} establish asymptotic/sublinear convergence.}
    %Moreover, \cite{zhang2019asyspa,sirb2016consensus,doan2017convergence,assran2020asynchronous,kungurtsev2023decentralized} all focus on partial asynchrony, while our algorithms can converge under total asynchrony.
	
    \noindent\textbf{\emph{Delay-dependent fixed step-size}}: This line of work includes~\cite{wu2017decentralized, tian2020achieving,zhang2019fully,bianchi2021distributed}. They assume the existence of a known upper bound $\tau$ on the information delay and use fixed parameters that rely on (and decrease with) $\tau$. {\color{black}For example, the step-sizes in \cite{wu2017decentralized,bianchi2021distributed} are of the order $O(\frac{1}{\tau+1})$, the step-size in \cite{zhang2019fully} is of $O((n^2\tau)^{-n\tau})$, and the step-size in \cite{tian2020achieving} is of $O(\eta^{n\tau})$ where $\eta=\min_{i\in\mc{V}} w_{ii}\in(0,1)$.} These step-sizes are difficult to use because the delay bound is usually unknown in practice. Moreover, in many systems, the worst-case delay is large but rarely attained, so step-sizes tuned in this way are unnecessarily small and lead to slow convergence. {\color{black}Among these works, \cite{wu2017decentralized} derive asymptotic convergence in solving problem \eqref{eq:consensusprob} with non-zero $h_i$'s, \cite{tian2020achieving,zhang2019fully} show linear convergence for problem \eqref{eq:consensusprob} with zero $h_i$'s and strongly convex and smooth $f_i$'s under the partial asynchrony assumption, and \cite{bianchi2021distributed} studies a different problem.}

    \noindent \textbf{\emph{Delay-free fixed step-size}}: This category includes \cite{nedic2010convergence, ubl2021totally, mishchenko2018delay}. The step-sizes are easier to determine and are typically larger than the delay-dependent ones, leading to a faster practical convergence compared to the other two types of step-sizes. However, existing works in this category can only handle quadratic problems \cite{nedic2010convergence, ubl2021totally} or star networks \cite{mishchenko2018delay}. {\color{black}They all establish linear convergence under the partial asynchrony assumption, but for  different problem classes. Specifically, \cite{ubl2021totally, mishchenko2018delay} assume the $f_i$'s in problem \eqref{eq:consensusprob} to be strongly convex and smooth and \cite{ubl2021totally} also require the $h_i$'s to be zero.}

    Our step-size \eqref{eq:stepsizecond} belongs to the last category, and has several novelties with respect to the state-of-the-art.
    %of the delay-free fixed step-size. Below, we distinguish our results with the literature \cite{nedic2010convergence,zhang2019asyspa,doan2017convergence,assran2020asynchronous,zhang2019fully,wu2017decentralized,tian2020achieving, kungurtsev2023decentralized,ubl2021totally,mishchenko2018delay,bianchi2021distributed} based on problem assumptions, asynchrony model, network model, and convergence properties:
     \begin{enumerate}
        % \item Among algorithms that can deal with both non-quadratic objective functions and non-star communication networks, Algorithm \ref{alg:DGD} is the only one that can converge with delay-free step-size conditions, the only one that can 
         \item Existing works using delay-free fixed step-sizes can only handle quadratic problems \cite{nedic2010convergence, ubl2021totally} or star networks \cite{mishchenko2018delay}. In contrast, Algorithm \ref{alg:DGD} can deal with both non-quadratic objective functions and non-star communication networks (Theorems \ref{thm:total}--\ref{thm:partial}), which is a substantial improvement.
        % \item Among algorithms that can converge under total asynchrony, Algorithm \ref{alg:DGD} is the only one that can simultaneously handle weakly convex, non-quadratic objective functions and non-star networks.
        \item Only Algorithm \ref{alg:DGD} and the algorithm in \cite{wu2017decentralized} can address non-quadratic objective functions and non-identical local constraints over non-star networks. However, \cite{wu2017decentralized} uses a probabilistic model to characterize the node that updates at each iteration, which is easier to analyse but less practical than partial asynchrony.
        \item {\color{black}Although \cite{zhang2019asyspa,doan2017convergence,assran2020asynchronous,zhang2019fully,wu2017decentralized,tian2020achieving, kungurtsev2023decentralized,bianchi2021distributed} can converge to the exact optimum, while Algorithm \ref{alg:DGD} only converges to an approximate minimum, the parameter range of Algorithm \ref{alg:DGD} is wider which typically leads to faster practical convergence. Moreover, the linear convergence rates in Theorem \ref{thm:partial}--\ref{thm:adapt_step} are derived for non-zero $h_i$'s while \cite{zhang2019fully} and \cite{tian2020achieving}, who also derive linear rates, require $h_i=0$ $\forall i\in\mc{V}$.}
     \end{enumerate}
Admittedly, the algorithms in~\cite{nedic2010convergence,tian2020achieving,zhang2019asyspa,kungurtsev2023decentralized,assran2020asynchronous,zhang2019fully} can handle directed networks, which is more general than the undirected network model in our work. % However, the asynchronous Prox-DGD does not require each pair of neighboring nodes to simultaneously communicate with each other

\iffalse    {\color{black} 
    The discussion is good to have, and should be in the paper, but now the same thing is discussed three times. Once in the beginning of this paragraph, once at the end, and once in the figure. I believe that once is enough (or possibly twice). Whenever I see a table like TABLE I, I immediately believe that the work is incremental (since the difference with it and the state of the art could not be explained in words). Additionally, the table suggest that your algorithm does not bring any additions to [17], which is misleading, I think.
    }
    \fi
    
    % In contrast, diminishing step-sizes in the works \cite{zhang2019asyspa,doan2017convergence,assran2020asynchronous,kungurtsev2023decentralized} decrease rapidly and can lead to slow practical convergence; The delay-dependent fixed step-sizes in \cite{wu2017decentralized, tian2020achieving,zhang2019fully,bianchi2021distributed} assume the existence of an upper bound on the information delay and use fixed parameters relying on and decreasing with the delay bound. This type of step-sizes suffers from difficult step-size determination and unnecessary slow convergence because the delay bound is often unknown and hard to obtain in advance, and typically large but rarely attained, which leads to small step-sizes and further slows down the convergence process.
    
    % Delay-free and non-diminishing step-sizes can lead to fast practical convergence but existing works \cite{nedic2010convergence, ubl2021totally, mishchenko2018delay} can only handle a limited class of problems or specific communication networks.

    \subsection{Convergence in terms of the Realized Level of Asynchrony}
    
%    I am not sure if I agree that it is intuitive that a delay-free algorithm has good adaptivity. Perhaps if you say that it has delay-free parameter, but still Theorem 2 demonstrates that it has state-of-the-art convergence under partial asynchrony. That could indicate a good level of adaptivity. 
\iffalse
{\color{black}
    The result in this section are interesting and important, but I believe that you could make them even more explicit than the figure. Because small and worst are very particular delay realizations and you should be able to give the value of $m^k$ explicitly for these two cases. In this way, you would actually \emph{prove} that convergence is fast in the small-delay case.
    }
\fi 
    As shown in Fig.~\ref{fig:histogram}, the typical delays in real systems may be (and often are)  much smaller than the worst-case.  In these scenarios, the convergence of asynchronous methods whose parameters depend on the worst-case delay may be slowed down significantly by the conservative parameters (See Fig. 1 in \cite{Wu23ICML}). %However, the parameter range in our algorithm does not involve any delay information, which intuitively implies better adaptivity of the convergence to the actual delays. 
    In this subsection, we will show that the convergence of the asynchronous Prox-DGD {\color{black} adjusts} to the actual level of asynchrony ({\color{black}the update frequency of each node and the delays}) rather than being hampered by the worst-case delay. {\color{black}Our main idea is to split the set $\N_0$ of all iteration indexes into the union of disjoint intervals $[k^m, k^{m+1})$, $m\in\N_0$, and show that after each interval, an upper bound on the optimality error $\|\bx^k-\bx^\star\|_{\infty}^b$ decreases. In particular, for any $m\in\N_0$,
    \begin{equation}\label{eq:shrink_error_adapt}
	   	\|\mathbf{x}^k-\mathbf{x}^\star\|_{\infty}^b\le \rho\|\mathbf{x}^{k^m}-\mathbf{x}^\star\|_{\infty}^b,\quad \forall k\ge k^{m+1}.
	\end{equation}
    The length of the intervals $[k^m, k^{m+1})$, $m\in\N_0$ reflects the level of asynchrony, which will be detailed later below Theorem \ref{thm:adapt}.}
    
%    We define $\{k^m\}_{m\in\N_0}$ be such that the following Markov property: for any $m\in\N_0$, when $k+1\ge k^{m+1}$, the computation of $\bx^{k+1}$ using \eqref{eq:eachblocks} does not depend on any  $\bx^{\ell}$ with $\ell<k^m$, i.e., for any $i\in\mc{V}$ and $j\in\bar{\mc{N}}_i$, if $k\ge k^{m+1}$ and $k\in\mc{K}_i$, then $s_{ij}^k\ge k^m$. 
    
    To introduce $\{k^m\}_{m\in\N_0}$, we define
	\begin{equation}
	t_i^k = \max\{\ell: \ell\le k, \ell\in \mc{K}_i\}
	\end{equation}
	as the iteration index of the most recent update by node $i$. {\color{black}Because $x_i^{t+1}=x_i^t$ for all $t\in [t_i^k+1, k]$, we have $x_i^{k+1}=x_i^{t_i^k+1}$ for any $k\in\N_0$.} Therefore, for all $i\in\mc{V}$,
    \begin{equation}\label{eq:eachblocks}
	x_i^{k+1} = \prox_{\alpha h_i}\left(\sum_{j\in\bar{\mc{N}}_i} w_{ij}x_j^{s_{ij}^{t_i^k}}-\alpha\nabla f_i(x_i^{t_i^k})\right).
	\end{equation}
    {\color{black}Note that \eqref{eq:eachblocks} is not an update rule, but expresses $x_i^{k+1}$ for all $i\in\mc{V}$ even if $i$ does not update at $k$.}
    The maximum age of the information used to obtain each block of $\bx^{k+1}$ is then
	\begin{equation}
	\tau^k = k - \min_{i\in\mc{V}}\min_{j\in\bar{\mc{N}}_i} s_{ij}^{t_i^k}.
	\end{equation}
    We define $\{k^m\}_{m\in\N_0}$ as: $k^0=0$ and
    \begin{equation}\label{eq:defkm}
	k^{m+1} = \min\{k: t-\tau^t\ge k^m~\forall t\ge k\}+1, \forall m\in\N_0,
    \end{equation}
    which describes the following Markov property: for any $m\in\N_0$, when $k+1\ge k^{m+1}$, the computation of $\bx^{k+1}$ does not depend on any  $\bx^{\ell}$ with $\ell<k^m$ (by \eqref{eq:eachblocks} and $k-\tau^k\ge k^m$).
    
%    such that the Markov property holds: for any $m\in\N_0$, when $k+1\ge k^{m+1}$, the computation of $\bx^{k+1}$ using \eqref{eq:eachblocks} does not depend on any  $\bx^{\ell}$ with $\ell<k^m$ (since $k-\tau^k\ge k^m$).

   % show that for an increasing index sequence $\{k^m\}_{m\in\N_0}\subseteq \N_0$,

    \iffalse To this end, define
	\begin{equation}
	t_i^k = \max\{\ell: \ell\le k, \ell\in \mc{K}_i\}
	\end{equation}
	as the iteration of the most recent update of node $i$. {\color{black}Because $x_i^{t+1}=x_i^t$ for all $t\in [t_i^k+1, k]$, we have $x_i^{k+1}=x_i^{t_i^k+1}$.} Therefore, for all $i\in\mc{V}$,
    \begin{equation}\label{eq:eachblocks}
	x_i^{k+1} = \prox_{\alpha h_i}\left(\sum_{j\in\bar{\mc{N}}_i} w_{ij}x_j^{s_{ij}^{t_i^k}}-\alpha\nabla f_i(x_i^{t_i^k})\right).
	\end{equation}
    The maximum age of the information used to obtain each block of $\bx^{k+1}$ is then
	\begin{equation}
	\tau^k = k - \min_{i\in\mc{V}}\min_{j\in\bar{\mc{N}}_i} s_{ij}^{t_i^k}.
	\end{equation}
    We also define a sequence $\{k^m\}_{m\in\N_0}$ as: $k^0=0$ and
    \begin{equation}\label{eq:defkm}
	k^{m+1} = \min\{k: t-\tau^t\ge k^m~\forall t\ge k\}+1, \forall m\in\N_0,
    \end{equation}
    which describes the following Markov property: for any $m\in\N_0$, when $k+1\ge k^{m+1}$, the computation of $\bx^{k+1}$ using \eqref{eq:eachblocks} does not depend on any  $\bx^{\ell}$ with $\ell<k^m$ (since $k-\tau^k\ge k^m$).
    \fi
    \begin{theorem}\label{thm:adapt}
		Suppose that Assumptions \ref{asm:prob}, \ref{asm:strongconvexity}, and the step-size condition \eqref{eq:stepsizecond} hold. Let $\{\bx^k\}$ be generated by the asynchronous Prox-DGD. If $k\ge k^m$, then \eqref{eq:shrink_error_adapt} holds and
	\begin{equation*}
	   	\|\mathbf{x}^k-\mathbf{x}^\star\|_{\infty}^b\le \rho^m\|\mathbf{x}^0-\mathbf{x}^\star\|_{\infty}^b,
	\end{equation*}
    where $\rho$ is given in \eqref{eq:rhodgd}.
    \end{theorem}
    \begin{proof}
		See Appendix \ref{append:thmadapt}.
	\end{proof}
    By Theorem \ref{thm:adapt}, the convergence speed of the asynchronous Prox-DGD is determined by $m^k=\max\{m:k^m\le k\}$. {\color{black}To be more clear, $\{k^m\}_{m\in\N_0}$ splits $\N_0$ into disjoint intervals and $m^k$ indicates the interval that $k$ belongs to ($k\in [k^{m^k}, k^{m^{k+1}})$).} Under Assumption \ref{asm:partialasynchrony}, \cite{Feyzmahdavian23} shows that $\tau^k\le B+D$ $\forall k\in\N_0$. Then, by the definitions of $\{k^m\}$ and $\{m^k\}$, we have $k^m\le m(B+D+1)$ and
    \begin{equation}\label{eq:mkmupperbound}
	m^k\ge \left\lfloor \frac{k}{B+D+1}\right\rfloor,~\forall k\in\N_0,
    \end{equation}
    which recovers the guarantees in Theorem \ref{thm:partial}. However, $m^k$ may be much larger than the lower bound \eqref{eq:mkmupperbound}, especially if the typical level of asynchrony is much smaller than the worst case, indicating a faster convergence. To make the discussion more precise, we consider two extreme cases.
    \begin{itemize}
        \item %At one extreme is a scenario where the worst-case bounds $B$ and $D$ are attained only once, and the algorithm runs with minimal information delays afterward. 
        At one extreme is a scenario where the worst-case bounds $B$ and $D$ are attained only once, and the algorithm runs cyclically without information delays afterwards. Then,
    \begin{equation}\label{eq:mkbest}
            m^k \ge \frac{k-(B+D+1)}{n}, \forall k\in\N_0.
        \end{equation}
        \item At the other extreme are scenarios where (15) hold with equality. % This happens, for example, in a system where one of the nodes updates every B iterations, and uses information that is delayed by D
        This happens, for example, in a system where for every $t\in\N_0$, if a node updates at the $t(B+D+1)$th iteration, and its last update is delayed by $D$ and occurs at the $t(B+D+1)-B$th iteration.
    \end{itemize}

By \eqref{eq:mkbest}, the quantity $m^k$ under the most favorable delays is determined by $n$ rather than $B$ and $D$ and is much larger than that of the worst-case~\eqref{eq:mkmupperbound}. In practice, the convergence will be somewhere in between these two extremes. To illustrate this, we run the asynchronous Prox-DGD for $10,000$ iterations in a similar setting as we used for recording the delays shown in Fig. \ref{fig:histogram}. Specifically, we used $n=10$ and recorded $B=37$ and $D=40$ during the execution. Fig. \ref{fig:adaptivity} plots the realized $\{m^k\}$ (``simu'') and those obtained under the ``best'' and ``worst'' partial asynchronous realizations defined above. We see that the realized $m^k$ values lie between the two theoretical extremes and are roughly triple those of the ``worst'' model. %Moreover, $m^k$ increases significantly ($50$, $111$, $613$ at $k=10000$) when the level of asynchrony decreases (worst $\rightarrow$ simulated $\rightarrow$ best), indicating that the convergence of the asynchronous Prox-DGD adapts well to the actual asynchrony level.}
    \begin{figure}
        \centering
        \includegraphics[scale=0.6]{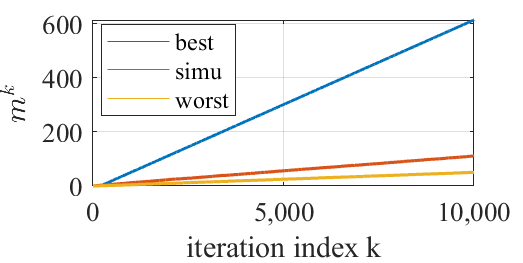}
        \caption{realized $m^k$ under different delay patterns}
        \label{fig:adaptivity}
        \vspace{-0.2cm}
    \end{figure}

    \section{Asynchronous DGD-ATC for Smooth Problems}
    DGD-ATC \cite{ATC_first} is a variant of DGD \cite{yuan2016convergence} using the adapt-then-combine technique. It  solves problem \eqref{eq:consensusprob} with $h_i\equiv 0$ $\forall i\in\mc{V}$ through the updates
	%
	%DGD-ATC has superior practical performance as experimentally demonstrated in \cite{pu2020asymptotic}, whose update formula is
	\begin{equation}\label{eq:DGD-ATC}
	\mathbf{x}^{k+1}= \mb{W}(\bx^k - \alpha \nabla f(\bx^k)),
	\end{equation}
	where $\alpha>0$ is the step-size and $\mb{W}=W\otimes I_d$ for an averaging matrix $W\in\R^{n\times n}$ associated with $\mc{G}$. {\color{black}%The main difference between Prox-DGD \eqref{eq:DGD} with $h_i=0\forall i\in\mc{V}$ and DGD-ATC \eqref{eq:DGD-ATC} is the order of the ``average" step and the ``gradient" step.
 Different from Prox-DGD \eqref{eq:DGD} with $h_i=0$ $\forall i\in\mc{V}$ where each node $i$ first averages $x_i$ with $x_j$ from its neighbors and then performs a gradient descent step, in DGD-ATC \eqref{eq:DGD-ATC} each node first executes a gradient descent step and then averages the resulting local term with that of its neighbors.} %{\color{black} Clarify the difference between DGD and DGD-ATC} which can be achieved by letting $W=(W'+I)/2$ for any averaging matrix $W'\in\mathbb{R}^{n\times n}$ associated with $\mc{G}$ \cite[section 2.3]{shi2015extra}. 
 In this section, we will show that similar to the asynchronous Prox-DGD, the asynchronous DGD-ATC can also converge under a delay-free step-size condition and has an even faster convergence when handling smooth problems, which will be demonstrated theoretically later and numerically in Section \ref{sec:exp}.

    \iffalse
     DGD-ATC \eqref{eq:DGD-ATC} can be viewed as a weighted gradient method \cite[Section 1.2.1]{bertsekas1995nonlinear} with weight matrix $\mb{W}$ for solving
	\begin{equation}\label{eq:problemATC}
	\underset{\bx\in\mathbb{R}^{nd}}{\operatorname{minimize}}~f(\bx)+\frac{\bx^T(\mb{W}^{-1}-I)\bx}{2\alpha}.
	\end{equation}
	Consequently, to study DGD-ATC, we assume that problem \eqref{eq:problemATC} has a non-empty optimal solution set.
	\begin{assumption}
		For any $\alpha>0$, the optimal solution set of problem \eqref{eq:problemATC} is non-empty.
	\end{assumption}
	\fi

    {\color{black}The papers \cite{zhao2015asynchronous,zhao2015asynchronousII} analyse DGD-ATC with fixed step-sizes, but only for strongly convex objective functions. In the lemma below, we show that under Assumption \ref{asm:prob} that does not require strong convexity, DGD-ATC admits at least one fixed-point, and the optimality gap is similar to that of the synchronous Prox-DGD in Lemma \ref{lemma:stronggap}.}

    \begin{lemma}\label{lemma:DGDATCoptimalitygap}
		Suppose that Assumption \ref{asm:prob} holds, $h_i\equiv 0$ $\forall i\in\mc{V}$, and the optimal set of problem \eqref{eq:consensusprob} is bounded. Then the fixed point set of DGD-ATC \eqref{eq:DGD-ATC} is non-empty and for any fixed point $\bx^\star\in\mbb{R}^{nd}$, $F(\bx^\star)\le F_{\operatorname{opt}}$ and \eqref{eq:betaDGD}--\eqref{eq:funcave} hold.
	\end{lemma}
	\begin{proof}
		See Appendix \ref{append:DGDATCoptimalitygap}.
	\end{proof}
%        Although the bounds on the optimality gap of both algorithms are of the same order (Lemmas \ref{lemma:stronggap}, \ref{lemma:DGDATCoptimalitygap}) and take the same value by Appendix, their fixed points are optimum of different problems (problems \eqref{eq:penalprob}, \eqref{eq:problemATC}) and by our simulation in Section \ref{sec:exp}, the asynchronous DGD-ATC usually yields better accuracy.
%	In Lemma \ref{lemma:DGDATCoptimalitygap} we do not require any condition on $\alpha>0$. 
	
	\subsection{Algorithm Description}
    The asynchronous DGD-ATC algorithm can be implemented in a similar setting as Algorithm \ref{alg:DGD} -- each node $i\in {\mathcal V}$ is activated at discrete points in time, does not need to be active to receive information from its neighbors, and uses a buffer ${\mathcal B}_i$ to store the information it receives from its neighbors between activation times. In the initialization, it shares $y_i=x_i-\alpha\nabla f_i(x_i)$ with all neighbors. When it is activated, it reads all $y_{ij}$ (the most recent $y_j$ it has received from neighboring node $j$) from the buffer, performs the update
    	% To implement DGD-ATC \eqref{eq:DGD-ATC} in an asynchronous way, each node $i$ holds $x_i\in\mathbb{R}^d$, $y_i\in\mathbb{R}^d$, and $y_{ij}\in\mathbb{R}^d$ for $j\in \mc{N}_i$, where $x_i$ is node $i$'s current local iterate, $y_i=x_i - \alpha \nabla f_i(x_i)$, and $y_{ij}$, $j\in\mc{N}_i$ records the most recent value of $y_j$ that node $i$ received from node $j$. Once activated, node $i\in\mc{V}$ first reads all $y_{ij}$ in its buffer $\mc{B}_i$. Next, it updates $x_i$ by
	\begin{equation}\label{eq:ATCxiupdate}
	x_i \leftarrow w_{ii}y_i+\sum_{j\in \mc{N}_i} w_{ij}y_{ij},
	\end{equation}
	and broadcasts $y_i=x_i-\alpha \nabla f_i(x_i)$ to all $j\in\mc{N}_i$ before going back to sleep. Once a node $j\in\mc{N}_i$ receives $y_i$, it stores $y_i$ in its buffer $\mc{B}_j$; see the detailed implementation in Algorithm \ref{alg:DGD-ATC}.

 % 1- \alpha min (mu_i(2-alpha L_i)
 %
 % 1- \alpha min(mu_u(2-\alpha L_i/w_ii)
	
    \begin{algorithm}[t!]
        \makeatletter
        \renewcommand\footnoterule{%
            \kern-3\p@
        \hrule\@width.4\columnwidth
        \kern2.6\p@}
        \makeatother
        \caption{Asynchronous DGD-ATC}
		\label{alg:DGD-ATC}
        \begin{minipage}{\linewidth}
        \renewcommand{\thempfootnote}{\arabic{mpfootnote}}
        \begin{algorithmic}[1]
			\STATE {\bfseries Initialization:} All the nodes agree on $\alpha>0$, and cooperatively set $w_{ij}$ $\forall \{i,j\}\in\mc{E}$.
			\STATE Each node $i\in\mc{V}$ chooses $x_i\in\mathbb{R}^d$, creates a local buffer $\mc{B}_i$, sets $y_i=x_i-\alpha\nabla f_i(x_i)$ and shares it with all $j\in\mc{N}_i$.
			\FOR{each node $i\in \mc{V}$}
			\STATE keep \emph{receiving $y_j$ from neighbors and store $y_j$ in $\mc{B}_i$ until it is activated}\footnote{In the first iteration, each node $i\in\mc{V}$ can be activated only after it received $y_j$ from all $j\in\mc{N}_i$.}. If for some $j\in\mc{N}_i$, node $i$ receives multiple $y_j$'s, then store the most recently received one.
			\STATE set $y_{ij}=y_j$ for all $y_j\in\mc{B}_i$.
%			\STATE set $y_{ij}=y_j$ for all $y_j\in\mc{B}_i$. %If multiple $(y_j,j)\in \mc{B}_i$ for some $j$, choose the most recently received one.
			\STATE empty $\mc{B}_i$.
			\STATE update $x_i$ by \eqref{eq:ATCxiupdate}.
			\STATE set $y_i=x_i-\alpha\nabla f_i(x_i)$.
			\STATE share $y_i$ with all neighbors $j\in\mc{N}_i$.
			\ENDFOR
			\STATE \textbf{Until} a termination criterion is met.
		\end{algorithmic}
        \end{minipage}
	\end{algorithm}
	
	Note that each $y_{ij}$ in \eqref{eq:ATCxiupdate} is a delayed $x_j-\alpha \nabla f_j(x_j)$. Then, similar to \eqref{eq:asyDGDupdateindex}, the asynchronous DGD-ATC can be described as follows. For each $i\in\mc{V}$ and $k\in\N_0$,
	\begin{equation}\label{eq:ATCupdateindex}
	x_i^{k+1} \!\!=\! \begin{cases}
	\sum_{j\in\bar{\mc{N}_i}} w_{ij}(x_j^{s_{ij}^k}\!-\!\alpha\nabla f_j(x_j^{s_{ij}^k})), & k\in \mc{K}_i,\\
	x_i^k, &\! \text{otherwise},
	\end{cases}
	\end{equation}
	where $k\in\N_0$ is the iteration index, $\mc{K}_i\subseteq \N_0$ denotes the set of iterations where node $i$ updates $x_i$, $s_{ij}^k\in [0, k]$, $j\in\mc{N}_i$ is the iteration index of the most recent $y_j$ that node $i$ has received from $j$, and $s_{ii}^k=k$. When $\mc{K}_i=\N_0$ $\forall i\in\mc{V}$ and $s_{ij}^k=k$ $\forall \{i,j\}\in\mc{E}, \forall k\in\N_0$, the asynchronous DGD-ATC reduces to the synchronous DGD-ATC.
	
	\subsection{Convergence Analysis}\label{ssec:convrateATC}

        We show similar convergence results for the asynchronous DGD-ATC as Theorems \ref{thm:total}, \ref{thm:partial}, and \ref{thm:adapt}. In addition to being an averaging matrix, this subsection assumes $W$ to be positive definite, which can be satisfied by letting $W = \frac{W'+I}{2}$ for any averaging matrix $W'$ associated with $\mc{G}$ \cite[section 2.3]{shi2015extra}.
        
	\begin{theorem}[total asynchrony]\label{thm:totalATC}
		Suppose that all the conditions in Lemma \ref{lemma:DGDATCoptimalitygap} and Assumptions \ref{asm:totalasynchrony}, \ref{asm:strongconvexity} hold and    \begin{equation}\label{eq:stepsizecondATC}
		\alpha\in\left(0,\frac{2}{\max_{i\in\mc{V}} L_i}\right).
		\end{equation}
		Then, $\{\bx^k\}$ generated by the asynchronous DGD-ATC converges to some fixed point of its synchronous counterpart. 
	\end{theorem}
	\begin{proof}
		See Appendix \ref{append:thmtotalATC}.
	\end{proof}

	\begin{theorem}[partial asynchrony]\label{thm:partialATC}
		Suppose that all the conditions in Lemma \ref{lemma:DGDATCoptimalitygap}, Assumption \ref{asm:partialasynchrony}, and the step-size condition \eqref{eq:stepsizecondATC} hold. Then, $\{\bx^k\}$ generated by the asynchronous DGD-ATC converges to some fixed point $\bx^\star$ of its synchronous counterpart. If, in addition, Assumption \ref{asm:strongconvexity} holds, then
		\begin{equation*}
		\|\mathbf{x}^k-\mathbf{x}^\star\|_{\infty}^b\le \hat{\rho}^{\lfloor k/(B+D+1)\rfloor}\|\mathbf{x}^0-\mathbf{x}^\star\|_{\infty}^b,
		\end{equation*}
		where
		\begin{align}
		\hat{\rho}=\sqrt{1-\alpha\min_{i\in\mc{V}}\left(\mu_i(2-\alpha L_i)\right)}\in(0,1).\label{eq:rhodgdatc}
		\end{align}
	\end{theorem}
	\begin{proof}
		See Appendix \ref{append:thmpartialATC}.
	\end{proof}

	\begin{theorem}\label{thm:adaptATC}
		Suppose that all the conditions in Lemma \ref{lemma:DGDATCoptimalitygap}, Assumptions \ref{asm:strongconvexity}, \ref{asm:partialasynchrony}, and the step-size condition \eqref{eq:stepsizecondATC} hold. Let $\{\bx^k\}$ be generated by the asynchronous DGD-ATC. If $k\ge k^m$ for some $m\in\N_0$, then
		\begin{equation*}
		\|\mathbf{x}^k-\mathbf{x}^\star\|_{\infty}^b\le \hat{\rho}^m\|\mathbf{x}^0-\mathbf{x}^\star\|_{\infty}^b.
		\end{equation*}
	\end{theorem}
        \begin{proof}
            See Appendix \ref{append:thmadaptATC}.
        \end{proof}
        
        Lemma \ref{lemma:DGDATCoptimalitygap} and Theorems \ref{thm:totalATC}--\ref{thm:adaptATC} indicate that the asynchronous DGD-ATC has the same error bounds as the asynchronous Prox-DGD on smooth problems, but it may converge faster. Since $w_{ii}\in (0,1)$, the asynchronous DGD-ATC allows for larger stepsizes \eqref{eq:stepsizecondATC} than the asynchronous Prox-DGD \eqref{eq:stepsizecond}. A larger step-size $\alpha$ tends to lead to faster convergence in practice. Indeed, even when using the same $\alpha$,  asynchronous DGD-ATC has a stronger convergence guarantee  since $w_{ii}\in (0,1)$ implies that $\hat{\rho}\le \rho$. 
        %The convergence rate $\hat{\rho}$ does not change with any network-related parameters while $\rho$ increases when the $w_{ii}$'s decrease. This indicates, 
        With the popular Metropolis averaging matrix~\cite{nedic17} where, $w_{ii}=1-\sum_{j\in\mc{N}_i} \frac{1}{\max(|\mc{N}_i|, |\mc{N}_j|)+1}$, the advantage of the asynchronous DGD-ATC over Prox-DGD becomes stronger in  denser networks.

         % \begin{equation*}
         %     \begin{split}
         %         \rho^2-\hat{\rho}^2 &= \alpha \left(\min_{i\in\mc{V}}\mu_i(2-\alpha L_i)-\min_{i\in\mc{V}} \mu_i\left(2-\alpha \frac{L_i}{w_{ii}}\right)\right)\\
         %         &\ge \alpha\min_{i\in\mc{V}} \left(\mu_i(2-\alpha L_i)-\mu_i(2-\alpha \frac{L_i}{w_{ii}})\right)\\
         %         &= \alpha^2\min_{i\in\mc{V}} \mu_iL_i\left(\frac{1}{w_{ii}}-1\right)\ge 0,
         %     \end{split}
         % \end{equation*}
         % where the last step uses $w_{ii}\in (0,1)$ $\forall i\in\mc{V}$.

        The comparison of the asynchronous DGD-ATC with the literature follows that of the asynchronous Prox-DGD in Section \ref{sssec:comparison}), with the exception that the former cannot solve non-smooth problems. We omit the comparison for brevity.

        % {\color{teal} The convergence speed of DGD-ATC is faster than DGD by comparing two $\rho$'s.

        % Let's also compare the optimality gaps of DGD and DGD-ATC under the same condition.
        % }

	\section{Numerical Experiments}\label{sec:exp}

    We evaluate the practical performance of the asynchronous Prox-DGD and DGD-ATC on decentralized training. Each node has its own local data set and collaborates to train a common classifier using logistic regression with an elastic-net regularization. This problem is on the form~\eqref{eq:consensusprob} with 
    \begin{align*}
        f_i(x) &=\frac{1}{|\mc{D}_i|}\sum_{j\in \mc{D}_i} \left(\log(1+e^{-b_j(a_j^Tx)})+\frac{\lambda_2}{2}\|x\|^2\right),\\
        h_i(x) &=\lambda_1\|x\|_1,
    \end{align*}
    where $\mc{D}_i$ represents the set of samples held by node $i$, $a_j$ is the feature of the $j$th sample, $b_j$ is its label, and $\lambda_1,\lambda_2$ are regularization parameters. We set $\lambda_2=10^{-3}$ and consider two values $10^{-3}$ and $0$ for $\lambda_1$ to simulate non-smooth and smooth problems, respectively. The experiments use the training sets of Covertype \cite{Dua:2019} that includes $581012$ samples with feature dimension $54$.

    We compare our algorithms with their synchronous counterparts and the asynchronous PG-EXTRA \cite{wu2017decentralized}. We do not compare with the algorithms in~\cite{zhang2019asyspa,doan2017convergence,kungurtsev2023decentralized,assran2020asynchronous,bianchi2021distributed,zhang2019fully,tian2020achieving} because they require either Lipschitz continuous objective functions \cite{zhang2019asyspa,doan2017convergence,kungurtsev2023decentralized} that does not hold for the simulated problem or negligibly small step-sizes ($<10^{-5}$ in our setting) \cite{assran2020asynchronous,zhang2019fully,tian2020achieving,bianchi2021distributed}.

    We set $n=16$, partition and allocate the data sets uniformly to each node, and implement all the methods on a multi-thread computer using the message-passing framework MPI4py \cite{dalcin2008mpi} where each thread serves as a node. Different from experiments in existing work, that use theoretical models of the computation and communication time \cite{bianchi2021distributed,wu2017decentralized,tian2020achieving} or delays \cite{doan2017convergence,zhang2019asyspa}, the timing in our experiments are generated from real computations and communications of the threads. The communication graph $\mc{G}$ has $20$ links and is connected, with a randomly generated topology. Each node $i\in\mc{V}$ is activated once its buffer $\mc{B}_i$ includes messages from at least $|\mc{N}_i|-1$ neighbours. {\color{black} We consider both homogeneous and heterogeneous node scenarios: In the former, the nodes have similar computation ability (each node is a thread and we do not slow it down), and in the later we manually slow down node $1$ by letting it sleep twice the time required to compute  the local gradient at each iteration.}

    {\color{black}We consider both theoretical and hand-tuned step-sizes for all methods. For theoretical step-sizes, we set $\alpha=\min_{i\in\mc{V}} (w_{ii}/L_i)$ in the asynchronous and synchronous Prox-DGD and $\alpha=1/\max_{i\in\mc{V}} L_i$ in the asynchronous and synchronous DGD-ATC, which satisfy the conditions in Theorems \ref{thm:total}--\ref{thm:partial}. We also simulate the adaptive back-tracking step-size ($c=0.8$) below Theorem \ref{thm:adapt_step} and for both $\alpha=\min_{i\in\mc{V}} (w_{ii}/L_i)$ and hand-tuned $\alpha$. We fine-tune the parameters of asynchronous PG-EXTRA within their theoretical ranges. The theoretical ranges involve the maximum delay which is set to the maximum observed delay ($148$) during a 10-second run of the method. As illustrated above Assumption \ref{asm:totalasynchrony}, if a node $i$ updates at the $k$th iteration, then we call $k-s_{ij}^k$, $j\in\bar{\mc{N}}_i$ delay where $s_{ij}^k$ is defined below \eqref{eq:asyDGDupdateindex}. In the setting of hand-tuned step-sizes, we ignore the theoretical step-size ranges and tune them freely to achieve as fast as possible convergence.}  % and the maximal delay at the $k$th iteration is $k-\min_{j\in\bar{\mc{N}}_i} s_{ij}^k$.

\iffalse

{\color{black} Note that if a node $i$ updates at the $k$th iteration, then the maximal delay at the $k$th iteration is $k-\min_{j\in\bar{\mc{N}}_i} s_{ij}^k$ where $s_{ij}^k$ is defined below \eqref{eq:asyDGDupdateindex}. This maximum delay involves only the tuple $(k, i, \{s_{ij}^k\}_{j\in\mc{N}_i})$, which is determined through the following process: Each node $i$ records the tuple $(k_i, T_i^{k_i}, \{k_j\}_{j\in\bar{\mc{N}}_i})$, which means the $k_i$th local update of node $i$ occurs at the wall-clock time $T_i^{k_i}$ using $x_j$ with local iteration index $k_j$ from its neighbor $j$. To achieve this, whenever a node $i$ shares $x_i$ to its neighbors, it also shares the local iteration index $k_i$. With $(k_i, T_i^{k_i}, \{k_j\}_{j\in\bar{\mc{N}}_i})$, we can order the local updates by the wall-clock time they finished, and then form the tuple $(k,i,\{s_{ij}^k\}_{j\in\bar{\mc{N}}_i})$.}
\fi

\begin{figure*}
\centering
\includegraphics[width=0.8\linewidth]{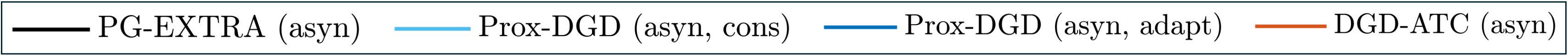}
\vspace{-0.4cm}

\end{figure*}

\begin{figure*}[!htb]
		\centering
		\subfigure[homo nodes, theo step]{\includegraphics[width=0.23\linewidth]{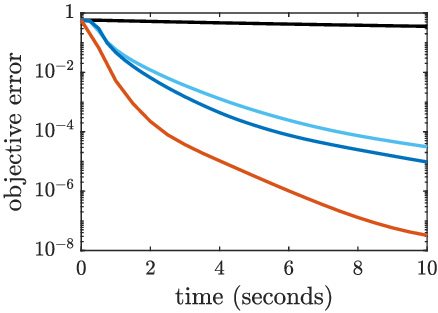}}
		\subfigure[homo nodes, tuned step]{
			\includegraphics[width=0.22\linewidth]{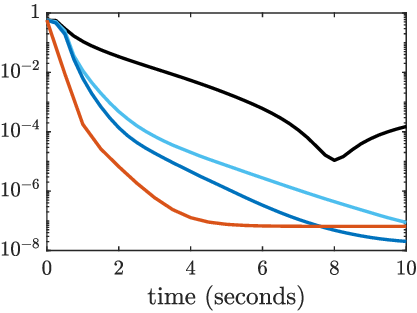}}
        \subfigure[heter nodes, theo step]{
			\includegraphics[width=0.22\linewidth]{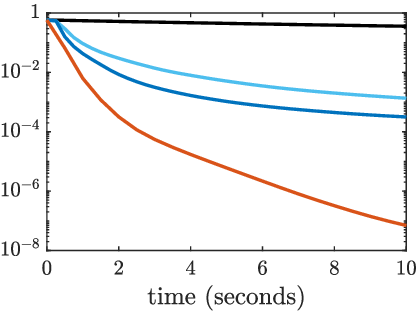}}
        \subfigure[heter nodes, tuned step]{\includegraphics[width=0.22\linewidth]{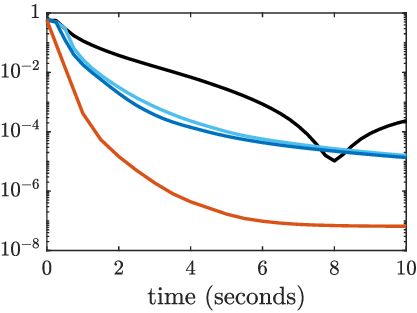}}
       \caption{Smooth problems ($\lambda_1=0$): comparison among asynchronous methods.}
		\label{fig:smooth}
	\end{figure*}

\begin{figure*}
\centering
\includegraphics[width=0.8\linewidth]{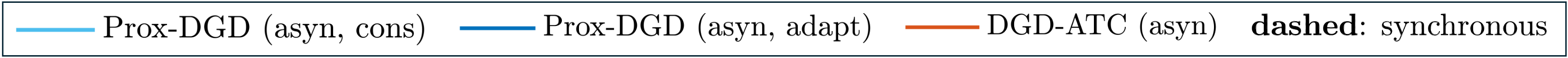}
\vspace{-0.4cm}

\end{figure*}

  \begin{figure*}[!htb]
		\centering
		\subfigure[run time, theo step]{\includegraphics[width=0.23\linewidth]{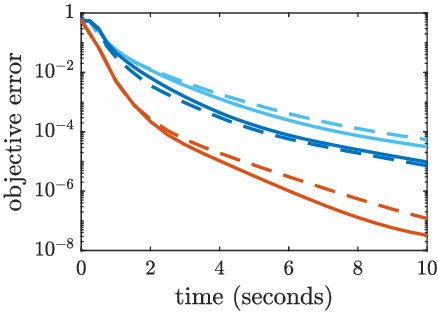}}
		\subfigure[communication, theo step]{
			\includegraphics[width=0.22\linewidth]{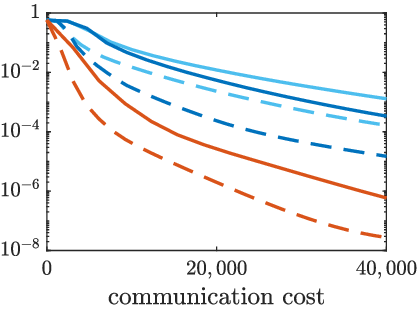}}
        \subfigure[run time, tuned step]{
			\includegraphics[width=0.22\linewidth]{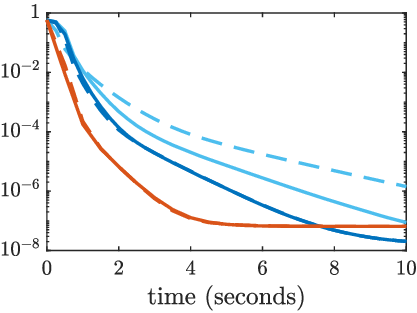}}
        \subfigure[communication, tuned step]{\includegraphics[width=0.22\linewidth]{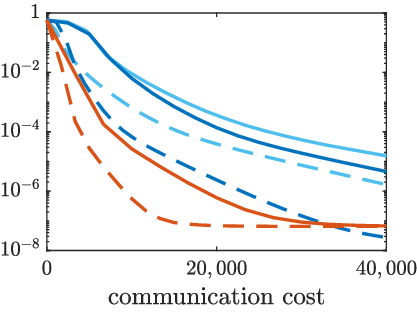}}
       \caption{Smooth problem ($\lambda_1=0$): comparison among asynchronous and synchronous Prox-DGD and DGD-ATC.}
		\label{fig:smooth-syn}
	\end{figure*}

\begin{figure*}
\centering
\includegraphics[width=0.65\linewidth]{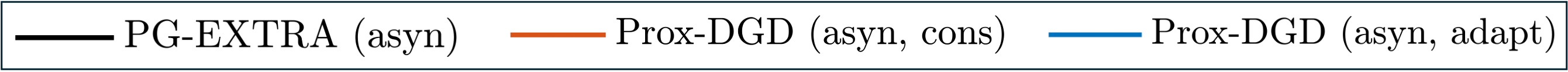}
\vspace{-0.4cm}

\end{figure*}

 	\begin{figure*}[!htb]
		\centering
		\subfigure[homo nodes, theo step]{\includegraphics[width=0.23\linewidth]{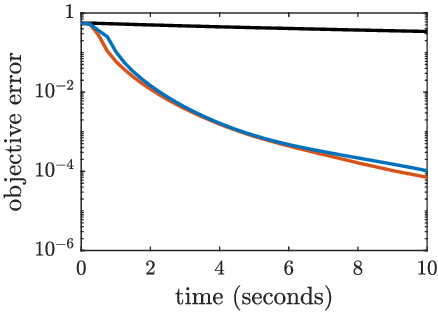}}
		\subfigure[homo nodes, tuned step]{
			\includegraphics[width=0.22\linewidth]{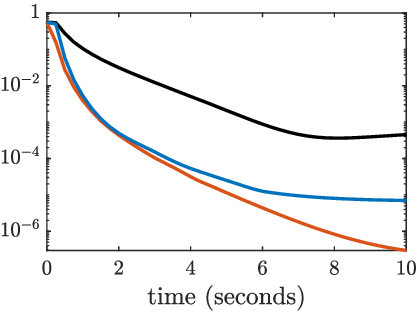}}
        \subfigure[heter nodes, theo step]{
			\includegraphics[width=0.22\linewidth]{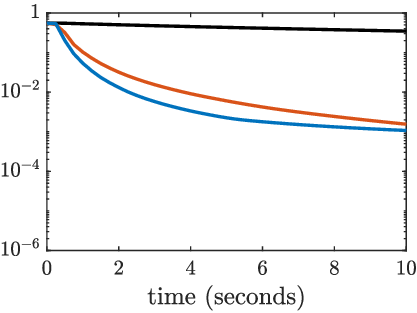}}
        \subfigure[heter nodes, tuned step]{\includegraphics[width=0.22\linewidth]{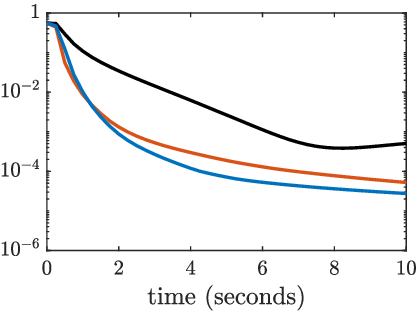}}
       \caption{Non-smooth problem ($\lambda_1=0$): comparison among asynchronous methods.}
		\label{fig:non-smooth}
	\end{figure*}

\begin{figure*}
\centering
\includegraphics[width=0.65\linewidth]{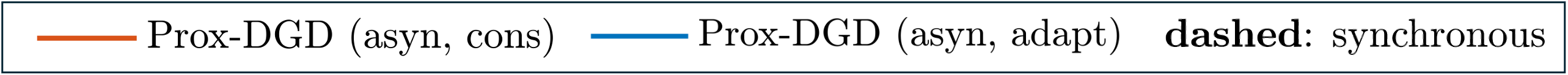}
\vspace{-0.4cm}

\end{figure*}

 	\begin{figure*}[!htb]
		\centering
		\subfigure[run time, theo step]{\includegraphics[width=0.23\linewidth]{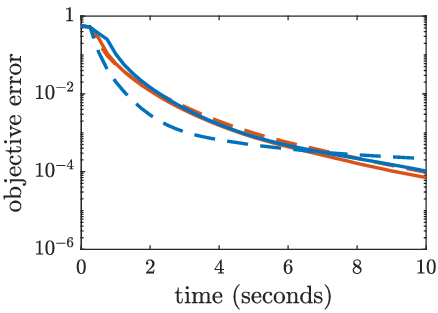}}
		\subfigure[communication, theo step]{
			\includegraphics[width=0.22\linewidth]{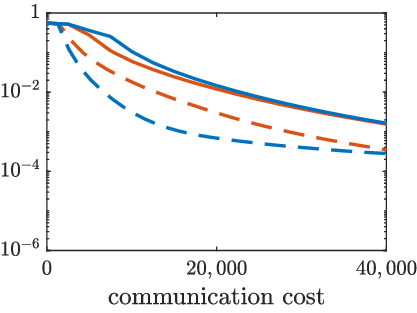}}
        \subfigure[run time, tuned step]{
			\includegraphics[width=0.22\linewidth]{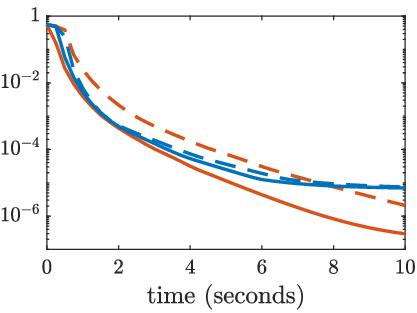}}
        \subfigure[communication, tuned step]{\includegraphics[width=0.22\linewidth]{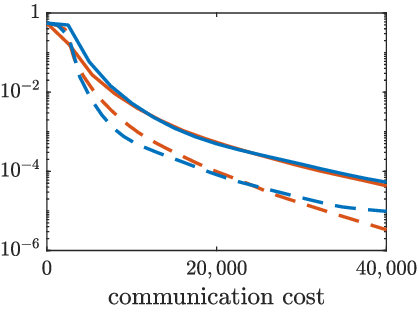}}
       \caption{Non-smooth problem ($\lambda_1>0$): comparison among asynchronous and synchronous Prox-DGD.}
		\label{fig:non-smooth-syn}
	\end{figure*}

	\begin{figure}
        \centering
        \includegraphics[scale=0.55]{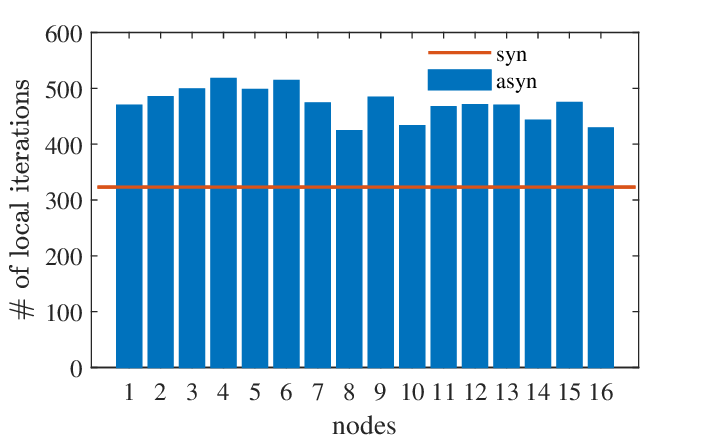}
        \caption{Number of local iterations in asynchronous and synchronous Prox-DGD (homogeneous nodes)}
        \label{fig:local_iters}
        \end{figure}

    We run all methods for $10$ seconds and plot the training error $f(\bar{x}(t))-f^\star$ at the average iterate $\bar{x}(t)=\frac{1}{n}\sum_{i=1}^n x_i(t)$ in Figs. \ref{fig:smooth}--\ref{fig:non-smooth-syn}, where $x_i(t)$ is the value of $x_i$ at time $t$ and $f^\star$ is the optimal value of \eqref{eq:consensusprob}. {\color{black}In these figures, ``syn", ``asy", ``cons", ``adapt", ``theo step", ``tuned step", ``homo nodes", and ``heter nodes" represent ``synchronous implementation", ``asynchronous implementation", ``constant step-size", ``adaptive step-size using back-tracking", ``theoretical step-size", ``hand-tuned step-size", ``homogeneous nodes", and ``heterogeneous nodes", respectively.}
    
    From these figures, we make the following observations. {\color{black}First, from Figures \ref{fig:smooth}--\ref{fig:smooth-syn}, the asynchronous DGD-ATC outperforms the asynchronous Prox-DGD when solving the smooth problem, which is in agreement with our theory, as discussed below Theorem \ref{thm:adaptATC}. Second, from Figures \ref{fig:smooth} and \ref{fig:non-smooth}, in all scenarios (homogeneous and heterogeneous nodes, smooth and non-smooth problems) the asynchronous Prox-DGD converges faster than the asynchronous PG-EXTRA. In the setting of theoretical step-sizes, it is because the parameters of asynchronous PG-EXTRA are conservative due to the large maximum delay, while the asynchronous Prox-DGD allows for much more relaxed delay-free parameters. In the setting of hand-tuned step-sizes, the best hand-tuned delay-dependent step-size of the asynchronous PG-EXTRA (two step-size parameters, one delay-free and one delay-dependent) is still not large. Specifically, in our experiments it is larger ($2 - 4$ times) than its theoretically maximum value but is still much smaller ($6 - 13$ times) than the theoretically maximum step-size of its synchronous counterpart, and simply increasing it leads to either slow convergence or divergence. Third, from Figures \ref{fig:smooth-syn} and \ref{fig:non-smooth-syn}, the asynchronous Prox-DGD and DGD-ATC are comparable with or converge faster than than their synchronous counterparts in terms of run time. This is a big improvement over existing asynchronous methods with fixed step-sizes since most of them are much slower than their synchronous counterparts in practical implementations, especially when using theoretically justified step-sizes. Fig. \ref{fig:local_iters} shows the number of local iterations per node over a 10-second run of both asynchronous and synchronous Prox-DGD. The results indicate that, on average, the asynchronous implementation performs approximately 1.46 times more local iterations than the synchronous implementation. In terms of communication cost, the asynchronous methods are slower which is also reasonable since the asynchronous updates include delayed information that slows down the convergence.} %Fourth, by Figures \ref{fig:smooth} and \ref{fig:non-smooth}, compared to the homogeneous setting, in the heterogeneous setting the asynchronous Prox-DGD significantly becomes slower while the asynchronous DGD-ATC method 
\section{Conclusion}
	
	We have investigated the asynchronous version of two distributed algorithms, Prox-DGD and DGD-ATC, for solving consensus optimization problems. We analysed the optimality gap of the synchronous Prox-DGD and DGD-ATC, and proved that their asynchronous variants can converge to the  fixed point set of the synchronous counterparts under \emph{delay-free} parameter conditions. %Our results cover both total and partial asynchrony and allow for both strongly and weakly convex objective functions.
    We also show that the convergence of both asynchronous algorithms  adjusts to the actual level of asynchrony and is not determined by the worst-case. Finally, we demonstrated superior practical convergence of the two asynchronous algorithms in implementations on a multi-core computer. {\color{black}Future works include extensions to stochastic optimization problems and directed networks \cite{xin2019distributed,qureshi2021s}.} %Future work includes developing asynchronous algorithms with delay-free parameter conditions for other distributed optimization problems.

	\appendix
	\subsection{Assumption \ref{asm:optsolexist} holds for quadratic $f_i$ and zero $h_i$}\label{append:qpoptsolexist}
	
	Let $A = \operatorname{diag}(A_1,\ldots,A_n)$ and $b=(b_1^T,\ldots,b_n^T)^T$. Because $f_i(x)=\|A_ix-b_i\|^2$ and $h_i\equiv 0$, we have $F(\bx) = \|A\bx-b\|^2$, so that
	\begin{equation*}
	F(\bx)+\frac{1}{2\alpha}\bx^T(I-\mb{W})\bx = \|A\bx-b\|^2+\frac{1}{2\alpha}\bx^T(I-\mb{W})\bx.
	\end{equation*}
	Minimizing the above function is equivalent to solving $(2A^TA+(I-\mb{W})/\alpha)\bx = 2A^Tb$, which has a solution if
 \begin{equation}\label{eq:Atbinrange}
     A^Tb \in\operatorname{Range}(2A^TA+(I-\mb{W})/\alpha).
 \end{equation}
Equation \eqref{eq:Atbinrange} indeed holds and, because $A^Tb\in\operatorname{Range}(A^TA)$, it can be proved by showing
\begin{equation}\label{eq:rangesubset}
    \operatorname{Range}(A^TA) \subseteq \operatorname{Range}(2A^TA+(I-\mb{W})/\alpha).
\end{equation}
Moreover, because the range of a real symmetric matrix is the orthogonal complement of its null space, \eqref{eq:rangesubset} is equivalent to
\begin{equation}\label{eq:nullsubset}
    \operatorname{Null}(2A^TA+(I-\mb{W})/\alpha)\subseteq \operatorname{Null}(A^TA).
\end{equation}
For any $\bx\in \operatorname{Null}(2A^TA+(I-\mb{W})/\alpha)$, because $I-\mb{W}\succeq \mb{0}$,
\begin{equation}
    \bx^TA^TA\bx\le \bx^T(2A^TA+(I-\mb{W})/\alpha)\bx=0,
\end{equation}
which, together with $A^TA\succeq \mb{0}$, yields $\bx\in \operatorname{Null}(A^TA)$ and also \eqref{eq:nullsubset}. Concluding all the above, the minimizer of $F(\bx)+\frac{1}{2\alpha}\bx^T(I-\mb{W})\bx$ exists.
	
	\subsection{Proof of Lemma \ref{lemma:DGDoptimalitygap}}\label{append:gapDGD}
{\color{black} Note that $\mc{G}$ is undirected and connected, and the matrix $W$ is symmetric, stochastic, and satisfies that $w_{ij}>0$ if $\{i,j\}\in\mc{E}$ or $i=j$, and $w_{ij}=0$ otherwise. Therefore, $W$ describes the transition probability of an irreducible and aperiodic Markov chain. Then, by \cite[Proposition 1, 2]{diaconis1991geometric}, $\beta\in [0,1)$.}
	
{\color{black}Next, we prove a property of $\mb{W}$. Because $\mc{G}$ is connected and $\mb{W}=W\otimes I_d$, we have
	\begin{align}
	\operatorname{Null}(I-\mb{W}) &= \{\bx\in\R^{nd}: x_1 = x_2 = \ldots = x_n\},\label{eq:nullIminusW}\\
    \operatorname{Range}(I-\mb{W}) &= \{\bx\in\R^{nd}: \sum_{i=1}^n x_i=\mb{0}\},\label{eq:rangeIminiusW}
	\end{align}
    Let $\bar{\bx}=\mb{1}_n\otimes \frac{1}{n}\sum_{i\in\mc{V}} x_i$ for any $\bx\in\R^{nd}$. By \eqref{eq:nullIminusW}--\eqref{eq:rangeIminiusW} and the definition of $\bar{\bx}$, we have
   \begin{equation}\label{eq:inrange}
   \bar{\bx}\in\operatorname{Null}(I-\mb{W}),\quad \bx-\bar{\bx}\in\operatorname{Range}(I-\mb{W}).
   \end{equation}
   Moreover, because $\lambda_1(W)=1$ and $\lambda_2(W)\le \beta<1$, we have
    \begin{equation}
        I-\mb{W}\succeq \mb{0},
    \end{equation}
    and the minimum non-zero eigenvalue of $I-\mb{W}$ is $1-\lambda_2(W)\ge 1-\beta>0$. 
   This, together with \eqref{eq:inrange}, yields} \begin{equation}\label{eq:IminusWQp}
	\begin{split}
	\bx^T(I-\mb{W})\bx =& (\bx - \bar{\bx})^T(I-\mb{W})(\bx - \bar{\bx})\allowdisplaybreaks\\
	\ge& (1-\beta)\|\bx - \bar{\bx}\|^2.
	\end{split}
    \end{equation}

	Finally, we use \eqref{eq:IminusWQp} to show \eqref{eq:errorDGD}. Because $\bx^\star$ is the optimum of \eqref{eq:penalprob}, we have
	\begin{equation}\label{eq:DGDminimization}
	\begin{split}
	F(\bx^\star)+\frac{1}{2\alpha}(\bx^\star)^T(I-\mb{W})\bx^\star\le F_{\operatorname{opt}},
	\end{split}
	\end{equation}
	which, together with $F(\bx^\star)\ge \min F(\bx)$ and \eqref{eq:IminusWQp}, yields that for any $i\in\mc{V}$, $\|x_i^\star-\bar{x}^\star\|\le \|\bx^\star-\mb{1}_n\otimes\bar{x}^\star\|\le \sqrt{\frac{(\bx^\star)^T(I-\mb{W})\bx^\star}{1-\beta}}\le \sqrt{\frac{2\alpha(F_{\operatorname{opt}}-\min F(\bx))}{1-\beta}}$. Therefore, \eqref{eq:errorDGD} holds.
	
 \subsection{Proof of Lemma \ref{lemma:stronggap}}\label{append:stronggap}
	
	We first consider the case of non-identical $h_i$'s.

	\subsubsection{\textbf{Non-identical $h_i$'s}} Because $\bx^\star$ is optimal to \eqref{eq:penalprob}, by the first-order optimality condition,
	\begin{equation}\label{eq:consensuserrorpartialgra}
	(I-\mb{W})\bx^\star/\alpha \in \partial F(\bx^\star).
	\end{equation}
	
	\emph{{Case i)}}: each $f_i+h_i$ is Lipschitz continuous. In this case, $F$ is Lipschitz and we denote its Lipschitz constant as $G>0$. Then, the norm of all subgradients of $F$ is bounded above by $G$, so that by \eqref{eq:consensuserrorpartialgra},
	\begin{equation}\label{eq:conserrorboundbyG}
	\|(I-\mb{W})\bx^\star\|/\alpha\le G.
	\end{equation}
	Moreover, because $\mb{1}_n$ is the eigenvector of $W$ corresponding to the unique maximal eigenvalue $1$, we have
	\begin{equation}
	\|W-\frac{\mb{1}_n\mb{1}_n^T}{n}\|\le \max (|\lambda_2(W)|, |\lambda_n(W)|)=\beta,
	\end{equation}
	which, together with $(\frac{\mb{1}_n\mb{1}_n^T}{n}\otimes I_d)\bx^\star=\bar{\bx}^\star$ and $\mb{W}=W\otimes I_d$, gives
	\begin{equation}\label{eq:Wxstar}
	\begin{split}
	\|\mb{W}(\bx^\star-\bar{\bx}^\star)\| &= \|(\mb{W}-\frac{\mb{1}_n\mb{1}_n^T}{n}\otimes I_d)(\bx^\star-\bar{\bx}^\star)\|\\
	&\le \|\mb{W}-\frac{\mb{1}_n\mb{1}_n^T}{n}\otimes I_d\|\cdot\|\bx^\star-\bar{\bx}^\star\|\\
	&\le \beta\|\bx^\star-\bar{\bx}^\star\|.
	\end{split}
	\end{equation}
	By \eqref{eq:Wxstar} and $\mb{W}\bar{\bx}^\star=\bar{\bx}^\star$,
	\begin{equation}\label{eq:IminusW}
	\begin{split}
	\|(I-\mb{W})\bx^\star\| &= \|(I-\mb{W})(\bx^\star-\bar{\bx}^\star)\|\\
	&\ge \|\bx^\star-\bar{\bx}^\star\|-\|\mb{W}(\bx^\star-\bar{\bx}^\star)\|\\
	&\ge (1-\beta)\|\bx^\star-\bar{\bx}^\star\|,
	\end{split}
	\end{equation}
	substituting \eqref{eq:IminusW} into \eqref{eq:conserrorboundbyG} yields \eqref{eq:betaDGD}.
	
	\emph{{Case ii)}}: $\operatorname{dom} h=\R^{nd}$ and $f_i+h_i$ is coercive. Because $\operatorname{dom} F=\R^{nd}$ and $F$ is coercive, $\mc{S}$ is compact {\color{black}(closed and bounded)} and, therefore,
   \begin{equation}
	\max_{\bx\in\mc{S}, \mb{v}\in\partial F(\bx)}\|\mb{v}\| <+\infty.
	\end{equation}
	By the above equation, \eqref{eq:consensuserrorpartialgra}, and \eqref{eq:IminusW}, we have
	\begin{equation}
	\|\bx^\star-\bar{\bx}^\star\|\le \frac{\alpha \max_{\bx\in\mc{S}, \mb{v}\in\partial F(\bx)}\|\mb{v}\|}{1-\beta},
	\end{equation}
	which yields \eqref{eq:betaDGD}.
	
	\subsubsection{\textbf{Identical $h_i$'s}} Define $\by^\star = \prox_{\alpha h}(\bar{\bx}^\star)$. Because
	\begin{equation}\label{eq:DGDfp}
	\bx^\star = \prox_{\alpha h}(\mb{W}\bx^\star-\alpha \nabla f(\bx^\star)),
	\end{equation}
	using the non-expansiveness of the proximal operator, we have
	\begin{equation}\label{eq:xstarystar}
	\begin{split}
	&\|\bx^\star-\by^\star\|\\
	= &\|\prox_{\alpha h}(\mb{W}\bx^\star-\alpha\nabla f(\bx^\star))-\prox_{\alpha h}(\bar{\bx}^\star)\|\\
	\le& \|\mb{W}\bx^\star-\alpha\nabla f(\bx^\star)-\bar{\bx}^\star\|\\
	\le& \|\mb{W}(\bx^\star-\bar{\bx}^\star)\|+\alpha\|\nabla f(\bx^\star)\|\\
        \le& \beta\|\bx^\star-\bar{\bx}^\star\|+\alpha\|\nabla f(\bx^\star)\|,
	\end{split}
	\end{equation}
	where the last two step use $\mb{W}\bar{\bx}^\star=\bar{\bx}^\star$ and \eqref{eq:Wxstar}, respectively. Because $h_i$'s are identical, so are all blocks of $\by^\star$, i.e., $\by^\star$ belongs to the consensus set $\{\bx: x_1=\ldots=x_n\}$. Moreover, $\bar{\bx}^\star$ is the projection of $\bx^\star$ onto the consensus set. Therefore,
	\begin{equation*}
	\|\bx^\star-\by^\star\| \ge \|\bx^\star-\bar{\bx}^\star\|.
	\end{equation*}
	By the above equation and \eqref{eq:xstarystar},
	\begin{equation}\label{eq:conserrorboundbygra}
	\|\bx^\star-\bar{\bx}^\star\|\le \frac{\alpha}{1-\beta}\|\nabla f(\bx^\star)\|.
	\end{equation}
	Moreover, because $f$ is $L$-smooth and because of \eqref{eq:conserrorboundbygra},
	\begin{equation}\label{eq:fbarandnonbar}
	\begin{split}
	f(\bar{\bx}^\star) - f(\bx^\star) &\le \langle\nabla f(\bx^\star), \bar{\bx}^\star-\bx^\star\rangle+\frac{L}{2}\|\bar{\bx}^\star-\bx^\star\|^2\\
 &\le \|\nabla f(\bx^\star)\|\cdot\|\bar{\bx}^\star-\bx^\star\|+\frac{L}{2}\|\bar{\bx}^\star-\bx^\star\|^2\\
	&\le \left(\frac{\alpha}{1-\beta}+\frac{L\alpha^2}{2(1-\beta)^2}\right)\|\nabla f(\bx^\star)\|^2.
	\end{split}
	\end{equation}
	Also, because all the $h_i$'s are convex and identical, we have
	\begin{equation}\label{eq:hbarx}
	\begin{split}
	h(\bar{\bx}^\star) &= \sum_{i\in\mc{V}} h_i(\frac{1}{n}\sum_{j\in\mc{V}} x_j^\star)\\
	&\le \sum_{i\in\mc{V}} \frac{1}{n}\sum_{j\in\mc{V}}h_i(x_j^\star)\\
	&= h(\bx^\star).
	\end{split}
	\end{equation}
	By \eqref{eq:fbarandnonbar}--\eqref{eq:hbarx} and $F(\bx^\star)\le F_{\operatorname{opt}}$ yield by \eqref{eq:DGDminimization}, we obtain
	\begin{equation}\label{eq:keystepFave}
	\begin{split}
	F(\bar{\bx}^\star) &\le F(\bx^\star)+\left(\frac{\alpha}{1-\beta}+\frac{L\alpha^2}{2(1-\beta)^2}\right)\|\nabla f(\bx^\star)\|^2\\
	&\le F_{\operatorname{opt}}+\left(\frac{\alpha}{1-\beta}+\frac{L\alpha^2}{2(1-\beta)^2}\right)\|\nabla f(\bx^\star)\|^2.
	\end{split}
	\end{equation}
	
	\emph{{Case i)}}: both $f_i$ and $h_i$ are bounded from the below. Let $\underline{f}$ and $\underline{h}$ be lower bounds of all $f_i$'s and all $h_i$'s, respectively. Because each $f_i$ is $L_i$-smooth,
	\begin{equation*}
	\begin{split}
	f_i(x_i^\star) &\ge f_i(x_i^\star-\frac{1}{L_i}\nabla f_i(x_i^\star))+\frac{1}{2L_i}\|\nabla f_i(x_i^\star)\|^2\\
	&\ge \underline{f}+\frac{1}{2L_i}\|\nabla f_i(x_i^\star)\|^2,
	\end{split}
	\end{equation*}
	so that
	\begin{equation}\label{eq:grabound}
	\|\nabla f(\bx^\star)\|^2 \le 2L(f(\bx^\star)-n\underline{f}).
	\end{equation}
	Moreover, by \eqref{eq:DGDminimization},
	\begin{equation}\label{eq:fstarbound}
	\begin{split}
	f(\bx^\star) &= F(\bx^\star) - h(\bx^\star)\le F_{\operatorname{opt}} - n\underline{h}.
	\end{split}
	\end{equation}
	Substituting \eqref{eq:fstarbound} into \eqref{eq:grabound} gives
	\begin{equation*}
	\|\nabla f(\bx^\star)\|\le \sqrt{2L(F_{\operatorname{opt}} - n(\underline{f}+\underline{h}))}.
	\end{equation*}
	Then by using \eqref{eq:conserrorboundbygra} and \eqref{eq:keystepFave}, we obtain \eqref{eq:betaDGD} and \eqref{eq:funcave}.

	\emph{Case ii)}: Lipschitz continuous $f_i$. In this case, $\nabla f(\bx)$ is bounded. Then by \eqref{eq:conserrorboundbygra} and \eqref{eq:keystepFave}, we have \eqref{eq:betaDGD} and \eqref{eq:funcave}.
	
	\emph{{Case iii)}}: each $f_i+h_i$ is coercive. In this case, $F$ is coercive, so that the following set is compact:
	\begin{equation}
	\mc{S} = \{\bx: F(\bx)\le F_{\operatorname{opt}}\}.
	\end{equation}
	Moreover, due to \eqref{eq:DGDminimization}, $\bx^\star\in\mc{S}$. Therefore,
	\begin{equation}
	\|\nabla f(\bx^\star)\| \le \max_{\bx\in\mc{S}} \|\nabla f(\bx)\|<+\infty.
	\end{equation}
	substituting which into by \eqref{eq:conserrorboundbygra} and \eqref{eq:keystepFave} yields \eqref{eq:betaDGD} and \eqref{eq:funcave}.	

 	\subsection{Proof of Theorem \ref{thm:total}}\label{append:thmtotal}
	The proof includes two steps. Step 1 rewrites the asynchronous Prox-DGD as an asynchronous fixed-point update in order to use an existing convergence theorem. Step 2 proves that the asynchronous fixed-point update satisfies the conditions in the convergence theorem, which guarantees the convergence of the asynchronous Prox-DGD.
	
	\textbf{Step 1: asynchronous fixed-point iteration.} Prox-DGD \eqref{eq:DGD} can be described by
	\begin{equation}\label{eq:fpu}
	\mathbf{x}^{k+1} = \T(\mathbf{x}^k),
	\end{equation}
	where
	\begin{align}
	\T(\bx)= \prox_{\alpha h}(\mathbf{W}\bx - \alpha \nabla f(\bx)),\label{eq:T}
	\end{align}
	Let $\T_i:\mathbb{R}^{nd}\rightarrow\mathbb{R}^d$ be the $i$th block of $\T$ for any $i\in \mc{V}$ and consider the asynchronous version of \eqref{eq:fpu}:
	\begin{equation}\label{eq:asyncop}
	x_i^{k+1} = \begin{cases}
	\T_i(\bz_i^k), & k\in\mc{K}_i,\\
	x_i^k, & \text{otherwise},
	\end{cases}
	\end{equation}
	where $\bz_i^k = (x_1^{t_{i1}^k}, \ldots, x_n^{t_{in}^k})$ for some non-negative integers $t_{ij}^k$. By letting
    \begin{equation}
	t_{ij}^k = 
	\begin{cases}
	s_{ij}^k, & j\in\bar{\mc{N}}_i,\\
	k, & \text{otherwise},
	\end{cases} \forall i\in\mc{V},~k\in\mc{K}_i.
    \end{equation}
    \eqref{eq:asyncop} with $\T$ in \eqref{eq:T} describes the asynchronous Prox-DGD.
	
	For the asynchronous update \eqref{eq:asyncop}, \cite{Feyzmahdavian23} presents the following convergence results for pseudo-contractive operator $\T$ defined as: for some $c\in(0,1)$,
	\begin{equation}\label{eq:pseudocontractive}
	\!\|\T(\bx)-\bx^\star\|_{\infty}^b\le c\|\bx-\bx^\star\|_{\infty}^b,\forall \bx\in\mathbb{R}^{nd}, \bx^\star\in \operatorname{Fix} \T,
    \end{equation}
    where $\operatorname{Fix} \T$ is the fixed point set of $\T$.
    \begin{lemma}[Theorem 23, \cite{Feyzmahdavian23}]\label{lemma:total}
		Suppose that Assumption \ref{asm:totalasynchrony} holds and
        \begin{align}\label{eq:initialfey}
	   0\in \mc{K}_i, \forall i\in\mc{V}.
	\end{align}
        If \eqref{eq:pseudocontractive} holds for some $c\in(0,1)$, then $\{\bx^k\}$ generated by the iteration \eqref{eq:asyncop} converges to the unique fixed point of $\T$.
	\end{lemma}
	Although Theorem 23 in \cite{Feyzmahdavian23} assumes \eqref{eq:initialfey} for simplicity, the convergence still holds without \eqref{eq:initialfey}. With Lemma \ref{lemma:total}, to show Theorem \ref{thm:total}, it suffices to show \eqref{eq:pseudocontractive} for $\T$ in \eqref{eq:T}.

	\textbf{Step 2: Proof of pseudo-contractivity \eqref{eq:pseudocontractive}}.  Because of \eqref{eq:DGDfp}, any optimum of problem \eqref{eq:penalprob} is a fixed point of $\T$ in \eqref{eq:T}. For simplicity of representation, we define
	\begin{align}
	y_i &= \sum_{j\in\bar{\mc{N}}_i} w_{ij}x_j-\alpha\nabla f_i(x_i),\label{eq:yidef}\\
	y_i^\star &= \sum_{j\in\bar{\mc{N}}_i} w_{ij}x_j^\star-\alpha\nabla f_i(x_i^\star).\label{eq:yistardef}
	\end{align}
	Then,
	\begin{align}
	\T_i(\bx) = \prox_{\alpha h_i}(y_i),~\T_i(\bx^\star) = \prox_{\alpha h_i}(y_i^\star).\label{eq:Tixandy}
	\end{align}
	For any $i\in\mc{V}$, by $x_i^\star=\T_i(\bx^\star)$ and the non-expansiveness of the proximal operator,
	\begin{equation}\label{eq:nonexpansiveATC}
	\begin{split}
	&\|\T_i(\bx)-x_i^\star\|^2 = \|\T_i(\bx)-\T_i(\bx^\star)\|^2\allowdisplaybreaks\\
	=&\|\prox_{\alpha h_i}(y_i)-\prox_{\alpha h_i}(y_i^\star)\|^2\le \|y_i-y_i^\star\|^2\allowdisplaybreaks\\
	=&\|\sum_{j\in\mc{N}_i} w_{ij}(x_j-x_j^\star) +\\
	&\quad w_{ii}(x_i-x_i^\star-\frac{\alpha}{w_{ii}}(\nabla f_i(x_i)-\nabla f_i(x_i^\star)))\|^2\allowdisplaybreaks\\
	\le&\sum_{j\in\mc{N}_i} w_{ij}\|x_j-x_j^\star\|^2+\\
	& w_{ii}\|x_i-x_i^\star-\frac{\alpha}{w_{ii}}(\nabla f_i(x_i)-\nabla f_i(x_i^\star))\|^2,
	\end{split}
	\end{equation}
	where the last step uses Jensen's inequality on the norm square. Since each $f_i$ is $\mu_i$-strongly convex and $\nabla f_i$ is Lipschitz continuous, by \cite[Equation (2.1.8)]{nesterov2003introductory},
	\begin{align}
	&\langle \nabla f_i(x_i)-\nabla f_i(x_i^\star), x_i-x_i^\star\rangle\ge\mu_i\|x_i-x_i^\star\|^2,\label{eq:proofstrongconvexity}\\
	&\langle \nabla f_i(x_i)-\nabla f_i(x_i^\star), x_i-x_i^\star\rangle\ge\frac{\|\nabla f_i(x_i)-\nabla f_i(x_i^\star)\|^2}{L_i}.\label{eq:proofsmooth}
	\end{align}
	Then,
	\begin{equation}\label{eq:shrinkDGD}
	\begin{split}
	&\|x_i-x_i^\star-\frac{\alpha}{w_{ii}}(\nabla f_i(x_i)-\nabla f_i(x_i^\star))\|^2\allowdisplaybreaks\\
	=&\|x_i-x_i^\star\|^2-2\frac{\alpha}{w_{ii}}\langle \nabla f_i(x_i)-\nabla f_i(x_i^\star), x_i-x_i^\star\rangle\allowdisplaybreaks\\
	&+\left(\frac{\alpha}{w_{ii}}\right)^2\|\nabla f_i(x_i)-\nabla f_i(x_i^\star)\|^2\allowdisplaybreaks\\
	\overset{\eqref{eq:proofsmooth}}{\le}& \|x_i-x_i^\star\|^2\!\!-\!\frac{\alpha}{w_{ii}}\left(2\!-\!\frac{L_i\alpha}{w_{ii}}\right)\langle \nabla f_i(x_i)\!-\!\nabla f_i(x_i^\star), x_i-x_i^\star\rangle\allowdisplaybreaks\\
	\overset{\eqref{eq:proofstrongconvexity}}{\le}&\left(1-\frac{\alpha}{w_{ii}}\left(2-\frac{L_i\alpha}{w_{ii}}\right)\mu_i\right)\|x_i-x_i^\star\|^2\allowdisplaybreaks.
	\end{split}
	\end{equation}
	Substituting \eqref{eq:shrinkDGD} into \eqref{eq:nonexpansiveATC} yields
	\begin{equation}\label{eq:contractDGD}
	\begin{split}
	\|\T_i(\bx)-x_i^\star\|^2\le \rho^2(\|\bx-\bx^\star\|_{\infty}^b)^2
	\end{split}
	\end{equation}
	with $\rho$ in \eqref{eq:rhodgd}, i.e., \eqref{eq:pseudocontractive} holds with $c=\rho$. Then by Lemma \ref{lemma:total}, the convergence of $\{\bx^k\}$ in Theorem \ref{thm:total} holds.
	%Then, by \eqref{eq:linearATC}, \eqref{eq:linearCAA}, Lemma \ref{lemma:total}, $\{\bx^k\}$ generated by \eqref{eq:asyncop} converges to $0$ under the total asynchrony model in Assumption \ref{asm:totalasynchrony}, which implies the convergence of the two asynchronous DGD. Note that 

	\subsection{Proof of Theorem \ref{thm:partial}}\label{append:thmpartial}
	
	\subsubsection{Asymptotic convergence}
	
	The proof uses Proposition 2.3 in \cite[Section 7.2]{bertsekas2015parallel}. Although it only considers the scalar case, i.e., $d=1$, its proof holds for the general case of $d>1$ by simply extending the absolute value to the Euclidean norm. The proposition is rewritten below for convenience.
	\begin{proposition}[Proposition 7.2.3,\cite{bertsekas2015parallel}]\label{proposition:copyBertsekas}
		Suppose that Assumption \ref{asm:partialasynchrony} holds. The sequence $\{\bx^k\}$ generated by \eqref{eq:asyncop} converges to some points in the fixed point set $X^\star$ of $\T$ if
		\begin{enumerate}[(a)]
			\item $X^\star$ is non-empty and closed.
			\item $\T$ is continuous and non-expansive in terms of the block-wise maximum norm: for all $\bx\in\mathbb{R}^{nd}$ and $\bx^\star\in X^\star$,
			\begin{equation}\label{eq:nonexpoperator}
			\|\T(\bx)-\bx^\star\|_{\infty}^b\le \|\bx-\bx^\star\|_{\infty}^b.
		  \end{equation}
		  \item For every $\bx\in\mathbb{R}^{nd}$ and $\bx^\star\in X^\star$ such that $\|\bx-\bx^\star\|_{\infty}^b=\inf_{\by\in X^\star} \|\bx-\by\|_{\infty}^b$, let
			\begin{align*}
			&\mc{I}(\bx,\bx^\star) = \{i|~\|x_i-x_i^\star\| = \|\bx-\bx^\star\|_{\infty}^b\}\\
			&\mc{U}(\bx, \bx^\star) = \{\by:~y_i=x_i\text{ for all }i\in\mc{I}(\bx,\bx^\star),\\
			& \text{and } \|y_i-x_i^\star\|<\|\bx-\bx^\star\|_{\infty}^b \text{ for all }i\notin \mc{I}(\bx,\bx^\star)\}.
			\end{align*}
			If $\|\bx-\bx^\star\|_{\infty}^b>0$, then there exists some $i\in\mc{I}(\bx,\bx^\star)$ such that $\T_i(\by)\ne y_i$ for all $\by\in\mc{U}(\bx,\bx^\star)$.
			\item For any $\bx\in\!\mbb{R}^{nd}$, $\bx^\star\!\in\!X^\star$, and $i\in\mc{V}$, if $\T_i(\bx)\ne x_i$, then
			\begin{equation}\label{eq:condd}
			\|\T_i(\bx)-x_i^\star\|<\|\bx-\bx^\star\|_{\infty}^b.
			\end{equation}
		\end{enumerate}
	\end{proposition}
	
	Next, we prove that all the conditions in Proposition \ref{proposition:copyBertsekas} hold for $\T$ in \eqref{eq:T}. Because the fixed point set of \eqref{eq:T} is the optimal solution set of \eqref{eq:penalprob} and, by Assumption \ref{asm:optsolexist}, the optimal set of problem \eqref{eq:penalprob} is closed and non-empty, so is the fixed point set of \eqref{eq:T}, i.e., condition (a) holds for $\T$ in \eqref{eq:T}.
	
	\textbf{Proof of condition (b)}: Since $\T$ in \eqref{eq:T} is continuous, to show condition (b), it suffices to show the non-expansiveness. For any $i\in\mc{V}$, by \eqref{eq:stepsizecond} and the first two steps of \eqref{eq:shrinkDGD},	\begin{equation}\label{eq:opATCconvexityandsmooth}
	\begin{split}
	&\|x_i-x_i^\star-\frac{\alpha}{w_{ii}}(\nabla f_i(x_i)-\nabla f_i(x_i^\star))\|^2\le\|x_i-x_i^\star\|^2,
	\end{split}
	\end{equation}
	substituting which into \eqref{eq:nonexpansiveATC} yields
	\begin{equation}\label{eq:Tixxstar}
	\begin{split}
	\!\!\|\T_i(\bx)-x_i^\star\|^2&\!\le \!\sum_{j\in\bar{\mc{N}}_i}\!w_{ij}\|x_j-x_j^\star\|^2\!\le\! (\|\bx-\bx^\star\|_{\infty}^b)^2.
	\end{split}
	\end{equation}
	Since \eqref{eq:Tixxstar} holds for all $i\in\mc{V}$, condition (b) holds.
	
	\textbf{Proof of condition (c).} We prove this condition by contradiction. Let $\bx$ and $\bx^\star$ be such that $\|\bx-\bx^\star\|_{\infty}^b=\inf_{\by\in X^\star} \|\bx-\by\|_{\infty}^b>0$. Suppose that for any $i\in \mc{I}(\bx,\bx^\star)$, there exists a vector $\by^i\in \mc{U}(\bx,\bx^\star)$ such that
	\begin{equation}\label{eq:yfix}
	\T_i(\by^i)=y_i^i.
	\end{equation}
	Below we will show that this contradicts the assumption $\|\bx-\bx^\star\|_{\infty}^b>0$. Fix $i\in \mc{I}(\bx,\bx^\star)$. By \eqref{eq:yfix} and \eqref{eq:T},
	\begin{equation}
	y_i^i = \T_i(\by^i) = \prox_{\alpha h_i}\biggl(\sum_{j\in\bar{\mc{N}}_i} w_{ij}y_j^i - \alpha\nabla f_i(y_i^i)\biggr).
	\end{equation}
	Moreover, by the third equality in \eqref{eq:nonexpansiveATC}, the convexity of norm, and \eqref{eq:opATCconvexityandsmooth}, $\|y_i^i-x_i^\star\|=\|\sum_{j\in\mc{N}_i} w_{ij}(y_j^i-x_j^\star) +w_{ii}(y_i^i-x_i^\star-\frac{\alpha}{w_{ii}}(\nabla f_i(y_i^i)-\nabla f_i(x_i^\star)))\|\le \sum_{j\in\bar{\mc{N}}_i} w_{ij}\|y_j^i-x_j^\star\|$,
	which, together with $i\in \mc{I}(\bx,\bx^\star)$ and $\by^i\in \mc{U}(\bx, \bx^\star)$, yields
	\begin{equation}\label{eq:bxbxstardis}
	\begin{split}
	&\|\bx-\bx^\star\|_{\infty}^b=\|x_i-x_i^\star\|=\|y_i^i-x_i^\star\|\\
	\le &\sum_{j\in\bar{\mc{N}}_i} w_{ij}\|y_j^i-x_i^\star\|.
	\end{split}
	\end{equation}
	Because $\|y_j^i-x_j^\star\|\le\|\bx-\bx^\star\|_{\infty}^b$ for all $j\in\bar{\mc{N}}_i$, $w_{ij}\in (0,1]$, and $\sum_{j\in\bar{\mc{N}}_i} w_{ij}=1$, equation \eqref{eq:bxbxstardis} implies
	\begin{equation*}
	\|y_j^i-x_j^\star\|=\|\bx-\bx^\star\|_{\infty}^b~\forall j\in\bar{\mc{N}}_i,
	\end{equation*}
	which, together with $\by^i\in \mc{U}(\bx,\bx^\star)$, indicates $\bar{\mc{N}}_i\subseteq \mc{I}(\bx,\bx^\star)$.

	Following this process for every $j\in\mc{N}_i$, we will have $\bar{\mc{N}}_j\subseteq \mc{I}(\bx,\bx^\star)$. By the connectivity of the graph $\mc{G}$ and following this process, we can derive
	\begin{equation*}
	\mc{I}(\bx,\bx^\star) = \mc{V},
	\end{equation*}
	and therefore, $\mc{U}(\bx,\bx^\star) = \{\bx\}$ which also indicates $\by^i=\bx$ $\forall i\in\mc{V}$. Then, by \eqref{eq:yfix}, $\T(\bx) = \bx$, i.e., $\bx\in X^\star$. Note that we require $\bx$ and $\bx^\star$ be such that $\|\bx-\bx^\star\|_{\infty}^b=\inf_{\by\in X^\star}\|\by-\bx^\star\|_{\infty}^b$. This, together with $\bx\in X^\star$, yields $\bx=\bx^\star$, which contradicts the assumption that $\|\bx-\bx^\star\|_{\infty}^b>0$.
	
	Concluding the above, there exists $i\in\mc{I}(\bx,\bx^\star)$ such that $\T_i(\by)\ne y_i$ for all $\by\in\mc{U}(\bx,\bx^\star)$, i.e., condition (c) holds.
	
	\textbf{Proof of condition (d).} To proceed, we first derive the following lemma.
	\begin{lemma}\label{lemma:equav}
		Let $\beta>0$ and $v_1,\ldots,v_m\in\mbb{R}^d$ be vectors satisfying $\|v_i\|\le\beta$ for all $1\le i\le m$. If there exists positive constants $\alpha_1,\ldots,\alpha_m$ such that $\sum_{i=1}^m \alpha_i = 1$ and \begin{equation}\label{eq:vinormeq}
		\|\sum_{i=1}^m \alpha_iv_i\| = \beta,
		\end{equation}
		then
		\begin{equation}\label{eq:vieq}
		v_1  = \ldots = v_m.
		\end{equation}
	\end{lemma}
	\begin{proof}
		Since $\|v_i\|\le\beta$ for all $1\le i\le m$ and $\beta=\|\sum_{i=1}^m \alpha_iv_i\|\le \sum_{i=1}^m \alpha_i \|v_i\|$, we have $\|v_i\|=\beta$ $\forall i\in\mc{V}$ and
		\begin{equation}\label{eq:suminoutnorm}
		\|\sum_{i=1}^m \alpha_iv_i\|= \sum_{i=1}^m \alpha_i \|v_i\|.
		\end{equation}
		Because all the $v_i$'s have the same norm, \eqref{eq:suminoutnorm} is possible only when all the $v_i$'s are identical, i.e., \eqref{eq:vieq} holds.
	\end{proof}
	
{\color{black}Next, based on Lemma \ref{lemma:equav}, we prove condition (d) by contradiction. We suppose that \eqref{eq:condd} does not hold and prove 
    \begin{equation}\label{eq:Tixieq}
        \T_i(\bx)=x_i,
    \end{equation}
     which contradicts the assumption $\T_i(\bx)\ne x_i$.
     
    Because \eqref{eq:condd} fails to hold and because of \eqref{eq:Tixxstar}, we have
	\begin{equation}\label{eq:tixeqxistar}
	\|\T_i(\bx)-x_i^\star\| = \|\bx-\bx^\star\|_{\infty}^b,
	\end{equation}
    i.e., both sides of \eqref{eq:Tixxstar} are identical. Note that the derivation of \eqref{eq:Tixxstar} uses \eqref{eq:nonexpansiveATC} and \eqref{eq:shrinkDGD} with $\mu_i=0$, and to make the two sides of \eqref{eq:Tixxstar} equal, all inequalities in \eqref{eq:nonexpansiveATC}, \eqref{eq:shrinkDGD} must hold with equality. Using this fact, we obtain several important equations. Specifically, letting both sides of the first inequality in \eqref{eq:nonexpansiveATC} equal yields    \begin{equation}\label{eq:proxeq}
	   \|\prox_{\alpha h_i}(y_i)-\prox_{\alpha h_i}(y_i^\star)\|^2= \|y_i-y_i^\star\|^2
    \end{equation}
    where $y_i,y_i^\star$ are introduced in \eqref{eq:yidef}--\eqref{eq:yistardef}, and letting both sides of the last inequality in \eqref{eq:shrinkDGD} equal gives
    \begin{equation}\label{eq:grasquanorm0}
        \frac{\alpha}{w_{ii}}\left(2\!-\!\frac{L_i\alpha}{w_{ii}}\right)\langle \nabla f_i(x_i)\!-\!\nabla f_i(x_i^\star), x_i-x_i^\star\rangle=0.
    \end{equation}
    By \eqref{eq:tixeqxistar}, \eqref{eq:proxeq}, \eqref{eq:Tixandy}, and \eqref{eq:yidef}--\eqref{eq:yistardef}, we have
    \begin{equation}\label{eq:convexcombinationeq}
         \begin{split}
            & \|\bx-\bx^\star\|_{\infty}^b = \|\sum_{j\in\mc{N}_i} w_{ij}(x_j-x_j^\star)+\\
	       &\quad w_{ii}(x_i-x_i^\star-\frac{\alpha}{w_{ii}}(\nabla f_i(x_i)-\nabla f_i(x_i^\star)))\|.
         \end{split}
	\end{equation}
% $\|x_j-x_j^\star\|\le \|\bx-\bx^\star\|_{\infty}^b$ $\forall j\in\bar{\mc{N}}_i$, $w_{ij}\in (0,1)$ $\forall j\in\bar{\mc{N}}_i$, and $\sum_{j\in\bar{\mc{N}}_i} w_{ij}=1$ 

Below, we first use \eqref{eq:proxeq} -- \eqref{eq:convexcombinationeq} to prove
\begin{align}
    & \T_i(\bx)-y_i = \T_i(\bx^\star)-y_i^\star = x_i^\star-y_i^\star,\label{eq:Tixyixistar}\\
    & \nabla f_i(x_i)=\nabla f_i(x_i^\star),\label{eq:identicalgra}\\
    & x_i-x_i^\star = x_j-x_j^\star,~\forall j\in\mc{N}_i,\label{eq:xixjidentical}
\end{align}
and then use \eqref{eq:Tixyixistar}--\eqref{eq:xixjidentical} to derive \eqref{eq:Tixieq}. According to \cite[Proposition 12.27]{bauschke2011convex},
	\begin{align*}
	&\|\prox_{\alpha h_i}(y_i)-\prox_{\alpha h_i}(y_i^\star)\|^2\\
	\le& \|y_i-y_i^\star\|^2-\|y_i-\prox_{\alpha h_i}(y_i)-(y_i^\star-\prox_{\alpha h_i}(y_i^\star))\|^2,
	\end{align*}
	by which and \eqref{eq:proxeq}, we have 
	\begin{equation}\label{eq:yiminusproxyi}
        y_i-\prox_{\alpha h_i}(y_i)=y_i^\star-\prox_{\alpha h_i}(y_i^\star).
    \end{equation}
	Moreover, $\prox_{\alpha h_i}(y_i)=\T_i(\bx)$ and $\prox_{\alpha h_i}(y_i^\star)=\T_i(\bx^\star)=x_i^\star$, which, together with \eqref{eq:yiminusproxyi}, yield \eqref{eq:Tixyixistar}. By \eqref{eq:stepsizecond}, \eqref{eq:grasquanorm0}, and \eqref{eq:proofsmooth}, we have \eqref{eq:identicalgra}. To prove \eqref{eq:xixjidentical}, we use Lemma \ref{lemma:equav} with $\beta=\|\bx-\bx^\star\|_{\infty}^b$, $m=|\bar{\mc{N}}_i|$, and $v_1,\ldots,v_m$ being $x_j-x_j^\star$, $j\in\mc{N}_i$ and $x_i-x_i^\star-\frac{\alpha}{w_{ii}}(\nabla f_i(x_i)-\nabla f_i(x_i^\star))$. By \eqref{eq:opATCconvexityandsmooth}, for any $1\le \ell\le m$, $\|v_{\ell}\|\le \beta$, which, together with \eqref{eq:convexcombinationeq}, implies that all the conditions in Lemma \ref{lemma:equav} hold and therefore,
	\begin{equation*}%\label{eq:xjminusxjstargra}
	x_j-x_j^\star = x_i-x_i^\star - \frac{\alpha}{w_{ii}}(\nabla f_i(x_i)-\nabla f_i(x_i^\star)),~\forall j\in\mc{N}_i.
	\end{equation*}
	By the above equation and \eqref{eq:identicalgra}, we obtain \eqref{eq:xixjidentical}. Using \eqref{eq:Tixyixistar}--\eqref{eq:xixjidentical}, we have
	\begin{equation*}
	\begin{split}
	\!\T_i(\bx) &\overset{\eqref{eq:Tixyixistar}}{=} y_i+x_i^\star-y_i^\star\allowdisplaybreaks\\
	&\overset{\eqref{eq:yidef}}{=}\sum_{j\in\bar{\mc{N}}_i} w_{ij}x_j-\alpha\nabla f_i(x_i)+x_i^\star-y_i^\star\allowdisplaybreaks\\
	&\overset{\eqref{eq:xixjidentical}}{=} \!\sum_{j\in\bar{\mc{N}}_i} w_{ij}(x_j^\star\!+x_i\!-x_i^\star)-\!\alpha\nabla f_i(x_i)+\!x_i^\star-y_i^\star\allowdisplaybreaks\\
	&\overset{\eqref{eq:identicalgra}}{=}  x_i+\sum_{j\in\bar{\mc{N}}_i}w_{ij}x_j^\star- \alpha\nabla f_i(x_i^\star) - y_i^\star \overset{\eqref{eq:yistardef}}{=} x_i,
	\end{split}
	\end{equation*}
	i.e., \eqref{eq:Tixieq} holds.
 
    Note that \eqref{eq:Tixieq} contradicts the assumption that $\T_i(\bx)\ne x_i$. Concluding the above, condition (d) holds.}
	\subsubsection{Linear convergence}
	The proof uses \cite[Theorem 24]{Feyzmahdavian23}.
	\begin{lemma}[Theorem 24, \cite{Feyzmahdavian23}]\label{lemma:partial}
		Suppose that Assumption \ref{asm:partialasynchrony} and \eqref{eq:initialfey} hold. If \eqref{eq:pseudocontractive} holds for some $c\in(0,1)$, then $\{\bx^k\}$ generated by the asynchronous iteration \eqref{eq:asyncop} satisfy
		\begin{equation}\label{eq:linearfey}
		\|\bx^k-\bx^\star\|_{\infty}^b\le c^{\frac{k}{B+D+1}}\|\bx^0-\bx^\star\|_{\infty}^b.
		\end{equation}
	\end{lemma}
	Note that without the condition \eqref{eq:initialfey}, the proof of \cite[Theorem 24]{Feyzmahdavian23} still holds, with the convergence rate \eqref{eq:linearfey} becomes
	\begin{equation}\label{eq:nokiincludeallzero}
	\|\bx^k-\bx^\star\|_{\infty}^b\le c^{\lfloor\frac{k}{B+D+1}\rfloor}\|\bx^0-\bx^\star\|_{\infty}^b.
	\end{equation}
    Moreover, according to \eqref{eq:contractDGD}, \eqref{eq:pseudocontractive} holds with $c=\rho$. Then, by Lemma \ref{lemma:partial}, we obtain \eqref{eq:linear_convergence}.

    \subsection{Proof of Theorem \ref{thm:adapt_step}}\label{append:proof_thm_adapt_step}
    {\color{black}
    Similar to Appendix \ref{append:thmpartial}, the synchronous version of \eqref{eq:adapt_DGDxiupdate} can be described by $x_i^{k+1}=\T_i(\bx^k)$ where
    \begin{equation}\label{eq:Ti_adapt}
        \T_i(\bx) = \prox_{\alpha h_i}(x_i+\gamma_i(\sum_{j\in\bar{\mc{N}}_i}w_{ij}x_j-x_i-\alpha\nabla f_i(x_i)).
    \end{equation}
    In this proof, we ignore the iteration indexes for simplicity. Denote $x_i^+=\T_i(\bx)$. The key of the proof is to show
    \begin{equation}\label{eq:key_adapt}
        \|x_i^+-x_i^\star\|^2\le \rho^2(\|\bx-\bx^\star\|^b_{\infty})^2
    \end{equation}
    for $\rho=\sqrt{\frac{1}{1+\alpha\min_{i\in\mc{V}} (\gamma_i^{\min}\mu_i)}}$. Define
    \[\tilde{w}_{ii} = 1+\gamma_i(w_{ii}-1), \quad \tilde{w}_{ij} = \gamma_iw_{ij},\quad \forall j\in\mc{N}_i\]
    and $\tilde{w}_{ij}=0$ otherwise. Clearly, $\sum_{j\in\mc{V}}\tilde{w}_{ij} = 1$ and, due to \eqref{eq:adapt_step}, we have $\gamma_i< \frac{1}{1-w_{ii}}$ and, therefore, $\tilde{w}_{ij}\in (0,1)$ $\forall j\in\bar{\mc{N}}_i$. Then, \eqref{eq:Ti_adapt} is equivalent to
    \begin{equation}\label{eq:fp_adapt}
       x_i^+= \T_i(\bx)=\prox_{\alpha h_i}\biggl(\sum_{j\in\bar{\mc{N}}_i} \tilde{w}_{ij}x_j  - \gamma_i\alpha \nabla f_i(x_i)\biggr).
    \end{equation}
    Note that \eqref{eq:fp_adapt} and Prox-DGD \eqref{eq:DGD} share the same fixed point set. Therefore, for the update \eqref{eq:fp_adapt}, there exist $g_i^+\in \partial h_i(x_i^+)$ and $g_i^\star\in\partial h_i(x_i^\star)$ such that
    \begin{align*}
        x_i^+ &= \sum_{j\in\bar{\mc{N}}_i} \tilde{w}_{ij}x_j - \gamma_i\alpha (\nabla f_i(x_i)+g_i^+),\\
        x_i^\star &= \sum_{j\in\bar{\mc{N}}_i} \tilde{w}_{ij}x_j^\star - \gamma_i\alpha (\nabla f_i(x_i^\star)+g_i^\star).
    \end{align*}
    By the above two equations,,
	\begin{equation}\label{eq:adapt_expansion}
	\begin{split}
	&\|x_i^+-x_i^\star\|^2 = \langle \sum_{j\in\bar{\mc{N}}_i} \tilde{w}_{ij}(x_j-x_j^\star), x_i^+-x_i^\star\rangle\\
	&-\gamma_i\alpha\langle \nabla f_i(x_i)+g_i^+-\nabla f_i(x_i^\star)-g_i^\star, x_i^+-x_i^\star\rangle.
	\end{split}
	\end{equation}
    Note that \eqref{eq:adapt_step} implies
    \[\begin{split}
        &\langle \nabla f_i(x_i^+) - \nabla f_i(x_i), x_i^+-x_i\rangle\\
        \le& \frac{1+\gamma_i(w_{ii}-1)}{2\gamma_i\alpha}\|x_i^+-x_i\|^2\\
        =& \frac{\tilde{w}_{ii}}{2\gamma_i\alpha}\|x_i^+-x_i\|^2
    \end{split},\]
    by which and the convexity of $f_i$,
    \[\begin{split}
        & \langle \nabla f_i(x_i), x_i^+-x_i\rangle\\
    \ge & f_i(x_i^+) -f_i(x_i)-\frac{\tilde{w}_{ii}}{2\gamma_i\alpha}\|x_i^+-x_i\|^2.
    \end{split}\]
    By the above equation and the convexity of $f_i$, 
    \begin{equation}\label{eq:adapt_conv_smooth}
        \begin{split}
            &\langle \nabla f_i(x_i), x_i^+-x_i^\star\rangle\\
            =&\langle \nabla f_i(x_i), x_i^+-x_i\rangle+\langle \nabla f_i(x_i), x_i-x_i^\star\rangle\\
            \ge& f_i(x_i^+)- f_i(x_i^\star)-\frac{\tilde{w}_{ii}}{2\gamma_i\alpha}\|x_i^+-x_i\|^2.
        \end{split}
    \end{equation}
By \eqref{eq:adapt_conv_smooth}, the convexity of $h_i$, and the $\mu_i$-strong convexity of $f_i+h_i$, we have
    \begin{equation*}
        \begin{split}
            &\langle \nabla f_i(x_i)+g_i^+-\nabla f_i(x_i^\star)-g_i^\star, x_i^+-x_i^\star\rangle\\
            \ge& F_i(x_i^+) -F_i(x_i^\star)-\langle\nabla f_i(x_i^\star)+g_i^\star, x_i^+-x_i^\star\rangle\\
            &-\!\frac{\tilde{w}_{ii}}{2\gamma_i\alpha}\|x_i^+-x_i\|^2\\
            \ge& \frac{\mu_i}{2}\|x_i^+-x_i^\star\|^2-\!\frac{\tilde{w}_{ii}}{2\gamma_i\alpha}\|x_i^+-x_i\|^2
        \end{split}
    \end{equation*}
    Moreover,
    \[\begin{split}
        &\langle \sum_{j\in\bar{\mc{N}}_i} \tilde{w}_{ij}(x_j-x_j^\star), x_i^+-x_i^\star\rangle\\
        =& \sum_{j\in\bar{\mc{N}}_i} \tilde{w}_{ij}\langle x_j-x_j^\star, x_i^+-x_i^\star\rangle\\
        \le&\frac{1}{2}\sum_{j\in\mc{N}_i} \tilde{w}_{ij}(\|x_j-x_j^\star\|^2 + \|x_i^+-x_i^\star\|^2)\\
        &+ \frac{\tilde{w}_{ii}}{2}(\|x_i^+-x_i^\star\|^2+\|x_i-x_i^\star\|^2-\|x_i^+-x_i\|^2)\\
        \le& \frac{1}{2}((\|\bx-\bx^\star\|_{\infty}^b)^2+\|x_i^+-x_i^\star\|^2-\tilde{w}_{ii}\|x_i^+-x_i\|^2).
    \end{split}\]
    Substituting the above two equations into \eqref{eq:adapt_expansion} yields
    \[\begin{split}
        &\|x_i^+-x_i^\star\|^2 \le (\|\bx-\bx^\star\|_{\infty}^b)^2-\mu_i\gamma_i\alpha\|x_i^+-x_i^\star\|^2,
    \end{split}\]
    which, together with $\gamma_i\ge \gamma_i^{\min}$, leads to \eqref{eq:key_adapt} with $\rho=\sqrt{\frac{1}{1+\alpha\min_{i\in\mc{V}} (\gamma_i^{\min}\mu_i)}}$.  With \eqref{eq:key_adapt} in hand and following the procedure of proving Theorem \ref{thm:total} and \eqref{eq:linear_convergence}, we obtain the results in Theorem \ref{thm:adapt_step}.
}
    
    \subsection{Proof of Theorem \ref{thm:adapt}}\label{append:thmadapt}
    By \eqref{eq:contractDGD} and the proof of Theorem 24 in \cite{Feyzmahdavian23},
    \begin{equation*}
	\|\mathbf{x}^{k+1}-\mathbf{x}^\star\|_{\infty}^b\le        \rho\max_{k-\tau^k\le t\le k}\|\mathbf{x}^t-\mathbf{x}^\star\|_{\infty}^b,
    \end{equation*}
    which, together with the definition of $\{k^m\}$, yields the result.
    
    \subsection{Proof of Lemma \ref{lemma:DGDATCoptimalitygap}}\label{append:DGDATCoptimalitygap}

    DGD-ATC \eqref{eq:DGD-ATC} can be viewed as a weighted gradient method \cite[Section 1.2.1]{bertsekas1995nonlinear} with weight matrix $\mb{W}$ for solving
	\begin{equation}\label{eq:problemATC}
	\underset{\bx\in\mathbb{R}^{nd}}{\operatorname{minimize}}~f(\bx)+\frac{\bx^T(\mb{W}^{-1}-I)\bx}{2\alpha}.
	\end{equation}
    Therefore, the existence of fixed point to DGD-ATC can be guaranteed by the non-emptiness of the optimal solution set of problem \eqref{eq:problemATC}.
	
	\subsubsection{Non-empty solution set of \eqref{eq:problemATC}}
	
	Let $z^\star$ be an optimum to \eqref{eq:consensusprob}. Define $\bz^\star=\mb{1}_n\otimes z^\star$, $\mc{L}_{\alpha}(\bx) = F(\bx)+\frac{\|\bx\|_{\mb{W}^{-1}-I}^2}{2\alpha}$, and
	\begin{equation*}
	\mc{C} = \{\bx: \mc{L}_{\alpha}(\bx)\le \mc{L}_{\alpha}(\bz^\star)\}.
	\end{equation*}
	Since $\min_{\bx\in\mbb{R}^{nd}} \mc{L}_{\alpha}(\bx)$ is equivalent to $\min_{\bx\in\mc{C}} \mc{L}_{\alpha}(\bx)$, the minimum of $\mc{L}_{\alpha}$ exists if the optimum of the later problem exists which holds true if $\mc{C}$ is non-empty and compact. Clearly, $\mc{C}$ is non-empty since $\bz^\star\in \mc{C}$.
	
	Below, we prove that $\mc{C}$ is compact. Let $\bx$ be an arbitrary vector in $\mc{C}$ and denote $\bar{\bx} = \mb{1}_n\otimes \frac{1}{n}\sum_{i\in\mc{V}} x_i$. By the $L$-smoothness of $f$,
	\begin{equation}\label{eq:smoothness}
	\begin{split}
	f(\bar{\bx}) - f(\bx) &\le \langle \nabla f(\bx), \bar{\bx}-\bx\rangle+\frac{L}{2}\|\bar{\bx}-\bx\|^2\\
	&\le \|\nabla f(\bx)\|\cdot\|\bar{\bx}-\bx\|+\frac{L}{2}\|\bar{\bx}-\bx\|^2.
	\end{split}
	\end{equation}
	By Assumption \ref{asm:prob} and $h_i\equiv 0$, each $f_i$ is bounded from the below, and we let $\underline{f}$ be a lower bound for all $f_i$'s. Then, by the $L_i$-smoothness of $f_i$,
	\begin{equation*}
	\underline{f}\le f_i(x_i-\frac{1}{L_i}\nabla f_i(x_i))\le f_i(x_i) - \frac{1}{2L_i}\|\nabla f_i(x_i)\|^2,
	\end{equation*}
	so that
	\begin{equation}\label{eq:boundedgra}
	\|\nabla f(\bx)\|^2 \le 2L(f(\bx) - n\underline{f}) \le 2L(f^\star - n\underline{f})
	\end{equation}
	where the last step is due to $f(\bx)\le \mc{L}_{\alpha}(\bx)\le \mc{L}_{\alpha}(\bz^\star)=f^\star$. In addition, because $f(\bx)\ge n\underline{f}$ and $\mc{L}_{\alpha}(\bx)\le f^\star$, we have
	\begin{equation}\label{eq:boundediff}
	\|\bx\|_{\mb{W}^{-1}-I}^2 = 2\alpha(\mc{L}_{\alpha}(\bx)-f(\bx))\le 2\alpha(f^\star-n\underline{f}).
	\end{equation}
	Since $\mc{G}$ is connected, we have $\operatorname{Range}(\mb{W}^{-1}-I)=\{\bx:x_1+\ldots+x_n=\mb{0}\}$, so that $\bx-\bar{\bx}$ is the projection of $\bx$ onto $\operatorname{Range}(\mb{W}^{-1}-I)$. Moreover, because $W\succ \mb{0}$, we have $\beta=\lambda_2(W)$, so that the minimum non-zero eigenvalue of $\mb{W}^{-1}-I$ is $\frac{1}{\lambda_2(W)}-1=\frac{1}{\beta}-1=\frac{1-\beta}{\beta}$. Therefore,
	\begin{equation*}
	\|\bx-\bar{\bx}\|^2\le \frac{\beta\|\bx\|_{\mb{W}^{-1}-I}^2}{1-\beta},
	\end{equation*}
	which, together with \eqref{eq:boundediff} and $\beta\le 1$, yields
	\begin{equation}\label{eq:gapxbarx}
	\|\bx-\bar{\bx}\|^2\le \frac{2\alpha(f^\star-n\underline{f})}{1-\beta}.
	\end{equation}
	Substituting \eqref{eq:boundedgra} and \eqref{eq:gapxbarx} into \eqref{eq:smoothness} and using $f(\bx)\le f^\star$ gives
	\begin{equation*}
	f(\bar{\bx})\le C_0:=f^\star+\left(\sqrt{\frac{\alpha L}{1-\beta}}+\frac{\alpha L}{1-\beta}\right)(f^\star-h),
	\end{equation*}
	or equivalently, $\sum_{i\in\mc{V}} f_i(\frac{1}{n}\sum_{i\in\mc{V}} x_i) \le C_0$. Moreover, by \cite[Proposition B.9]{bertsekas1995nonlinear} and the bounded optimum set of problem \eqref{eq:consensusprob}, every level set of $\sum_{i\in\mc{V}} f_i(x)$ is bounded and, therefore, the following set is compact:
	\begin{align}\label{eq:levelset}
	\bigg\{y\in\mbb{R}^d: \sum_{i\in\mc{V}} f_i(y)\le C_0\bigg\}.
	\end{align}
	Due to the arbitrariness of $\bx\in \mc{C}$, we have that for any $\bx\in\mc{C}$, $\frac{1}{n}\sum_{i=1}^n x_i$ belongs to the compact set \eqref{eq:levelset} and \eqref{eq:gapxbarx} holds. Therefore, $\mc{C}$ is compact.
	
	Concluding all the above, the optimal solution set of problem \eqref{eq:problemATC} is non-empty.

	\subsubsection{Optimality gap}
	Because $\bx^\star$ is a fixed point of \eqref{eq:DGD-ATC},
	\begin{equation}
	\bx^\star = \mb{W}(\bx^\star - \alpha\nabla f(\bx^\star)),
	\end{equation}
	which, together with $I \succeq\mb{W}\succ\mb{0}$, yields
	\begin{equation}\label{eq:IminusWboundedbygra}
	\begin{split}
	\|(I-\mb{W})\bx^\star\| = \alpha\|\mb{W}\nabla f(\bx^\star)\|\le& \alpha\|\nabla f(\bx^\star)\|.
	\end{split}
	\end{equation}
    Moreover, because $\bx^\star\in\mc{C}$, \eqref{eq:boundedgra} with $\bx=\bx^\star$ holds, i.e.,
    \begin{equation}\label{eq:boundedgraATC}
        \|\nabla f(\bx^\star)\|^2\le 2L(f^\star - n\underline{f}).
    \end{equation}
\iffalse    Also note that
    \begin{equation}\label{eq:IminusW}
	\begin{split}
	\|(I-\mb{W})\bx^\star\| &= \|(I-\mb{W})(\bx^\star-\bar{\bx}^\star)\|\\
	&\ge \|\bx^\star-\bar{\bx}^\star\|-\|\mb{W}(\bx^\star-\bar{\bx}^\star)\|\\
	&\ge (1-\beta)\|\bx^\star-\bar{\bx}^\star\|,
	\end{split}
	\end{equation}\fi
    By \eqref{eq:IminusW}, \eqref{eq:IminusWboundedbygra}, and \eqref{eq:boundedgraATC}, we obtain
    \begin{equation}\label{eq:consensusATC}
	\|\bx^\star-\bar{\bx}^\star\|\le \frac{\alpha}{1-\beta}\sqrt{2L(f^\star - n\underline{f})},
	\end{equation}
    i.e., \eqref{eq:betaDGD} holds.
	
    Next, we prove \eqref{eq:funcave}. Since $\bz^\star$ is the minimum of $\mc{L}_{\alpha}$,
	\begin{equation}\label{eq:fxstarfstar}
	f(\bx^\star) \le \mc{L}_{\alpha}(\bx^\star)\le \mc{L}_{\alpha}(\bz^\star) = f^\star.
	\end{equation}
    By letting $\bx=\bx^\star$ in \eqref{eq:smoothness} and using \eqref{eq:boundedgraATC}--\eqref{eq:consensusATC}, we obtain
    \begin{equation*}
	\begin{split}
	f(\bar{\bx}^\star) \le& f(\bx^\star) + \left(\frac{\alpha}{1-\beta}+\frac{L\alpha^2}{2(1-\beta)^2}\right)\cdot 2L(f^\star - n\underline{f}),
	\end{split}
    \end{equation*}
    which, together with \eqref{eq:fxstarfstar}, yields \eqref{eq:funcave} and completes the proof.
	
	\subsection{Proof of Theorem \ref{thm:totalATC}}\label{append:thmtotalATC}
	The proof is similar to that of Theorem \ref{thm:total}. We first rewrite the update \eqref{eq:DGD-ATC} in an operator form and prove the pseudo-contractivity \eqref{eq:pseudocontractive}, and then use Lemma \ref{lemma:total} to show convergence.
	
	DGD-ATC \eqref{eq:DGD-ATC} can be described by the general fixed-point update \eqref{eq:fpu},
	where
	\begin{align}
	\T(\bx)= \mathbf{W}(\bx - \alpha \nabla f(\bx)).\label{eq:TATC}
	\end{align}
	Let $\T_i:\mathbb{R}^{nd}\rightarrow\mathbb{R}^d$ be the $i$th block of $\T$. The update \eqref{eq:asyncop} with $\T$ in \eqref{eq:TATC} describes the asynchronous DGD-ATC.
	
	Below, we show \eqref{eq:pseudocontractive}. For any $i\in\mc{V}$, since $x_i^\star=\T_i(\bx^\star)$,
	\begin{equation}\label{eq:nonexpansiveCAA}
	\begin{split}
	&\|\T_i(\bx)-x_i^\star\|^2 = \|\T_i(\bx)-\T_i(\bx^\star)\|^2\\
	=&\|\sum_{j\in\bar{\mc{N}}_i} w_{ij}(x_j-x_j^\star-\alpha(\nabla f_j(x_j)-\nabla f_j(x_j^\star)))\|^2\\
	\le&\sum_{j\in\bar{\mc{N}}_i} w_{ij}\|x_j-x_j^\star-\alpha(\nabla f_j(x_j)-\nabla f_j(x_j^\star))\|^2,
	\end{split}
	\end{equation}
	where the last step uses Jensen's inequality on the norm square. Similar to \eqref{eq:shrinkDGD} with $w_{ii}=1$,
	\begin{equation*}
	\begin{split}
	&\|x_j-x_j^\star-\alpha(\nabla f_j(x_j)-\nabla f_j(x_j^\star))\|^2\\
	\le&(1-\alpha(2-L_j\alpha))\|x_j-x_j^\star\|^2\\
	\le& \hat{\rho}^2(\|\bx-\bx^\star\|_{\infty}^b)^2.
	\end{split}
	\end{equation*}
	Substituting the above equation into \eqref{eq:nonexpansiveCAA} yields
	\begin{equation}\label{eq:linearATC}
	\|\T_i(\bx)-x_i^\star\|^2\le \hat{\rho}^2(\|\bx-\bx^\star\|_{\infty}^b)^2,
	\end{equation}
	which results in \eqref{eq:pseudocontractive} with $c=\hat{\rho}$.
	
	By \eqref{eq:pseudocontractive} and Lemma \ref{lemma:total}, the convergence result holds.% Completes the proof.
	
	%Then, by \eqref{eq:linearATC}, \eqref{eq:linearCAA}, Lemma \ref{lemma:total}, $\{\bx^k\}$ generated by \eqref{eq:asyncop} converges to $0$ under the total asynchrony model in Assumption \ref{asm:totalasynchrony}, which implies the convergence of the two asynchronous DGD. Note that 

	\subsection{Proof of Theorem \ref{thm:partialATC}}\label{append:thmpartialATC}
	
	The proof is similar to the proof of Theorem \ref{thm:partial}.
	
	\subsubsection{Asymptotic convergence}
	
	We prove the asymptotic convergence by showing that all the conditions in Proposition \ref{proposition:copyBertsekas} hold. Because the fixed point set of $\T$ in \eqref{eq:TATC} is the optimal solution set of minimizing the convex and closed function $\mc{L}_{\alpha}$ defined at the beginning of Appendix \ref{append:DGDATCoptimalitygap}, it is closed. The non-emptiness of the fixed point set is shown by Lemma \ref{lemma:DGDATCoptimalitygap}. As a result, condition (a) in Proposition \ref{proposition:copyBertsekas} holds. The proofs of the conditions (c), (d) for the asynchronous DGD-ATC are very similar to their proof for the asynchronous Prox-DGD and are therefore omitted for simplicity.
	
	Below, we prove that condition (b) also holds. Since $\T$ in \eqref{eq:TATC} is continuous, to show condition (b), it suffices to show the non-expansiveness. By using \eqref{eq:stepsizecondATC} and similar derivation of \eqref{eq:opATCconvexityandsmooth}, we have
	\begin{equation*}
	\begin{split}
	&\|x_j-x_j^\star-\alpha(\nabla f_j(x_j)-\nabla f_j(x_j^\star))\|^2 \le\|x_j-x_j^\star\|^2.
	\end{split}
	\end{equation*}
	Substituting the above equation into \eqref{eq:nonexpansiveCAA} yields
	\begin{equation*}
	\begin{split}
	\|\T_i(\bx)-x_i^\star\|^2&\le \sum_{j\in\bar{\mc{N}}_i}w_{ij}\|x_j-x_j^\star\|^2\le (\|\bx-\bx^\star\|_{\infty}^b)^2.
	\end{split}
	\end{equation*}
	This completes the proof of condition (b) for $\T$ in \eqref{eq:TATC}. Then by Proposition \ref{proposition:copyBertsekas}, $\{\bx^k\}$ converges to a fixed point of $\T$.

    \subsubsection{Linear convergence}	
    By \eqref{eq:linearATC}, the pseudo-contractivity condition in Lemma \ref{lemma:partial} holds with $c=\hat{\rho}$. Then by \eqref{eq:nokiincludeallzero}, the linear convergence holds.%  Note that although we do not assume \eqref{eq:initialfey}, the proof of \cite[Theorem 24]{Feyzmahdavian23} still holds, with the convergence rate \eqref{eq:linearfey} becomes
	% \begin{equation*}
	% \|\bx^k-\bx^\star\|_{\infty}^b\le \hat{\rho}^{\lfloor\frac{k}{B+D+1}\rfloor}\|\bx^0-\bx^\star\|_{\infty}^b.
	% \end{equation*}
	% Completes the proof.
    \subsection{Proof of Theorem \ref{thm:adaptATC}}\label{append:thmadaptATC}

    By \eqref{eq:linearATC} and following the proof in Appendix \ref{append:thmadapt}, we obtain the result.

    \bibliographystyle{ieeetran}
    \bibliography{reference}

\begin{thebibliography}{10}
\providecommand{\url}[1]{#1}
\csname url@samestyle\endcsname
\providecommand{\newblock}{\relax}
\providecommand{\bibinfo}[2]{#2}
\providecommand{\BIBentrySTDinterwordspacing}{\spaceskip=0pt\relax}
\providecommand{\BIBentryALTinterwordstretchfactor}{4}
\providecommand{\BIBentryALTinterwordspacing}{\spaceskip=\fontdimen2\font plus
\BIBentryALTinterwordstretchfactor\fontdimen3\font minus
  \fontdimen4\font\relax}
\providecommand{\BIBforeignlanguage}[2]{{%
\expandafter\ifx\csname l@#1\endcsname\relax
\typeout{** WARNING: IEEEtran.bst: No hyphenation pattern has been}%
\typeout{** loaded for the language `#1'. Using the pattern for}%
\typeout{** the default language instead.}%
\else
\language=\csname l@#1\endcsname
\fi
#2}}
\providecommand{\BIBdecl}{\relax}
\BIBdecl

\bibitem{assran2020advances}
M.~Assran, A.~Aytekin, H.~R. Feyzmahdavian, M.~Johansson, and M.~G. Rabbat,
  ``Advances in asynchronous parallel and distributed optimization,''
  \emph{Proceedings of the IEEE}, vol. 108, no.~11, pp. 2013--2031, 2020.

\bibitem{huba2022papaya}
D.~Huba, J.~Nguyen, K.~Malik, R.~Zhu, M.~Rabbat, A.~Yousefpour, C.-J. Wu,
  H.~Zhan, P.~Ustinov, H.~Srinivas \emph{et~al.}, ``Papaya: Practical, private,
  and scalable federated learning,'' \emph{Proceedings of Machine Learning and
  Systems}, vol.~4, pp. 814--832, 2022.

\bibitem{nedic2010convergence}
A.~Nedi{\'c} and A.~Ozdaglar, ``Convergence rate for consensus with delays,''
  \emph{Journal of Global Optimization}, vol.~47, pp. 437--456, 2010.

\bibitem{mishchenko2018delay}
K.~Mishchenko, F.~Iutzeler, J.~Malick, and M.-R. Amini, ``A delay-tolerant
  proximal-gradient algorithm for distributed learning,'' in
  \emph{International Conference on Machine Learning}.\hskip 1em plus 0.5em
  minus 0.4em\relax PMLR, 2018, pp. 3587--3595.

\bibitem{zhang2019asyspa}
J.~Zhang and K.~You, ``{AsySPA}: An exact asynchronous algorithm for convex
  optimization over digraphs,'' \emph{IEEE Transactions on Automatic Control},
  vol.~65, no.~6, pp. 2494--2509, 2019.

\bibitem{doan2017convergence}
T.~T. Doan, C.~L. Beck, and R.~Srikant, ``On the convergence rate of
  distributed gradient methods for finite-sum optimization under communication
  delays,'' \emph{Proceedings of the ACM on Measurement and Analysis of
  Computing Systems}, vol.~1, no.~2, pp. 1--27, 2017.

\bibitem{kungurtsev2023decentralized}
V.~Kungurtsev, M.~Morafah, T.~Javidi, and G.~Scutari, ``Decentralized
  asynchronous non-convex stochastic optimization on directed graphs,''
  \emph{\emph{accepted to} IEEE Transactions on Control of Network Systems},
  2023.

\bibitem{assran2020asynchronous}
M.~S. Assran and M.~G. Rabbat, ``Asynchronous gradient push,'' \emph{IEEE
  Transactions on Automatic Control}, vol.~66, no.~1, pp. 168--183, 2020.

\bibitem{spiridonoff2020robust}
A.~Spiridonoff, A.~Olshevsky, and I.~C. Paschalidis, ``Robust asynchronous
  stochastic gradient-push: Asymptotically optimal and network-independent
  performance for strongly convex functions,'' \emph{Journal of machine
  learning research}, vol.~21, no.~58, pp. 1--47, 2020.

\bibitem{zhang2019fully}
J.~Zhang and K.~You, ``Fully asynchronous distributed optimization with linear
  convergence in directed networks,'' \emph{arXiv preprint arXiv:1901.08215v2},
  2019.

\bibitem{wu2017decentralized}
T.~Wu, K.~Yuan, Q.~Ling, W.~Yin, and A.~H. Sayed, ``Decentralized consensus
  optimization with asynchrony and delays,'' \emph{IEEE Transactions on Signal
  and Information Processing over Networks}, vol.~4, no.~2, pp. 293--307, 2017.

\bibitem{tian2020achieving}
Y.~Tian, Y.~Sun, and G.~Scutari, ``Achieving linear convergence in distributed
  asynchronous multiagent optimization,'' \emph{IEEE Transactions on Automatic
  Control}, vol.~65, no.~12, pp. 5264--5279, 2020.

\bibitem{bianchi2021distributed}
M.~Bianchi, W.~Ananduta, and S.~Grammatico, ``The distributed dual ascent
  algorithm is robust to asynchrony,'' \emph{IEEE Control Systems Letters},
  vol.~6, pp. 650--655, 2022.

\bibitem{cannelli2020asynchronous}
L.~Cannelli, F.~Facchinei, G.~Scutari, and V.~Kungurtsev, ``Asynchronous
  optimization over graphs: Linear convergence under error bound conditions,''
  \emph{IEEE Transactions on Automatic Control}, vol.~66, no.~10, pp.
  4604--4619, 2020.

\bibitem{latafat2022primal}
P.~Latafat and P.~Patrinos, ``Primal-dual algorithms for multi-agent structured
  optimization over message-passing architectures with bounded communication
  delays,'' \emph{Optimization Methods and Software}, vol.~37, no.~6, pp.
  2052--2079, 2022.

\bibitem{su2022convergence}
Y.~Su, Z.~Wang, M.~Cao, M.~Jia, and F.~Liu, ``Convergence analysis of dual
  decomposition algorithm in distributed optimization: Asynchrony and
  inexactness,'' \emph{IEEE Transactions on Automatic Control}, 2022.

\bibitem{ubl2021totally}
M.~Ubl and M.~Hale, ``Totally asynchronous large-scale quadratic programming:
  Regularization, convergence rates, and parameter selection,'' \emph{IEEE
  Transactions on Control of Network Systems}, vol.~8, no.~3, pp. 1465--1476,
  2021.

\bibitem{Wu23ICML}
X.~Wu, C.~Liu, S.~Magn\'{u}sson, and M.~Johansson, ``Delay-agnostic
  asynchronous coordinate update algorithm,'' in \emph{Proceedings of the 40th
  International Conference on Machine Learning}, vol. 202.\hskip 1em plus 0.5em
  minus 0.4em\relax PMLR, 2023, pp. 37\,582--37\,606.

\bibitem{wang2021asynchronous}
Y.~Wang, Q.~Zhao, and X.~Wang, ``An asynchronous gradient descent based method
  for distributed resource allocation with bounded variables,'' \emph{IEEE
  Transactions on Automatic Control}, vol.~67, no.~11, pp. 6106--6111, 2021.

\bibitem{yuan2016convergence}
K.~Yuan, Q.~Ling, and W.~Yin, ``On the convergence of decentralized gradient
  descent,'' \emph{SIAM Journal on Optimization}, vol.~26, no.~3, pp.
  1835--1854, 2016.

\bibitem{dalcin2008mpi}
L.~Dalc{\'\i}n, R.~Paz, M.~Storti, and J.~D’El{\'\i}a, ``{MPI} for python:
  Performance improvements and {MPI}-2 extensions,'' \emph{Journal of Parallel
  and Distributed Computing}, vol.~68, no.~5, pp. 655--662, 2008.

\bibitem{lecun1998gradient}
Y.~LeCun, L.~Bottou, Y.~Bengio, and P.~Haffner, ``Gradient-based learning
  applied to document recognition,'' \emph{Proceedings of the IEEE}, vol.~86,
  no.~11, pp. 2278--2324, 1998.

\bibitem{Sun17}
T.~Sun, R.~Hannah, and W.~Yin, ``Asynchronous coordinate descent under more
  realistic assumption,'' in \emph{Proceedings of the 31st International
  Conference on Neural Information Processing Systems}, 2017, pp. 6183--6191.

\bibitem{zeng2018nonconvex}
J.~Zeng and W.~Yin, ``On nonconvex decentralized gradient descent,'' \emph{IEEE
  Transactions on signal processing}, vol.~66, no.~11, pp. 2834--2848, 2018.

\bibitem{pu2020asymptotic}
S.~Pu, A.~Olshevsky, and I.~C. Paschalidis, ``Asymptotic network independence
  in distributed stochastic optimization for machine learning: Examining
  distributed and centralized stochastic gradient descent,'' \emph{IEEE signal
  processing magazine}, vol.~37, no.~3, pp. 114--122, 2020.

\bibitem{wu23CDC}
X.~Wu, C.~Liu, S.~Magnusson, and M.~Johansson, ``Delay-agnostic asynchronous
  distributed optimization,'' \emph{\emph{accepted to} IEEE Conference on
  Decision and Control (CDC)}, 2023.

\bibitem{giselsson2013accelerated}
P.~Giselsson, M.~D. Doan, T.~Keviczky, B.~De~Schutter, and A.~Rantzer,
  ``Accelerated gradient methods and dual decomposition in distributed model
  predictive control,'' \emph{Automatica}, vol.~49, no.~3, pp. 829--833, 2013.

\bibitem{Vapnik}
V.~Vapnik, ``An overview of statistical learning theory,'' \emph{IEEE
  Transactions on Neural Networks}, vol.~10, no.~5, pp. 988--999, 1999.

\bibitem{nedic2009distributed}
A.~Nedic and A.~Ozdaglar, ``Distributed subgradient methods for multi-agent
  optimization,'' \emph{IEEE Transactions on Automatic Control}, vol.~54,
  no.~1, pp. 48--61, 2009.

\bibitem{nedic17}
A.~Nedi\'{c}, A.~Olshevsky, and W.~Shi, ``Achieving geometric convergence for
  distributed optimization over time-varying graphs,'' \emph{SIAM Journal on
  Optimization}, vol.~27, no.~4, pp. 2597--2633, 2017.

\bibitem{liu2022decentralized}
C.~Liu, Z.~Zhou, J.~Pei, Y.~Zhang, and Y.~Shi, ``Decentralized composite
  optimization in stochastic networks: A dual averaging approach with linear
  convergence,'' \emph{\emph{accepted to} IEEE Transactions on Automatic
  Control}, 2022.

\bibitem{shi2015proximal}
W.~Shi, Q.~Ling, G.~Wu, and W.~Yin, ``A proximal gradient algorithm for
  decentralized composite optimization,'' \emph{IEEE Transactions on Signal
  Processing}, vol.~63, no.~22, pp. 6013--6023, 2015.

\bibitem{shi2015extra}
------, ``{EXTRA}: An exact first-order algorithm for decentralized consensus
  optimization,'' \emph{SIAM Journal on Optimization}, vol.~25, no.~2, pp.
  944--966, 2015.

\bibitem{bauschke2011convex}
H.~H. Bauschke, P.~L. Combettes \emph{et~al.}, \emph{Convex analysis and
  monotone operator theory in Hilbert spaces (second edition)}.\hskip 1em plus
  0.5em minus 0.4em\relax Springer, 2017.

\bibitem{Li23}
Y.~Li, P.~G. Voulgaris, D.~M. Stipanović, and N.~M. Freris, ``Communication
  efficient curvature aided primal-dual algorithms for decentralized
  optimization,'' \emph{IEEE Transactions on Automatic Control}, vol.~68,
  no.~11, pp. 6573--6588, 2023.

\bibitem{choi2023convergence}
W.~Choi and J.~Kim, ``On the convergence analysis of the decentralized
  projected gradient descent,'' \emph{arXiv preprint arXiv:2303.08412}, 2023.

\bibitem{zhang2021penalty}
J.~Zhang, H.~Liu, A.~M.-C. So, and Q.~Ling, ``A penalty alternating direction
  method of multipliers for convex composite optimization over decentralized
  networks,'' \emph{IEEE Transactions on Signal Processing}, vol.~69, pp.
  4282--4295, 2021.

\bibitem{bertsekas2015parallel}
D.~P. Bertsekas and J.~Tsitsiklis, \emph{Parallel and distributed computation:
  numerical methods}.\hskip 1em plus 0.5em minus 0.4em\relax Athena Scientific,
  2015.

\bibitem{pmlr-v80-zhou18b}
Z.~Zhou, P.~Mertikopoulos, N.~Bambos, P.~Glynn, Y.~Ye, L.-J. Li, and
  L.~Fei-Fei, ``Distributed asynchronous optimization with unbounded delays:
  How slow can you go?'' in \emph{Proceedings of the 35th International
  Conference on Machine Learning}, vol.~80.\hskip 1em plus 0.5em minus
  0.4em\relax PMLR, 10--15 Jul 2018, pp. 5970--5979.

\bibitem{Feyzmahdavian23}
H.~R. Feyzmahdavian and M.~Johansson, ``Asynchronous iterations in
  optimization: New sequence results and sharper algorithmic guarantees.''
  \emph{J. Mach. Learn. Res.}, vol.~24, pp. 158--1, 2023.

\bibitem{ATC_first}
F.~S. Cattivelli and A.~H. Sayed, ``Diffusion lms strategies for distributed
  estimation,'' \emph{IEEE transactions on signal processing}, vol.~58, no.~3,
  pp. 1035--1048, 2009.

\bibitem{zhao2015asynchronous}
X.~Zhao and A.~H. Sayed, ``Asynchronous adaptation and learning over
  networks—part i: Modeling and stability analysis,'' \emph{IEEE Transactions
  on Signal Processing}, vol.~63, no.~4, pp. 811--826, 2015.

\bibitem{zhao2015asynchronousII}
------, ``Asynchronous adaptation and learning over networks—part ii:
  Performance analysis,'' \emph{IEEE Transactions on Signal Processing},
  vol.~4, no.~63, pp. 827--842, 2015.

\bibitem{Dua:2019}
\BIBentryALTinterwordspacing
D.~Dua and C.~Graff, ``{UCI} machine learning repository,'' 2017. [Online].
  Available: \url{http://archive.ics.uci.edu/ml}
\BIBentrySTDinterwordspacing

\bibitem{xin2019distributed}
R.~Xin and U.~A. Khan, ``Distributed heavy-ball: A generalization and
  acceleration of first-order methods with gradient tracking,'' \emph{IEEE
  Transactions on Automatic Control}, vol.~65, no.~6, pp. 2627--2633, 2019.

\bibitem{qureshi2021s}
M.~I. Qureshi, R.~Xin, S.~Kar, and U.~A. Khan, ``S-addopt: Decentralized
  stochastic first-order optimization over directed graphs,'' \emph{IEEE
  Control Systems Letters}, vol.~5, no.~3, 2021.

\bibitem{diaconis1991geometric}
P.~Diaconis and D.~Stroock, ``Geometric bounds for eigenvalues of markov
  chains,'' \emph{The annals of applied probability}, pp. 36--61, 1991.

\bibitem{nesterov2003introductory}
Y.~Nesterov, \emph{Introductory lectures on convex optimization: A basic
  course}.\hskip 1em plus 0.5em minus 0.4em\relax Springer Science \& Business
  Media, 2003, vol.~87.

\bibitem{bertsekas1995nonlinear}
D.~P. Bertsekas \emph{et~al.}, ``Nonlinear programming,'' 1995.

\end{thebibliography}

\begin{IEEEbiography}[{\includegraphics[width=1.05in,height=1.25in,clip]{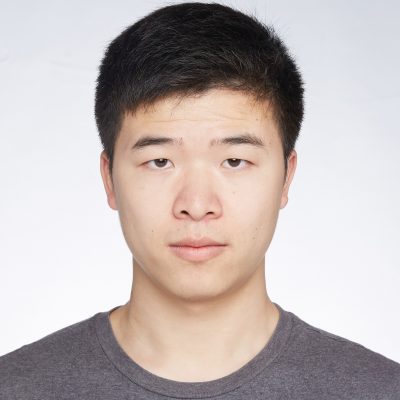}}]{Xuyang Wu} received the B.S. degree in Information and Computing Science from Northwestern Polytechnical University, Xi'an, China in 2015, and the Ph.D. degree from the University of Chinese Academy of Sciences, China in 2020. He is currently a postdoctoral researcher in the Division of Decision and Control Systems at KTH Royal Institute of Technology, Stockholm, Sweden. His research interests include distributed/parallel optimization and machine learning.
\end{IEEEbiography}

\begin{IEEEbiography}[{\includegraphics[width=1.05in,height=1.25in,clip]{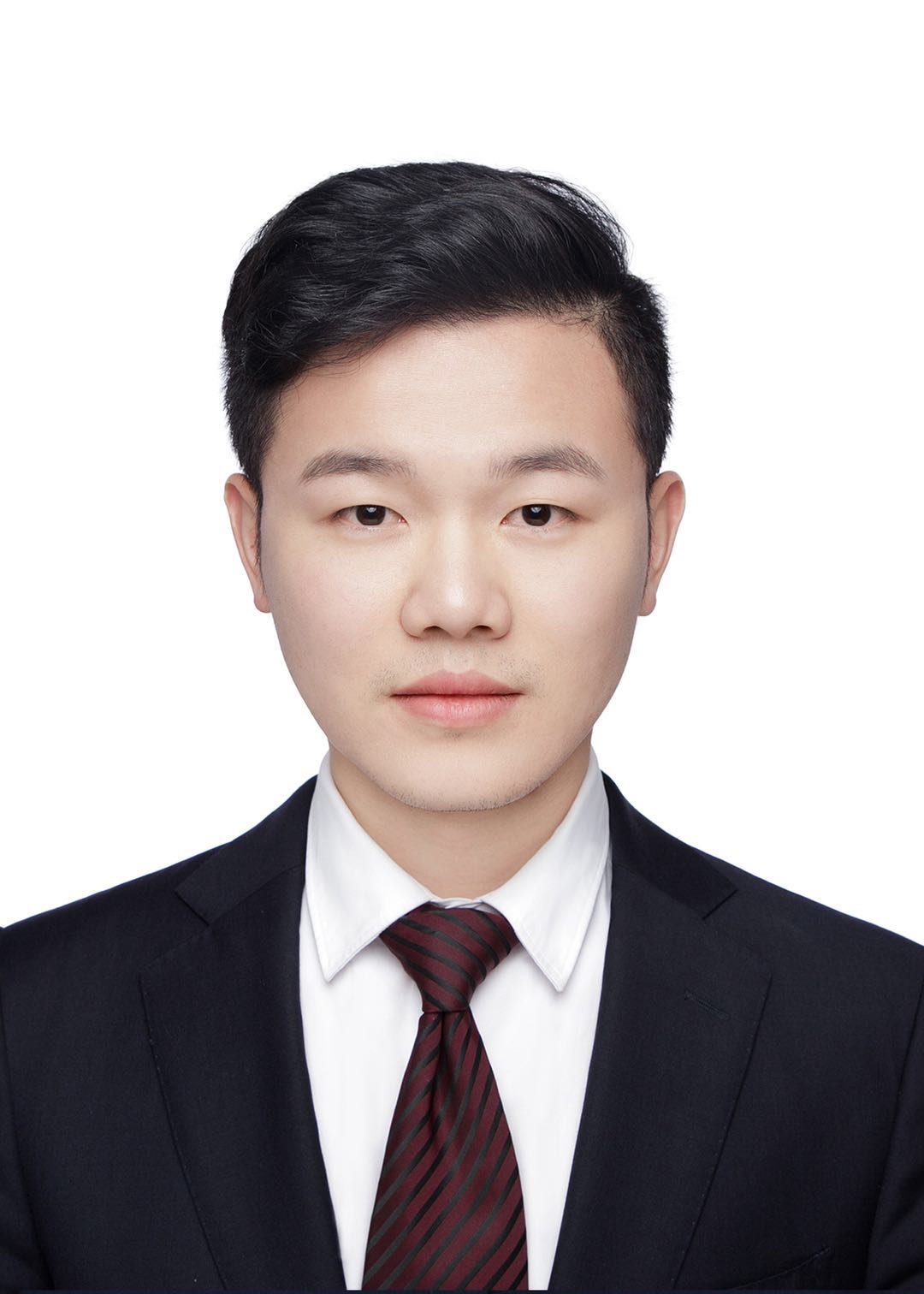}}]{Changxin Liu} received the Ph.D. degree in mechanical engineering from the University of Victoria, Victoria, BC, Canada, in 2021. He is currently a Postdoctoral Researcher with the School of Electrical Engineering and Computer Science, KTH Royal Institute of Technology, Stockholm, Sweden. His research interests revolve around the control, optimization, and learning of networked systems, with a specific focus on their applications in cyber-physical systems and robotics.

Dr. Liu was a recipient of the NSERC Postdoctoral Fellowship in 2023. He currently serves as Associate Editor for Circuits, Systems, and Signal Processing.
\end{IEEEbiography}

\begin{IEEEbiography}[{\includegraphics[width=1.05in,height=1.25in,clip]{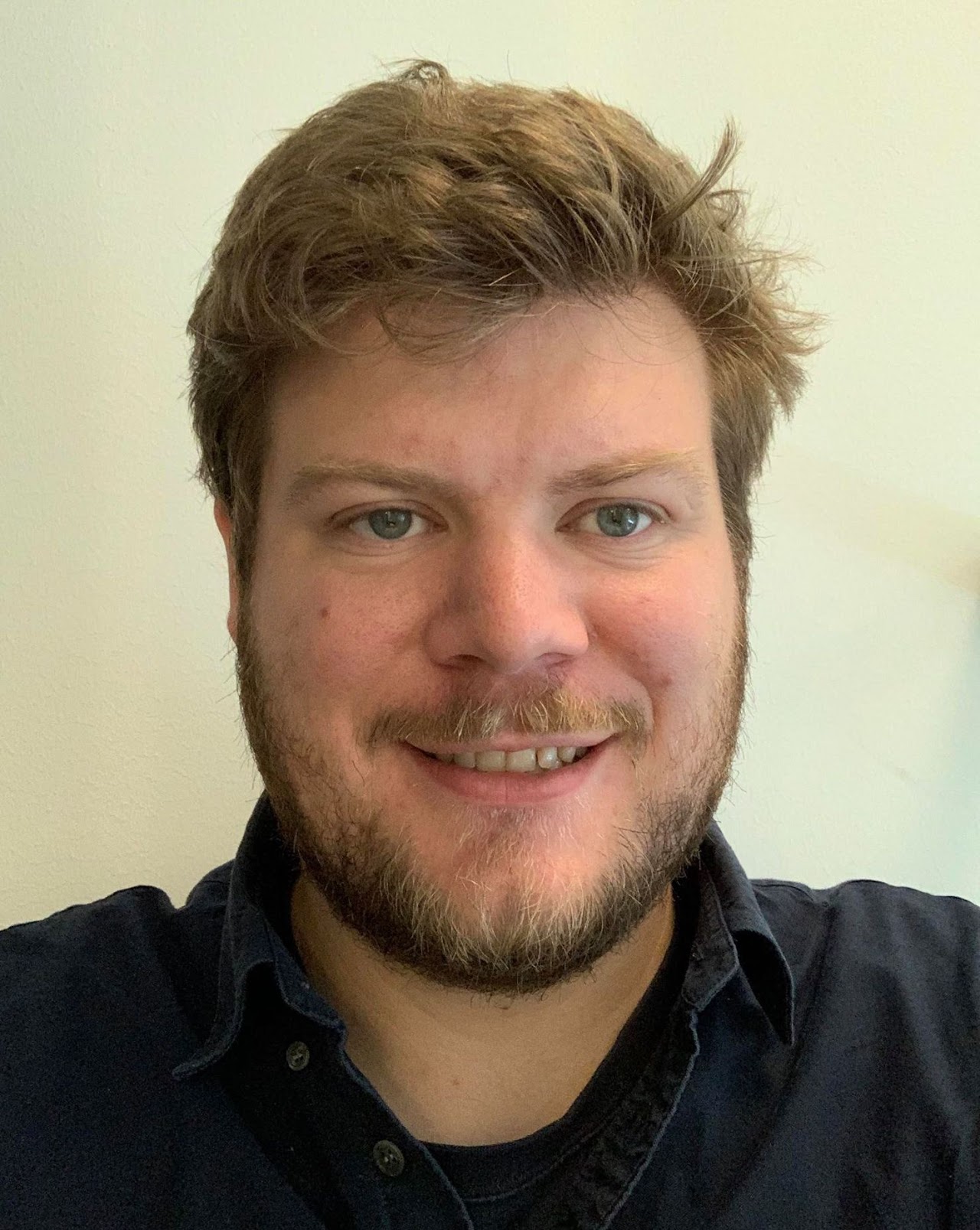}}]{Sindri Magn\'{u}sson}
 is an Associate Professor at the Department of Computer and Systems Science, Stockholm University, Sweden. He received the B.Sc. degree in Mathematics from University of Iceland, Reykjav\'{i}k Iceland, in 2011, the Masters degree in Applied Mathematics (Optimization and Systems Theory) from KTH Royal Institute of Technology, Stockholm Sweden, in 2013, and the PhD in Electrical Engineering from the same institution, in 2017. He was a postdoctoral researcher 2018-2019 at Harvard University, Cambridge, MA and a visiting PhD student at Harvard University for 9 months in 2015 and 2016. His research interests include distributed optimization, machine learning and data-driven decision-making both theory and applications in complex networks. He currently serves as an Associate Editor for IEEE/ACM Transactions on Networking. He has received a number of awards, including the best student paper award (as a supervisor) at the IEEE International Conference on Acoustics, Speech, \& Signal Processing (ICASSP) and the Swedish Research Council (VR) starting grant within natural and engineering sciences.
\end{IEEEbiography}

\begin{IEEEbiography}[{\includegraphics[width=1.05in,height=1.25in,clip]{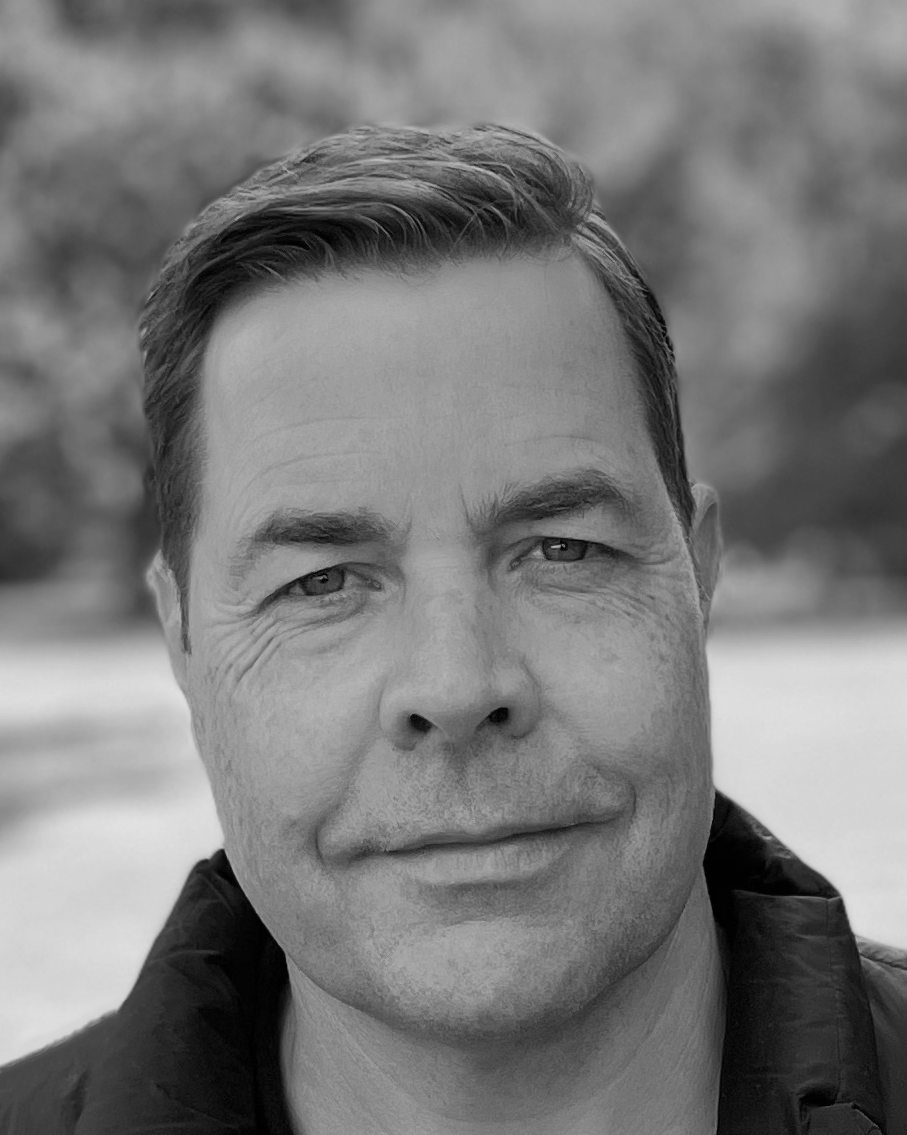}}]{Mikael Johansson} received the M.Sc. and Ph.D. degrees in electrical engineering from
Lund University, Sweden, in 1994 and
1999, respectively. He held postdoctoral positions with
Stanford University and 
the University of California at Berkeley, before joining
the KTH Royal Institute of Technology,
Stockholm, Sweden, in 2002, where he is currently a Full Professor.
He has played a leading role in several national and international
research projects on optimization, control, and communications.
He has coauthored two books and over 200 papers, several of
which are highly cited and have received recognition in terms
of best paper awards. His research revolves around
large-scale and distributed optimization, autonomous decision making, control, and machine learning. Dr. Johansson has served on the Editorial Boards of Automatica
and the IEEE Transactions on Network Systems
and on the program committee for several top conferences in control, communication, and machine learning. 
\end{IEEEbiography}
\end{document}